\documentclass[11pt]{article}
\usepackage{graphicx}
\usepackage{amsmath}
\usepackage{amssymb}
\usepackage{theorem}
\usepackage{enumitem}
\usepackage{fancyhdr}

 \usepackage{hyperref}

\numberwithin{equation}{section}

\newtheorem{thm}{Theorem}[section]
\newtheorem{lemma}[thm]{Lemma}
\newtheorem{prop}[thm]{Proposition}
\newtheorem{cor}[thm]{Corollary}
{\theorembodyfont{\rmfamily}
\newtheorem{defn}[thm]{Definition}
\newtheorem{example}[thm]{Example}

\newtheorem{rmk}[thm]{Remark}
}

\newcommand{\qed}{\hfill \mbox{\raggedright \rule{.07in}{.1in}}}
 
\newenvironment{proof}{\vspace{1ex}\noindent{\bf
Proof}\hspace{0.5em}}{\hfill\qed\vspace{1ex}}
\newenvironment{pfof}[1]{\vspace{1ex}\noindent{\bf Proof of
#1}\hspace{0.5em}}{\hfill\qed\vspace{1ex}}

\newcommand{\R}{{\mathbb R}}

\newcommand{\C}{{\mathbb C}}
\newcommand{\Z}{{\mathbb Z}}

\newcommand{\N}{{\mathbb N}}
\newcommand{\T}{{\mathbb T}}
\renewcommand{\H}{{\mathbb H}}

\newcommand{\overbar}[1]{\mkern 1.5mu\overline{\mkern-1.5mu#1\mkern-1.5mu}\mkern 1.5mu}

\newcommand{\overbarr}[1]{\mkern 3.5mu\overline{\mkern-3.5mu#1\mkern-0.5mu}\mkern 0.5mu}

\newcommand{\barH}{{\overbar{\H}}}

\newcommand{\intphi}{|\varphi|_1}
\newcommand{\bphi}{{\bar\varphi}}
\newcommand{\bmu}{\bar\mu}
\newcommand{\bY}{{\overbar Y}}
\newcommand{\tY}{{\widetilde Y}}
\newcommand{\bF}{\overbarr F}
\newcommand{\sumd}{{\textstyle\sum_d}}

\newcommand{\hJ}{{\widehat J}}
\newcommand{\hR}{{\widehat R}}
\newcommand{\hT}{{\widehat T}}
\newcommand{\hV}{{\widehat V}}
\newcommand{\hw}{{\widehat w}}

\newcommand{\cA}{{\mathcal A}}
\newcommand{\cB}{{\mathcal B}}
\newcommand{\cD}{{\mathcal D}}
\newcommand{\cF}{{\mathcal F}}
\newcommand{\cH}{{\mathcal H}}
\newcommand{\cR}{{\mathcal R}}
\newcommand{\cW}{{\mathcal W}}

\newcommand{\period}{{\mathcal L}}

\newcommand{\infd}{{\SMALL\inf_d}}
\newcommand{\infYj}{{\SMALL\inf_{Y_j}}}
\newcommand{\supYj}{{\SMALL\sup_{Y_j}}}
\newcommand{\infFjd}{{\SMALL\inf_{F^jd}}}

\newcommand{\eps}{\epsilon}
\newcommand{\dist}{\operatorname{dist}}

\newcommand{\diam}{\operatorname{diam}}
\newcommand{\supp}{\operatorname{supp}}
\newcommand{\spec}{\operatorname{spec}}
\renewcommand{\Re}{\operatorname{Re}}

\newcommand{\SMALL}{\textstyle}

\title{Superpolynomial and Polynomial Mixing for 
\\ Semiflows and Flows}

\author{
Ian Melbourne\thanks{
 Mathematics Institute, University of Warwick, Coventry, CV4 7AL, UK}}

\date{30 September 2017. Revised 3 September 2018}

\begin{document}

\maketitle

\begin{abstract}
We give a review of results on superpolynomial decay of correlations, and polynomial decay of correlations for nonuniformly expanding semiflows and nonuniformly hyperbolic flows.   A self-contained proof is given for semiflows.
Results for flows are stated without proof (the proofs are contained in separate joint work with B\'alint and Butterley).
Applications include intermittent solenoidal flows, suspended H\'enon attractors, Lorenz attractors and singular hyperbolic attractors, and various Lorentz gas models including the infinite horizon Lorentz gas.
\end{abstract}

\tableofcontents

\section{Introduction}

Let $(\Lambda,\mu_\Lambda)$ be a probability space.   Given a measure-preserving
flow $T_t:\Lambda\to \Lambda$ and observables $v,w\in L^2(\Lambda)$,
we define the correlation function
$\rho_{v,w}(t)=\int_\Lambda v\;w\circ T_t\,d\mu_\Lambda-\int_\Lambda v\,d\mu_\Lambda\int_\Lambda w\,d\mu_\Lambda$.
The flow is {\em mixing} if 
$\lim_{t\to\infty}\rho_{v,w}(t)=0$ for all $v,w\in L^2(\Lambda)$.

Of interest is the rate of decay of correlations, or rate of mixing, namely the rate at
which $\rho_{v,w}$ converges to zero.
For nontrivial mixing flows, the decay rate is
arbitrarily slow for $L^2$ observables.
Hence the aim is to establish decay rates under regularity hypotheses on the
flow $T_t$, the measure $\mu_\Lambda$, and the observables $v,w$.

Consider a mixing uniformly hyperbolic
(Axiom~A) basic set $\Lambda\subset M$ for a smooth flow $T_t:M\to M$ with
invariant probability measure 
$\mu_\Lambda$ that is an equilibrium state for a H\"older potential~\cite{BowenRuelle75}.   
Bowen \& Ruelle~\cite{BowenRuelle75} asked
whether $\Lambda$ is exponentially mixing
($\rho_{v,w}(t)=O(e^{-ct})$ for some $c>0$) for sufficiently
regular $v$, $w$.   (In the discrete time case, it is well-known that
mixing Axiom~A diffeomorphisms are exponentially mixing.)
However, Pollicott~\cite{Pollicott84} and Ruelle~\cite{Ruelle83} showed that mixing Axiom~A flows need 
not mix exponentially, and Pollicott~\cite{Pollicott85} 
showed that the decay rates could be arbitrarily slow.  

Building upon results of Chernov~\cite{Chernov98}, 
Dolgopyat~\cite{Dolgopyat98a} showed that geodesic flows on
compact surfaces of negative curvature are exponentially mixing.
Liverani~\cite{Liverani04} extended this
result to arbitrary dimensional geodesic flows in negative curvature and 
more generally to contact Anosov flows.   Subsequent developments include~\cite{AraujoM16,AGY06,BaladiDemersLiverani18,BaladiVallee05,BMMW17,Pollicott99,Stoyanov01,Tsujii10,Tsujiisub}.
However, exponential mixing remains poorly understood in general.
Two major open questions, even for Anosov flows, are
\begin{itemize}

\parskip = -4pt
\item 
Is exponential mixing typical in any reasonable sense?
\item Does mixing imply exponential mixing rate?
\end{itemize}
 Indeed, it is only recently that the first robust examples of exponentially mixing Axiom~A flows~\cite{ABV16} and Anosov flows~\cite{ButterleyWarsub} have been obtained.  

Dolgopyat~\cite{Dolgopyat98b} introduced the weaker notion of {\em rapid mixing}
(superpolynomial decay of correlations) where 
$\rho_{v,w}(t)=O(t^{-q})$ for sufficiently regular observables for any fixed $q\ge1$, 
and showed that rapid mixing is `prevalent' for Axiom~A flows:
it suffices that the flow contains two periodic solutions
with periods whose ratio is Diophantine.
In addition, Dolgopyat~\cite{Dolgopyat98a} showed that for Anosov flows,
joint nonintegrability of the stable and unstable foliations (an open and dense
condition by Brin~\cite{Brin75,Brin75b}) implies rapid mixing.   
Field~{\em et al.}~\cite{FMT07} introduced the notion of 
{\em good asymptotics} and used this to
prove that amongst $C^r$ Axiom~A flows, $r\ge2$, an
open and dense set of flows is rapid mixing.

The topic of this review article is the extension of the rapid mixing method to
the nonuniformly hyperbolic setting.
In~\cite{M07,M09}, we obtained results on rapid mixing and polynomial mixing for nonuniformly hyperbolic semiflows and flows, combining the rapid mixing method of Dolgopyat~\cite{Dolgopyat98b} with advances by Young~\cite{Young98,Young99} in the discrete time setting.
The purpose of this review article is twofold.
In Part I, we state the results in a way that incorporates applications that have arisen since~\cite{M07,M09}.
(This should simplify the exposition in future applications avoiding complicated referencing of the type in for example~\cite[Section~4]{AMV15}).  
In Part~II, we provide complete self-contained proofs for semiflows.  
The proofs for flows are given in a separate joint paper with
P\'eter B\'alint \& Oliver Butterley~\cite{BBMsub}.

In Subsections~\ref{sec:introrapid} and~\ref{sec:introslow} below, we summarize the applications discussed in~\cite{M07,M09} as well as subsequent applications to Lorentz gases with cusps and in flowers, higher-dimensional Lorentz gases, Lorenz attractors, and intermittent solenoidal flows.

The treatment here (and in~\cite{BBMsub}) differs in two main respects from~\cite{M07,M09}.  The differences apply equally to flows and semiflows so we just mention flows here.  
First, flows are viewed as suspensions over a uniformly hyperbolic map with a possibly unbounded roof function (rather than as suspensions over a uniformly or nonuniformly hyperbolic map with a bounded roof function).  
This simplifies the analysis since it suffices to consider twisted transfer operators with one complex parameter, whereas two complex parameters were required in~\cite{M07}.   A consequence of this is that
the number of periodic orbits in the Diophantine condition in Section~\ref{sec:Dio} reduces from~$4$ to~$3$.
(On the other hand, the methods in~\cite{M07,M09} can be applied also to toral extensions of nonuniformly expanding maps, see~\cite{BaladiHachemi08,MTtoralsub}, and the results here cannot.)
Second,~\cite{Dolgopyat98b,M07,M09} use analyticity properties of the Laplace transform of $\rho_{v,w}$ in the complex plane, whereas we use smoothness on the imaginary axis incorporating ideas developed more recently with Dalia Terhesiu, especially~\cite{MT17}.  One advantage is that Theorem~\ref{thm:rapid} on rapid mixing works for roof functions that lie in $L^p$ for all $p\ge1$, whereas~\cite{M07} requires an exponential tails condition.  Another advantage is that the estimate in Lemma~\ref{lem:trunc} is much cleaner than the corresponding estimate in~\cite{M09}.  (Also, the approach here leads to an estimate in Lemma~\ref{lem:largeb} that is much better than might have been expected.)

In addition, we resolve two issues that were overlooked in~\cite{M09}: 
\begin{itemize}

\parskip = -4pt
\item[(i)] The Dolgopyat argument focuses on controlling high frequencies in the Laplace transform of the correlation function $\rho_{v,w}$ but it is also necessary to deal with the singularity at $0$.  This requires extra arguments in the slowly mixing cases.
\item[(ii)] The hypotheses in~\cite{M09} for passing from semiflows to flows are too strong for the intended applications where there is often slow contraction along stable leaves.  
\end{itemize}

Issue~(i) is covered in the proofs for semiflows in Part~II of this article.
Regarding point~(ii), similar problems have arisen even for discrete time dynamical systems.
Various polynomial mixing results in~\cite{ChernovZhang05,ChernovZhang08,Markarian04} described in Subsection~\ref{sec:introslow}
rely on a result of~\cite{Young99} which is formulated only for noninvertible dynamical systems.  
The extra argument required for passing to 
invertible systems is due to Gou\"ezel~\cite{GouezelPC} based on work of~\cite{ChazottesGouezel12}, and can be found in~\cite[Appendix~B]{MT14} and~\cite[Theorem~2.10]{KKMapp}.
Resolving the corresponding problem for flows turns out to be more subtle (even formulating the correct hypotheses is difficult), and is completed in 
B\'alint {\em et al.}~\cite{BBMsub}.   

In the remainder of the introduction, we give an overview of rapid and polynomial mixing, as well as consequent statistical properties, for various classes of flows.

\subsection{Rapid mixing for nonuniformly hyperbolic flows}
\label{sec:introrapid}

Young~\cite{Young98} established exponential mixing
for a large class of nonuniformly hyperbolic maps 
modelled by Young towers with exponential tails,
including dispersing billiards and H\'enon-like maps~\cite{BenedicksYoung00}.
In~\cite{M07}, we extended the ideas of
Dolgopyat~\cite{Dolgopyat98b} to the corresponding class of
nonuniformly hyperbolic flows where there is a suitable Poincar\'e map 
modelled by a Young tower with exponential tails.
Roughly speaking, one of the main results described in this review article is that
\begin{quote}
{\em A `prevalent' set of nonuniformly hyperbolic flows with exponential tails are rapid mixing.}
\end{quote}
Rapid mixing is established for sufficiently regular observables,
and prevalence is understood in terms of a Diophantine condition on the periodic data and also in terms of the good asymptotics condition~\cite{FMT07} which is open and dense.

\vspace{-2ex}
\paragraph{Planar Lorentz gases}

A rich class of examples of three-dimensional nonuniform hyperbolic flows is provided by planar Lorentz gas models~\cite{Sinai70, ChernovMarkarian}.    
See~\cite{ChernovYoung00} for a survey of results about these models.
From now on, all Lorentz gases are planar unless stated otherwise.

In particular, the periodic Lorentz gas is a billiard flow on 
$\T^2\setminus\Omega$
where $\Omega$ is a disjoint union of convex regions with $C^3$ boundaries.
(The phase-space of the flow is three-dimensional; planar position and 
direction.)   The flow has a natural global cross-section
corresponding to collisions
and the Poincar\'e map is called the {\em billiard map}.
The Lorentz flow 
satisfies the {\em finite horizon} condition if the time between collisions is uniformly bounded.   
For the billiard map associated to the finite horizon Lorentz gas, Bunimovich, Sina{\u\i} \& Chernov~\cite{BunimovichSinaiChernov91} proved
stretched exponential mixing rates, 
and exponential mixing was established by Young~\cite{Young98}.
Chernov~\cite{Chernov99} extended Young's method to prove exponential mixing for billiard maps corresponding to (i) periodic Lorentz gases with infinite horizon, (ii) periodic Lorentz gases with external forcing, and (iii) Lorentz gases in bounded domains with sides that are convex inwards and corner points with positive angles.
(For the third case, billiards with corners, a technical condition in~\cite{Chernov99} was removed in~\cite{SimoiToth14}.)  For the corresponding flows, we have:

\begin{quote}
{\em Finite horizon planar periodic Lorentz flows, and planar Lorentz flows with
cusps or corner points, are rapid mixing.
A prevalent set of externally forced finite horizon planar periodic Lorentz flows
 are rapid mixing.}  
\end{quote}

Except for the flows with cusps, this is a consequence of~\cite{M07}. 
(When there is no external forcing, the contact structure ensures that
the prevalence assumption is unnecessary~\cite{M09}, see
Remark~\ref{rmk:contact}.)

The case of cusps is treated in~\cite{BM08}.  Here, the billiard map mixes only at the rate $O(n^{-1})$ so~\cite{M07} does not apply directly.  However, collisions within a cusp occur in quick succession so it is reasonable to expect much quicker, possibly even exponential, mixing in continuous time.  As shown in~\cite{BM08}, the billiard map on the `natural' cross-section can be replaced by a Poincar\'e map (with cross-section bounded away from the cusps) that is modelled by a Young tower with exponential tails in such a way that~\cite{M07} applies.
The same idea applies to a class of billiards called
Bunimovich flowers~\cite{Bunimovich73} (see also~\cite[Section~3.1]{BM08} for a  precise description).  It was shown in~\cite{ChernovZhang05} that the billiard map mixes at roughly the rate $O(n^{-2})$.  By~\cite{BM08}, 
\begin{quote}
{\em For Bunimovich flowers, the corresponding Lorentz flow is rapid mixing.}
\end{quote}

For the finite horizon planar periodic Lorentz flow,
Chernov~\cite{Chernov07} proved stretched exponential mixing, and
exponential mixing was finally achieved in the recent paper of
Baladi, Demers \& Liverani~\cite{BaladiDemersLiverani18}.  The methods of~\cite{BaladiDemersLiverani18,Chernov07} rely crucially on the contact structure.  It is reasonable to expect that these methods also apply to the Lorentz flows with corners and cusps, as well as the Bunimovich flower examples, though the technical details have not been verified at the time of writing.  
However, it seems unlikely that the result on prevalent rapid mixing for externally forced Lorentz flows~\cite{M07} will be improved upon in the near future --- the lack of contact structure and the irregularity of the foliations means that the current technology for proving (stretched) exponential mixing is insufficient.

Once again, we emphasize that the rapid mixing results above are proved for observables which
are smooth along the flow direction.  Unfortunately this excludes certain important physical observables  such as position and velocity.
In contrast, the results in~\cite{BaladiDemersLiverani18,Chernov07} cover all H\"older (and dynamically H\"older) observables, including position and velocity, but the methods apply less generally.

For higher-dimensional periodic billiards, there is a technical and unresolved problem concerning ``growth of complexity''~\cite{BCST03}.  It is conjectured that growth of complexity is typically bounded, but there are no examples where the growth is known to be subexponential  and there are counterexamples where the growth of complexity is exponential~\cite{BalintToth12}.  B\'alint \& T\'oth~\cite{BalintToth08} proved that this is the only obstruction, obtaining the following conditional result:
 For higher-dimensional Lorentz gases, if the growth of complexity is subexponential, then the billiard map is modelled by a Young tower with exponential tails and in particular is exponentially mixing.  Hence by~\cite{M07},

\begin{quote}
{\em Higher-dimensional periodic Lorentz gas flows with subexponential growth of complexity are rapid mixing.}
\end{quote}

\vspace{-3ex}
\paragraph{Flows near homoclinic tangencies}
Benedicks \& Carleson~\cite{BenedicksCarleson91} studied the H\'enon map
$T_{a,b}(x,y)=(1-ax^2+y,bx)$ and proved the existence of a strange attractor 
for a positive measure set of parameters $a,b$.  The attractor admits an
SRB measure~\cite{BenedicksYoung93} and is exponentially mixing
by Benedicks \& Young~\cite{BenedicksYoung00}.
Mora \& Viana~\cite{MoraViana93} showed that H\'enon-like attractors arise
for positive measure sets of parameters
in the unfoldings of homoclinic tangencies for surface diffeomorphisms
and this was extended to higher dimensions by~\cite{Viana93,DiazRochaViana96}.
By~\cite{M07},
\begin{quote}
{\em
Positive measure sets of flows near a homoclinic tangency are rapid mixing.}
\end{quote}
Again, (stretched) exponential mixing is beyond the current technology.

\vspace{-2ex}
\paragraph{Lorenz attractors and singular hyperbolic attractors}

The classical Lorenz attractor is exponentially mixing~\cite{AraujoM16},
but the proof relies on the smoothness of the stable foliation for the flow~\cite{AraujoM17,AMV15}.
General Lorenz attractors (and singular hyperbolic attractors) have a H\"older stable foliation~\cite{AraujoM17} but this foliation need not be smooth.
Building upon~\cite{APPV09}, it is shown in~\cite{AraujoMsub} that 
an open and dense set of Lorenz attractors (and codimension two singular hyperbolic attractors) are rapid mixing even when the stable foliation is not smooth.

\subsection{Statistical limit laws for time-one maps} 
\label{sec:stats}

In situations where we obtain rapid mixing for a nonuniformly hyperbolic flow, certain statistical properties such as the central limit theorem (CLT) and almost sure invariance principle (ASIP) hold for the time-one map of the flow.  (Such properties for the flow itself are much simpler since they are inherited from the statistical properties of the Poincar\'e map and no mixing properties are required~\cite{Ratner73,DenkerPhilipp84,MT04}.)

In particular, given any of the rapidly mixing flows described in 
Subsection~\ref{sec:introrapid}
and a mean zero observable $v:\Lambda\to\R$ sufficiently smooth in the flow direction, setting
$v_n=\sum_{j=0}^{n-1} v\circ T_j$, the following hold by~\cite{AMV15}: 
\begin{itemize}
\item[(a)]
(CLT)
$n^{-1/2}v_n\to_d N(0,\sigma^2)$ where
$\sigma^2=\lim_{n\to\infty}n^{-1}\int_\Lambda v_n^2\,d\mu_\Lambda
=\sum_{n=-\infty}^\infty \int_\Lambda v\,v\circ T_n\,d\mu_\Lambda$.
\item[(b)] (Nondegeneracy) If $\sigma^2=0$, then for every periodic orbit $q$ there exists
$t_q>0$ such that $\int_0^{t_q} v(T_tq)\,dt=0$.  (The value of $t_q$ is independent of $v$ but in general is not directly related to the period of $q$.  A formula for $t_q$ is given in Remark~\ref{rmk:tq}.)
\item[(c)] (ASIP) Passing to an enriched probability space, there exists a sequence $X_0,X_1,\ldots$ of iid normal random variables with mean $0$ and variance $\sigma^2$ such that 
$v_n=\sum_{j=0}^{n-1}X_j+O(n^{1/4}(\log n)^{1/2}(\log\log n)^{1/4})$ a.e.
\end{itemize}
The proof uses~\cite{CunyMerlevede15} to obtain the ASIP for time-one maps of rapidly mixing semiflows and then passes as in~\cite{MT02} from the semiflow to the flow.  See also~\cite{MT02} for the CLT.  
The ASIP has various consequences including the CLT and law of the iterated logarithm as well as their functional versions.  For a more complete list of consequences see~\cite{PhilippStout75}.

\subsection{Polynomial mixing for nonuniformly hyperbolic flows}
\label{sec:introslow}

Young~\cite{Young99} established polynomial mixing rates for a large class of nonuniformly expanding maps including intermittent maps of Pomeau-Manneville type~\cite{PomeauManneville80}.  
We consider the corresponding class of
nonuniformly hyperbolic flows with Poincar\'e map 
modelled by a polynomially mixing Young tower.
Roughly speaking, the second main result in this review article is that for $q>0$,
\begin{quote}
{\em Amongst nonuniformly hyperbolic flows for which the Poincar\'e map has mixing rate $O(n^{-q})$, a prevalent set have mixing rate $O(t^{-q})$.}
\end{quote}
Here, prevalent has the same meaning as in Subsection~\ref{sec:introrapid}.
This result applies in particular to intermittent solenoidal flows (Example~\ref{ex:LSVflow}).

\begin{rmk}  The statistical properties described in Subsection~\ref{sec:stats} for rapid mixing flows hold also for flows that decay at rate $O(t^{-q})$ for $q$ large enough.  The CLT requires only that $q>1$.  For the ASIP with error term
$O(n^{1/4}(\log n)^{1/2}(\log\log n)^{1/4})$,
it suffices that $q>2\sqrt2$ (see~\cite[Theorem~5.2]{AMV15}).
\end{rmk}

Again, planar Lorentz gases provide a rich source of examples.
For the billiard maps, Markarian~\cite{Markarian04} combined the work of~\cite{Chernov99, Young99} to prove the mixing rate $O(n^{-1})$ for Bunimovich stadia~\cite{Bunimovich79}, and the method was extended by Chernov \& Zhang~\cite{ChernovZhang05,ChernovZhang08} to obtain the rate $O(n^{-1})$ for semidispersing billiards (rectangular tables with at least one convex obstacle).   Our results apply to the corresponding flows (again the contact structure ensures that there is no prevalence
restriction).

It has been conjectured (especially for the infinite horizon Lorentz gas flow in~\cite{FriedmanMartin88,MatsuokaMartin97}) that the decay rate for the flow is $O(t^{-1})$.
(An elementary argument in~\cite{BalintGouezel06} shows that this rate is optimal; see~\cite[Proposition~9.14]{BBMsub}.)
The conjectured decay rate $O(t^{-1})$ is proved in~\cite{BBMsub}:
\begin{quote}
{\em 
Lorentz flows in Bunimovich stadia, semidispersing Lorentz flows, and infinite horizon planar periodic Lorentz gas flows have mixing rate $O(t^{-1})$.
}
\end{quote}

\paragraph{Notation}
We use the ``big $O$'' and $\ll$ notation interchangeably, writing $a_n=O(b_n)$ or $a_n\ll b_n$ if there is a constant $C>0$ such that
$a_n\le Cb_n$ for all $n\ge1$.  As usual, $a_n\sim b_n$ means that $\lim_{n\to\infty}a_n/b_n=1$.   There are various ``universal'' constants $C_1,\dots,C_5\ge1$ depending only on the flow that do not change throughout the article.

\part{Statement of results}

In this part of the review article, we describe in detail the main results on rapid and polynomial mixing for nonuniformly expanding semiflows and nonuniformly hyperbolic flows.
In Section~\ref{sec:susp}, we recall the notion of a suspension (semi)flow,
as well as associated notation.
Sections~\ref{sec:NUE} and~\ref{sec:NUH} contain the main results for
semiflows and flows respectively.
These rely on a technical condition, absence of approximate eigenfunctions (Definition~\ref{def:approx}).
In Section~\ref{sec:approx}, we discuss criteria ensuring that this condition holds.

\section{Preliminaries on suspensions}
\label{sec:susp}

It is standard to model flows and semiflows as suspensions.  Here, we review the basic definitions and notation that will be used throughout the review article.
For brevity we speak of semiflows, even when some of them are flows.

\vspace{-2ex}
\paragraph{Suspension semiflows in their own right}
Let $(Y,\mu)$ be a probability space and let $F:Y\to Y$ be a measure-preserving transformation.  
Let $\varphi:Y\to\R^+$ be an integrable roof function.
Define the suspension 
$Y^\varphi=\{(y,u)\in Y\times[0,\infty): u\in[0,\varphi(y)]\}/\sim$ where
$(y,\varphi(y))\sim(Fy,0)$.   The suspension 
semiflow $F_t:Y^\varphi\to Y^\varphi$ is given by $F_t(y,u)=(y,u+t)$ 
computed modulo identifications.
We obtain an $F_t$-invariant probability measure on
$Y^\varphi$ given by $\mu^\varphi=\mu\times{\rm Lebesgue}/\int_Y\varphi\,d\mu$.

\vspace{-2ex}
\paragraph{Suspension semiflows as models for ambient semiflows}
Let $T_t:\Lambda\to \Lambda$ be a semiflow defined on an ambient measurable space $\Lambda$.
(We assume that $(x,t)\mapsto T_tx$ is measurable from $\Lambda\times[0,\infty)\to \Lambda$.)
Let $Y\subset \Lambda$ be a measurable subset with probability measure $\mu$.
We suppose that $F:Y\to Y$ is a measure-preserving transformation and $\varphi:Y\to\R^+$ is 
an integrable function such that 
$T_{\varphi(y)}y=Fy$ for all $y\in Y$.
Then we can form the suspension semiflow $F_t:Y^\varphi\to Y^\varphi$ with 
invariant probability measure $\mu^\varphi$.

Define $\pi:Y^\varphi\to \Lambda$ by $\pi(y,u)=T_uy$.  This is well-defined since
$\pi(y,\varphi(y))=T_{\varphi(y)}y=Fy=\pi(Fy,0)$, and defines a measurable semiconjugacy from $F_t$ to $T_t$.  Hence the measure $\mu_\Lambda=\pi_*\mu^\varphi$ is a $T_t$-invariant probability measure on $\Lambda$.

Given observables $v,\,w\in L^2(\Lambda)$, define $\tilde v=v\circ\pi$, 
$\tilde w=w\circ\pi$.  Then $\tilde v,\,\tilde w\in L^2(Y^\varphi)$ and
\begin{align} \label{eq:M}
\int_\Lambda v\;w\circ T_t\,d\mu_\Lambda-\int_\Lambda v\,d\mu_\Lambda\int_\Lambda w\,d\mu_\Lambda
=\int_{Y^\varphi} \tilde v\;\tilde w\circ F_t\,d\mu^\varphi-\int_{Y^\varphi} \tilde v\,d\mu^\varphi\int_{Y^\varphi} \tilde w\,d\mu^\varphi.
\end{align}
Hence rates of mixing for the ambient semiflow $T_t$ and observables $v,\,w$ on $(\Lambda,\mu_\Lambda)$ reduces to
rates of mixing for the suspension semiflow $F_t$ and observables $\tilde v,\,\tilde w$ on $(Y^\varphi,\mu^\varphi)$.

\begin{rmk} For the systems considered in this article, the measure $\mu$ on $Y$ is ergodic by construction, and the form of $F$ ensures ergodicity of  $\mu^\varphi$ and hence $\mu_\Lambda$.
\end{rmk}

\begin{rmk}  \label{rmk:inf}
Throughout this review article we suppose that $\inf\varphi>0$, though
some of the results, particularly Theorems~\ref{thm:slow} and~\ref{thm:rapid} below, do not require this.
In any case, this condition is not very restrictive --- if $\inf\varphi=0$, then this can be circumvented either by shrinking the cross-section $Y$ or by waiting for a later return to $Y$ (see for example~\cite{BM08}).
\end{rmk}

\section{Mixing rates for nonuniformly expanding semiflows}
\label{sec:NUE}

In this section, we review results on rapid mixing and polynomial mixing for
nonuniformly expanding semiflows.  
In Subsection~\ref{sec:GM}, we define a class of Gibbs-Markov semiflows and state the main results for such semiflows.  
In Subsection~\ref{sec:ambient}, we consider nonuniformly expanding H\"older semiflows on an ambient metric space $M$, with 
H\"older observables, and show how these reduce to the situation in Subsection~\ref{sec:GM}.  As an application we consider intermittent semiflows. 
In Subsection~\ref{sec:dyn}, we generalise to 
dynamically H\"older semiflows and observables.

\subsection{Gibbs-Markov semiflows}
\label{sec:GM}

In this subsection, we consider a class of Gibbs-Markov semiflows built as suspensions over Gibbs-Markov maps.
Standard references for background material on Gibbs-Markov maps
are~\cite[Chapter~4]{Aaronson} and~\cite{AaronsonDenker01}.

Suppose that $(Y,\mu)$ is a probability space with
an at most countable measurable partition $\{Y_j,\,j\ge1\}$
and let $F:Y\to Y$ be a measure-preserving transformation.
For $\theta\in(0,1)$, define $d_\theta(y,y')=\theta^{s(y,y')}$ where the {\em separation time} $s(y,y')$ is the least integer $n\ge0$ such that $F^ny$ and $F^ny'$ lie in distinct partition elements in $\{Y_j\}$.
It is assumed that the partition $\{Y_j\}$ separates trajectories, so $s(y,y')=\infty$ if and only if $y=y'$.  Then $d_\theta$ is a metric, called a {\em symbolic metric}.

A function $v:Y\to\R$ is {\em $d_\theta$-Lipschitz} if
$|v|_\theta=\sup_{y\neq y'}|v(y)-v(y')|/d_\theta(y,y')$ is finite.
Let $\cF_\theta(Y)$ be the Banach space of Lipschitz functions with norm $\|v\|_\theta=|v|_\infty+|v|_\theta$.

More generally (and with a slight abuse of notation), we say that a function
$v:Y\to\R$ is {\em piecewise $d_\theta$-Lipschitz} if 
$|1_{Y_j}v|_\theta=\sup_{y,y'\in Y_j,\,y\neq y'}|v(y)-v(y')|/d_\theta(y,y')$
is finite for all $j$.  If in addition, $\sup_j|1_{Y_j}v|_\theta<\infty$ then we say that $v$ is {\em uniformly piecewise $d_\theta$-Lipschitz}.
Note that such a function $v$ is bounded on partition elements but need not be bounded on $Y$.

\begin{defn} \label{def:GM}
A measure-preserving transformation map $F:Y\to Y$ is called a {\em (full branch) Gibbs-Markov map}
if 
\begin{itemize}

\parskip=-2pt
\item $F|_{Y_j}:Y_j\to Y$ is a measurable bijection for each $j\ge1$, and
\item The potential function $p=\log\frac{d\mu}{d\mu\circ F}:Y\to\R$ is uniformly piecewise $d_\theta$-Lipschitz for some $\theta\in(0,1)$.
\end{itemize}
\end{defn}

\begin{defn} \label{def:inf}
A suspension semiflow $F_t:Y^\varphi\to Y^\varphi$ is called a {\em Gibbs-Markov semiflow} if 
there exist constants $C_1\ge1$, $\theta\in(0,1)$ such that
$F:Y\to Y$ is a Gibbs-Markov map, $\varphi:Y\to\R^+$ is an integrable roof function with $\inf\varphi>0$, and
\begin{align} \label{eq:inf}
|1_{Y_j}\varphi|_{\theta}\le C_1\infYj\varphi\quad\text{for all $j\ge1$}.
\end{align}
\end{defn}
It follows that $\supYj\varphi\le 2C_1\infYj\varphi$ for all $j\ge1$.
As mentioned in Remark~\ref{rmk:inf} the assumption $\inf\varphi>0$ is not really needed, though condition~\eqref{eq:inf} would need to be changed to
$|1_{Y_j}\varphi|_{\theta}\le C_1(\infYj\varphi+1)$
to avoid being too restrictive
(similarly, in the definition of $\cF_\theta(Y^\varphi)$ below),
and this would result in more complicated formulas throughout.

\subsubsection{Approximate eigenfunctions}

For $b\in\R$, we define the operators
\[
M_b:L^\infty(Y)\to L^\infty(Y), \qquad M_bv=e^{ib\varphi}v\circ F.
\]
Evidently, $L^\infty(Y)$ here denotes complex-valued essentially bounded measurable functions on~$Y$.  As is standard, we often pass to complexifications of various Banach spaces without comment, for example when discussing spectral properties.  In particular, the functions in $F_\theta(Y)$ in Definition~\ref{def:approx} below are complex-valued.

\begin{defn} \label{def:Z0}  A subset $Z_0\subset Y$ is a {\em finite subsystem} of $Y$
if $Z_0=\bigcap_{n\ge0} F^{-n}Z$ where $Z$ is the union of finitely many 
elements from the partition $\{Y_j\}$.
(Note that $F|_{Z_0}:Z_0\to Z_0$ is a full one-sided
shift on finitely many symbols.)
\end{defn}

\begin{defn} \label{def:approx}
We say that $M_b$ has {\em approximate eigenfunctions} on a subset 
$Z_0\subset Y$ 
if for any $\alpha_0>0$, there exist constants $\alpha,\xi>\alpha_0$ and $C>0$,
and sequences $|b_k|\to\infty$, 
$\psi_k\in [0,2\pi)$, $u_k\in \cF_\theta(Y)$ with $|u_k|\equiv1$
and $|u_k|_\theta\le C|b_k|$, such that 
setting $n_k=[\xi\ln |b_k|]$,
\begin{align}
\label{eq:approx}
|(M_{b_k}^{n_k}u_k)(y)-e^{i\psi_k}u_k(y)|\le C|b_k|^{-\alpha}
\quad\text{for all $y\in Z_0$, $k\ge1$.}
\end{align}
\end{defn}

\begin{rmk} \label{rmk:approx}
For brevity, the statement ``Assume absence of approximate eigenfunctions''
is the assumption that there exists at least one finite subsystem $Z_0$ such that $M_b$ does not have approximate eigenfunctions on $Z_0$.
\end{rmk}

\subsubsection{Observables on $Y^\varphi$}
\label{sec:obs}

Given $v:Y^\varphi\to\R$ and $\theta\in(0,1)$, we define
\[
|v|_\theta=\sup_{(y,u),(y',u)\in Y^\varphi,\,y\neq y'}\frac{|v(y,u)-v(y',u)|}{\varphi(y)d_\theta(y,y')}, \qquad \|v\|_\theta=|v|_\infty+|v|_\theta.
\]
Let $\cF_\theta(Y^\varphi)$ be the space of observables
$v:Y^\varphi\to\R$ with $\|v\|_\theta<\infty$.

Also, for $\eta\in(0,1]$, we define
\[
|v|_{\infty,\eta}=\sup_{(y,u),(y,u')\in Y^\varphi,\,u\neq u'}\frac{|v(y,u)-v(y,u')|}{|u-u'|^\eta}, \qquad \|v\|_{\theta,\eta}=\|v\|_\theta+|v|_{\infty,\eta}.
\]
(Here and elsewhere, $|u-u'|$ denotes absolute value, with $u,u'$ regarded as elements of $[0,\infty)$.)
Let $\cF_{\theta,\eta}(Y^\varphi)$ be the space of observables
$v:Y^\varphi\to\R$ with $\|v\|_{\theta,\eta}<\infty$.

We say that $w:Y^\varphi\to\R$ is {\em differentiable in the flow direction} if the limit
$\partial_tw=\lim_{t\to0}(w\circ F_t-w)/t$ exists pointwise.  Note
that $\partial_tw=\frac{\partial w}{\partial u}$
on the set $\{(y,u):y\in Y,\,0< u< \varphi(y)\}$.
For $m\ge0$, let $w$ be $m$-times differentiable in the flow direction
and define $|w|_{\infty,m}=\sum_{j=0}^m|\partial_t^jw|_\infty$.
Let $L^{\infty,m}(Y^\varphi)$ be the space of observables
$w:Y^\varphi\to\R$ with $|w|_{\infty,m}<\infty$.

\subsubsection{Decay of correlations}
We can now state the main results for Gibbs-Markov semiflows.

\begin{thm} \label{thm:slow}
Let $F_t:Y^\varphi\to Y^\varphi$ be a Gibbs-Markov semiflow
such that $\mu(\varphi>t)=O(t^{-\beta})$ for some $\beta>1$.
Assume absence of approximate eigenfunctions.

Then for any $\eta\in(0,1]$, there exists $m\ge1$, $\theta\in(0,1)$ and $C>0$ such that
\[
|\rho_{v,w}(t)|\le C\|v\|_{\theta,\eta}|w|_{\infty,m} \,t^{-(\beta-1)}
\quad\text{for all $v\in \cF_{\theta,\eta}(Y^\varphi)$, $w\in L^{\infty,m}(Y^\varphi)$, $t>1$}.
\]
\end{thm}

For rapid mixing,
we obtain a slightly stronger result where $v$ lies in $\cF_\theta(Y^\varphi)$
(instead of $\cF_{\theta,\eta}(Y^\varphi)$).
For convenience, we state the result on rapid mixing separately.

\begin{thm} \label{thm:rapid}
Let $F_t:Y^\varphi\to Y^\varphi$ be a Gibbs-Markov semiflow
such that $\varphi\in L^q(Y)$ for all $q\ge1$.
Assume absence of approximate eigenfunctions.

Then $F_t$ is rapid mixing:
for any $q\in\N$, there exists $m\ge1$, $\theta\in(0,1)$ and $C>0$ such that
\[
|\rho_{v,w}(t)|\le C\|v\|_\theta|w|_{\infty,m} \,t^{-q}
\quad\text{for all $v\in \cF_\theta(Y^\varphi)$, $w\in L^{\infty,m}(Y^\varphi)$, $t>1$}.
\]
\end{thm}
The proofs of Theorems~\ref{thm:slow} and~\ref{thm:rapid} are given in Sections~\ref{sec:slow} and~\ref{sec:rapid} respectively.

Theorem~\ref{thm:slow} is sharp.
Precise asymptotics and error rates hold for ``nicely'' supported observables:

\begin{thm}[ {\cite[Theorem~2.4(a)]{MT17}} ] \label{thm:MT}
Assume the setup of Theorem~\ref{thm:slow} and suppose further
that 
$\supp v,\,\supp w\in Y\times[0,\inf\varphi]$.   
For $\beta\in(1,2)$, set $s=2(\beta-1)$.  For $\beta\ge2$, choose any $s<\beta$.
Then 
\[
\rho_{v,w}(t)=\intphi^{-1}\int_{Y^\varphi}v\,d\mu^\varphi
\int_{Y^\varphi}w\,d\mu^\varphi\,\int_t^\infty \mu(\varphi>t')\,dt'+O(
\|v\|_\theta|w|_{\infty,m}\, t^{-s}).
\]

In particular, if $\int_t^\infty \mu(\varphi>t')\,dt'\sim ct^{-(\beta-1)}$ for some $c>0$, then
\[
\SMALL \rho_{v,w}(t)\sim c\intphi^{-1}\int_{Y^\varphi}v\,d\mu^\varphi
\int_{Y^\varphi}w\,d\mu^\varphi\,t^{-(\beta-1)}
\quad\text{as $t\to\infty$}.
\]

\vspace{-5ex}
\qed
\end{thm}

We also mention a result which is the continuous time
analogue of~\cite[Theorem~1.3, last statement]{Gouezel04a}.

\begin{thm}[ {\cite[Theorem~2.4(b)]{MT17}} ] \label{thm:MT0}
Assume the setup of Theorem~\ref{thm:MT}.  If in addition
$\int_{Y^\varphi}v\,d\mu^\varphi=0$ or
$\int_{Y^\varphi}w\,d\mu^\varphi=0$, then
$|\rho_{v,w}(t)| \le C
\|v\|_\theta|w|_{\infty,m} \,t^{-(\beta-\eps)}$ for all $\eps>0$.~
\qed
\end{thm}

\subsection{Nonuniformly expanding semiflows on metric spaces}
\label{sec:ambient}

Let $T_t:M\to M$ be a semiflow defined on a
metric space $(M,d)$ with $\diam M\le 1$.  
Fix $\eta\in(0,1]$.

Given $v:M\to\R$, define ${|v|}_{C^\eta}=\sup_{x\neq x'}|v(x)-v(x')|/d(x,x')^\eta$ and ${\|v\|}_{C^\eta}=|v|_\infty+{|v|}_{C^\eta}$.
Let $C^\eta(M)=\{v:M\to\R:{\|v\|}_{C^\eta}<\infty\}$.
Also, define
${|v|}_{C^{0,\eta}}=\sup_{x\in M,\,t>0}|v(T_tx)-v(x)|/t^\eta$ and
let $C^{0,\eta}(M)=\{v:M\to\R:|v|_\infty+{|v|}_{C^{0,\eta}}<\infty\}$.
(Such observables are H\"older in the flow direction.)

We say that $w:M\to\R$ is differentiable in the flow direction if the limit
$\partial_tw=\lim_{t\to0}(w\circ T_t-w)/t$ exists pointwise.
Define $|w|_{\infty,m}=\sum_{j=0}^m|\partial_t^jw|_\infty$ and let
$L^{\infty,m}(M)=\{w:M\to\R:|w|_{\infty,m}<\infty\}$.

Let $X\subset M$ be a Borel subset and
define $C^\eta(X)$ using the metric $d$ restricted to $X$.
We suppose that 
$T_{h(x)}x\in X$ for all $x\in X$,
 where $h:X\to\R^+$ lies in $C^\eta(X)$ and $\inf h>0$.
In addition, we suppose that there exists $C>0$ such that
\begin{align} \label{eq:holder}
d(T_tx,T_tx')\le Cd(x,x')^\eta
\quad\text{for all $t\in[0,|h|_\infty]$, $x,x'\in M$.}
\end{align}

Define $f:X\to X$ by $fx=T_{h(x)}x$.
We assume that $f$ is modelled by a Young tower with polynomial tails~\cite{Young99}.  This means that there is a Borel subset $Y\subset X$ with 
finite Borel measure $\mu_0$ and a return time $\tau:Y\to\Z^+$ 
with $\mu_0(\tau>n)=O(n^{-\beta})$ for some $\beta>1$ such that
$Fy=f^{\tau(y)}y\in Y$ for almost all $y\in Y$.
Moreover, 
there is an at most countable measurable partition $\{Y_j\}$ such that
$\tau$ is constant on partition elements,
and constants $\lambda>1$, $C>0$, such that for all $j\ge1$, 
\begin{itemize}
\item[(1)] $F$ restricts to a (measure-theoretic) bijection from $Y_j$  onto $Y$.
\item[(2)] $d(Fy,Fy')\ge \lambda d(y,y')$ for all $y,y'\in Y_j$.
\item[(3)] $\zeta_0=\frac{d\mu_0}{d\mu_0\circ F}$
satisfies $|\log \zeta_0(y)-\log \zeta_0(y')|\le Cd(Fy,Fy')^\eta$ for all $y,y'\in Y_j$.
\item[(4)] $d(f^\ell y,f^\ell y')\le Cd(Fy,Fy')$
for all $0\le \ell\le \tau(y)$, $y,y'\in Y_j$.
\end{itemize}

A standard consequence of conditions (1)--(3) is that
there is a unique ergodic $F$-invariant absolutely continuous probability measure $\mu$ on $Y$.   Moreover, $d\mu/d\mu_0\in L^\infty(Y)$ and 
$F:Y\to Y$ is a Gibbs-Markov map.

Define the induced roof function $\varphi:Y\to\R^+$
by $\varphi(y)=\sum_{\ell=0}^{\tau(y)-1}h(f^\ell y)$.
Since $\varphi\le |h|_\infty\tau$, it follows that
$\mu(\varphi>t)=O(t^{-\beta})$ and in particular $\varphi$ is integrable.
We can now define the suspension semiflow $F_t:Y^\varphi\to Y^\varphi$ and ergodic $F_t$-invariant and $T_t$-invariant probability measures $\mu^\varphi$ and
$\mu_M=\pi_*\mu^\varphi$ as in Section~\ref{sec:susp}.

\begin{prop}    \label{prop:NUE}
Let $T_t:M\to M$ be a semiflow satisfying conditions (1)--(4) and~\eqref{eq:holder} with $h$ H\"older and $\inf h>0$.  Set $\theta=\lambda^{-\eta^2}$.  Then
\\[.75ex]
(a) $F_t:Y^\varphi\to Y^\varphi$ is a Gibbs-Markov semiflow.
\\[.75ex]
(b)  Observables $v\in C^\eta(M)\cap C^{0,\eta}(M)$ lift to observables
$\tilde v=v\circ\pi:Y^\varphi\to \R$ that lie in $\cF_{\theta,\eta}(Y^\varphi)$.
Moreover, there is a constant $C>0$ such that
$\|\tilde v\|_{\theta,\eta}\le C({\|v\|}_{C^\eta}+{\|v\|}_{C^{0,\eta}})$.
\end{prop}

\begin{proof}
Set $\theta_1=\lambda^{-\eta}$.  
Let $y,y'\in Y$ with separation time $n=s(y,y')$.
By (2),
\[
1\ge \diam Y\ge d(F^ny,F^ny')\ge \lambda^n d(y,y'),
\]
so $d(y,y')^\eta \le (\lambda^{-n})^\eta= \theta_1^n=d_{\theta_1}(y,y')$.

Define $h_\ell=\sum_{j=0}^{\ell-1}h\circ f^j$.
By (4), for $y,y'\in Y_j$, $\ell=0,\dots,\tau(y)-1$, 
\begin{align} \label{eq:h} \nonumber
|h_\ell(y) & -h_\ell(y')|  \le \sum_{j=0}^{\tau(y)-1}|h(f^j y)-h(f^j y')| 
 \le |h|_\eta \sum_{j=0}^{\tau(y)-1}d(f^j y,f^j y')^\eta 
\ll \tau(y) d(Fy,Fy')^\eta 
 \\ & \le \tau(y) d_{\theta_1}(Fy,Fy')
 \le \theta_1^{-1}(\inf h)^{-1} \infYj\varphi\, d_{\theta_1}(y,y')
\ll \infYj\varphi\, d_{\theta_1}(y,y').
\end{align}
Taking $\ell=\tau(y)$ in~\eqref{eq:h}, we obtain
$|\varphi(y)  -\varphi(y')|  \ll 
 \infYj\varphi\, d_{\theta}(y,y')$ proving part~(a).

Next, let $(y,u),\,(y',u)\in Y^\varphi$ with $s(y,y')\ge1$.
There exists $\ell,\ell'\in\{0,\dots,\tau(y)-2\}$ such that
\[
u\in[h_\ell(y),h_{\ell+1}(y)]\cap [h_{\ell'}(y'),h_{\ell'+1}(y')].
\]
Suppose without loss that $\ell\le\ell'$.
Then
\[
u=h_\ell(y)+r=h_\ell(y')+r',
\]
where $r\in[0,|h|_\infty]$ and
 $r'= u-h_\ell(y')\ge u- h_{\ell'}(y')\ge0$.
Note that $T_uy=T_rT_{h_\ell(y)}y=T_rf^\ell y$ and similarly
$T_uy'=T_{r'}f^\ell y'$.
Hence
$\tilde v(y,u)-\tilde v(y',u) = v(T_rf^\ell y)-v(T_{r'}f^\ell y')$.
By conditions~(4) and~\eqref{eq:holder},
\begin{align*}
|v(T_rf^\ell y)-v(T_rf^\ell y')| & \le {|v|}_{C^\eta}
 d(T_rf^\ell y,T_rf^\ell y')^\eta
\ll {|v|}_{C^\eta} d(f^\ell y,f^\ell y')^{\eta^2}
\\ & \ll {|v|}_{C^\eta} d(Fy,Fy')^{\eta^2} \le  \theta^{-1}{|v|}_{C^\eta} d_{\theta}(y,y'),
\end{align*}
and by~\eqref{eq:h},
\[
|v(T_rf^\ell y')-v(T_{r'}f^\ell y')| \le {|v|}_{C^{0,\eta}} |r-r'|^\eta
 ={|v|}_{C^{0,\eta}} |h_\ell(y)-h_\ell(y')|^\eta
 \ll {|v|}_{C^{0,\eta}} 
\varphi(y)d_{\theta}(y,y').
\]
Hence 
$|\tilde v(y,u)-\tilde v(y',u)|\ll 
({|v|}_{C^\eta}+ {|v|}_{C^{0,\eta}}) \varphi(y)d_{\theta}(y,y')$,
whenever $s(y,y')\ge1$.  For \mbox{$s(y,y')=0$}, we have the estimate
$|\tilde v(y,u)-\tilde v(y',u)|\le 2|v|_\infty=2|v|_\infty d_\theta(y,y')
\ll |v|_\infty \,\varphi(y) d_\theta(y,y')$, so in all cases we obtain
$|\tilde v(y,u)-\tilde v(y',u)|\ll ({\|v\|}_{C^\eta}+{|v|}_{C^{0,\eta}})
\varphi(y) d_\theta(y,y')$.
Also, 
\[
|\tilde v(y,u)-\tilde v(y,u')|
=|v(T_uy)-v(T_{u'}y)|\le {|v|}_{C^{0,\eta}} |u-u'|^{\eta}.
\]
Hence $\tilde v\in \cF_{\theta,\eta}(Y^\varphi)$
completing the proof of part~(b).
\end{proof}

It follows that the main results in Subsection~\ref{sec:GM} go over to semiflows $T_t:M\to M$ satisfying (1)--(4):  

\begin{thm} \label{thm:NUE}
Let $T_t:M\to M$ be a semiflow satisfying conditions (1)--(4) and~\eqref{eq:holder} with $h$ H\"older and $\inf h>0$.
Suppose that $\mu_0(\tau>t)=O(t^{-\beta})$ for some $\beta>1$.
Assume absence of approximate eigenfunctions for the corresponding
Gibbs-Markov semiflow $F_t:Y^\varphi\to Y^\varphi$.

Then there exists $m\ge1$ and $C>0$ such that
\[
|\rho_{v,w}(t)|\le C({\|v\|}_{C^\eta}+{|v|}_{C^{0,\eta}})|w|_{\infty,m} \,t^{-(\beta-1)},
\]
for all $v\in C^\eta(M)\cap C^{\eta,0}(M)$, $w\in L^{\infty,m}(M)$, $t>1$.
\end{thm}

\begin{proof}
By Proposition~\ref{prop:NUE}(a), $F_t$ is a Gibbs-Markov semiflow with
$\mu(\varphi>t)=O(t^{-\beta})$.
By Proposition~\ref{prop:NUE}(b),
$\tilde v=v\circ\pi \in \cF_{\theta,\eta}(Y^\varphi)$.
Also, 
$\tilde w=w\circ\pi\in L^{\infty,m}(Y^\varphi)$.
Hence by~\eqref{eq:M} and Theorem~\ref{thm:slow}, 
$|\rho_{v,w}(t)|=|\rho_{\tilde v,\tilde w}(t)|\ll
\|\tilde v\|_{\theta,\eta}|\tilde w|_{\infty,m}\, t^{-(\beta-1)}
\ll ({\|v\|}_{C^\eta}+{|v|}_{C^{0,\eta}})|w|_{\infty,m}\, t^{-(\beta-1)}$ for 
$m$ sufficiently large.  
\end{proof}

If $h$ is the first return to $X$ and $\tau$ is the first return to $Y$, then observables supported on $\bigcup_{t\in[0,\inf h]}T_tY$ lift to observables supported on $Y\times[0,\inf\varphi]$ and we can apply Theorem~\ref{thm:MT} to obtain lower bounds for such observables.

\begin{example}[Intermittent semiflows]
\label{ex:LSVsemiflow}

We consider a class of two-dimensional H\"older
semiflows $T_t:M\to M$, with an intermittent Poincar\'e map of
Pomeau-Manneville type~\cite{PomeauManneville80}.
Let $X\subset M$ be the interval $[0,1]$ and suppose as above that
$h:X\to\R^+$ is H\"older with $\inf h>0$ such that
$f(x)=T_{h(x)}x$ defines a map $f:X\to X$.  We continue to assume condition~\eqref{eq:holder}.

Suppose that 
$f:X\to X$ is given by $f(x)=\begin{cases} x(1+2^\gamma x^\gamma) & x\in[0,\frac12) \\ 2x-1 & x\in[\frac12,1] \end{cases}$
where $\gamma\in(0,1)$ is fixed.
The maps $f$ were studied by~\cite{Hu04,LSV99,Young99} and decay of correlations for H\"older observables is $O(n^{-(\beta-1)})$ where $\beta=1/\gamma$.
This rate is sharp~\cite{Sarig02,Gouezel04a}.

Here we obtain the analogous results for semiflows.
Let $Y=[\frac12,1]$.
The first return map $F:Y\to Y$ satisfies conditions~(1)--(4) and the first return time $\tau:Y\to\Z^+$ satisfies $\mu(\tau>n)\sim c_0n^{-\beta}$ for some $c_0>0$.  Hence we obtain decay of correlations $O\big({\|v\|}_{C^\eta}|w|_{\infty,m}\,t^{-(\beta-1)}\big)$ for the semiflow by Theorem~\ref{thm:NUE}.
If in addition $h$ is the first return to $X$, then 
a standard calculation (see for example~\cite[Theorem~1.3]{Gouezel04}
or~\cite[Proposition~2.17]{MT17})
 shows that
$\varphi(y)\sim h(0)\tau(y)$ as $y\to\frac12^+$, 
$\mu(\varphi>t)\sim c_1 t^{-\beta}$ where $c_1=h(0)^\beta c_0$,
and 
$\int_t^\infty \mu(\varphi>t')\,dt'\sim ct^{-(\beta-1)}$
where $c=c_1/(\beta-1)$.   Hence
for H\"older observables 
supported on $\bigcup_{t\in[0,\inf h]}T_tY$ we obtain the asymptotic
$\rho_{v,w}(t)\sim c\intphi^{-1}\int_M v\,d\mu_M
\int_M w\,d\mu_M\,t^{-(\beta-1)}$.

The Markovian nature of the map $f$ simplified the analysis above but is not necessary.   In the nonMarkov case, the first return map $F$ is certainly not Gibbs-Markov, but there still exists a induced map $F:Y\to Y$ (not the first return map) on a suitable set $Y$ such that $F$ is Gibbs-Markov and the return time $\tau$ has the same tails as the first return time~\cite[Section~7]{Young99}.  Hence we still obtain the upper bound $O(t^{-(\beta-1)})$ for mixing rates of H\"older observables by Theorem~\ref{thm:NUE}.
For sharpness of this bound we refer to forthcoming work of~\cite{BMTprep} which extends~\cite{MT17} to a general functional analytic framework including nonMarkov intermittent (semi)flows.
\end{example}

\subsection{Dynamically H\"older semiflows and observables}
\label{sec:dyn}

It is standard that the assumptions on $T_t$, $h$ and $v$ in Theorem~\ref{thm:NUE}
can be relaxed from H\"older to dynamically H\"older, as we now describe.

We continue to assume that $v\in C^{0,\eta}(M)$ and that $\inf h>0$.  
Condition~\eqref{eq:holder} on the semiflow is removed, as are the assumptions that $h$ is H\"older and $v\in C^\eta(M)$.  
Instead, we require that there are constants $C>0$, $\gamma\in(0,1)$ such that for 
all $y,y'\in Y_j$, $j\ge1$,
\begin{align*}
& |h(f^\ell y)-h(f^\ell y')|\le C\gamma^{s(y,y')}
\quad\text{for $0\le\ell\le \tau(y)-1$,}
\\
& |v(T_uy)-v(T_uy')|\le C\varphi(y)\gamma^{s(y,y')}
\quad\text{for $u\in[0,\varphi(y)]\cap[0,\varphi(y')]$.}
\end{align*}
Here, $s(y,y')$ is the separation time for the Gibbs-Markov map $F$.

It is easily verified that the proof of Proposition~\ref{prop:NUE},
and hence Theorem~\ref{thm:NUE}, goes through under these more relaxed assumptions on $T_t$, $h$ and $v$.

\begin{rmk}  The notion of dynamically H\"older can be relaxed even further, see~\cite[Section~7.3]{BBMsub}.  This turns out to be crucial for Lorentz gas examples~\cite[Remark~9.2]{BBMsub}.
\end{rmk}

\section{Mixing rates for nonuniformly hyperbolic flows}
\label{sec:NUH}

In this section, we review results on rapid mixing and polynomial mixing for
nonuniformly hyperbolic flows.  
In Subsection~\ref{sec:skew}, we define a class of Gibbs-Markov flows with skew product structure (the roof function $\varphi$ is constant along stable leaves) and state our result on rates of mixing for such flows.  
In Subsection~\ref{sec:nonskew}, we discuss several situations where the skew product assumption can be relaxed; these include all the examples in the introduction.

\subsection{Skew product Gibbs-Markov flows}
\label{sec:skew}

Let $(Y,d)$ be a metric space with $\diam Y\le1$,
and let $F:Y\to Y$ be a piecewise continuous map with
ergodic $F$-invariant probability measure~$\mu$.
Let $\cW^s$ be a cover of $Y$ by disjoint measurable subsets called {\em stable leaves}.
For each $y\in Y$, let $W^s(y)$ denote the stable leaf containing~$y$.
We require that $F(W^s(y))\subset W^s(Fy)$ for all $y\in Y$.

Let $\bY$ denote the space obtained from $Y$ after quotienting by $\cW^s$,
with natural projection $\bar\pi:Y\to\bY$.
We assume that the quotient map $\bF:\bY\to\bY$ is a Gibbs-Markov map as in Definition~\ref{def:GM}, with
partition $\{\bY\!_j\}$, separation time $s(y,y')$,
 and ergodic invariant probability measure $\bmu =\bar\pi_*\mu$.

Let $Y_j=\bar\pi^{-1}\bY\!_j$; these form a partition of $Y$ and each $Y_j$ is a union of stable leaves.
The separation time extends to $Y$, setting
$s(y,y')=s(\bar\pi y,\bar\pi y')$ for $y,y'\in Y$.

Next, we require that there is a
measurable subset $\tY\subset Y$ such that
for every $y\in Y$ there is a unique $\tilde y\in\tY\cap
W^s(y)$.  Let $\pi:Y\to\tY$ define the associated projection $\pi y=\tilde y$.  (Note that $\tY$ can be identified with $\bY$, but in general $\pi_*\mu\neq\bmu$.)

We assume that there are constants $C>0$, $\gamma\in(0,1)$ such that
\begin{itemize}
\item[(a)] $d(F^ny,F^ny')\le C\gamma^n$ for all $n\ge0$ and all $y,y'\in Y$ with
$y'\in W^s(y)$.
\item[(b)] $d(F^ny,F^ny')\le C\gamma^{s(y,y')-n}$ for all $n\ge0$ and all $y,y'\in \tY$.
\end{itemize}

Let $\varphi:Y\to\R^+$ be an integrable roof function with $\inf\varphi>0$, and define the suspension flow $F_t:Y^\varphi\to Y^\varphi$ with ergodic invariant probability measure $\mu^\varphi$ (see Section~\ref{sec:susp}).\footnote{Strictly speaking, $F_t$ is not always a flow since $F$ need not be invertible.  However, $F_t$ is used as a model for various flows, and it is then a flow when $\varphi$ is the first return to $Y$, so it is convenient to call it a flow.}

In this subsection, we suppose that $\varphi$ is constant along stable leaves and hence projects to a well-defined roof function $\varphi:\bY\to\R^+$.
It follows that the suspension flow $F_t$ projects to
a suspension semiflow $\bF_t:\bY^\varphi\to\bY^\varphi$.
We assume that $\varphi:\bY\to\R^+$ satisfies
condition~\eqref{eq:inf}, so $\bF_t$ is a Gibbs-Markov semiflow.
We call $F_t$ a {\em skew product Gibbs-Markov flow}, and we say
that $F_t$ has {\em approximate eigenfunctions} if 
$\bF_t$ has approximate eigenfunctions (Definition~\ref{def:approx}).

Fix $\eta\in(0,1]$ and let $\cH_{\gamma,\eta}(Y^\varphi)$ be the space of observables
$v:Y^\varphi\to\R$ with
$\|v\|_{\gamma,\eta}<\infty$, where
\[
|v|_{\gamma,\eta}=\sup_{(y,u),\,(y',u')\in Y^\varphi \atop (y,u)\neq (y',u')}\frac{|v(y,u)-v(y',u')|}{\varphi(y)\{d(y,y')+\gamma^{s(y,y')}\}+|u-u'|^\eta},
\qquad \|v\|_{\gamma,\eta}=|v|_\infty+|v|_{\gamma,\eta}.
\]

We also define $\cH_{\gamma,0,m}(Y^\varphi)$ to consist of
observables that lie in $\cH_\gamma(Y^\varphi)$ and are $m$-times differentiable in the flow direction with derivatives in $\cH_\gamma(Y^\varphi)$,
with norm $\|w\|_{\gamma,0,m}=\sum_{j=0}^m \|\partial_t^jw\|_\gamma$.

In~\cite{BBMsub}, we prove:

\begin{thm} \label{thm:skew} Suppose that $F_t:Y^\varphi\to Y^\varphi$ is a skew product Gibbs-Markov flow
such that $\mu(\varphi>t)=O(t^{-\beta})$ for some $\beta>1$.
Assume absence of approximate eigenfunctions.
Then there exists $m\ge1$ and $C>0$ such that
\[
|\rho_{v,w}(t)|\le C\|v\|_{\gamma,\eta}\|w\|_{\gamma,0,m} \,t^{-(\beta-1)}
\quad\text{for all $v\in \cH_{\gamma,\eta}(Y^\varphi)$, $w\in \cH_{\gamma,0,m}(Y^\varphi)$, $t>1$}.
\]
\end{thm}

\subsection{General nonuniformly hyperbolic flows}
\label{sec:nonskew}

In this subsection, we drop the assumption that $\varphi$ is constant along stable leaves,
 and mention various classes of nonuniformly hyperbolic flows that do not possess a skew product structure to which our methods apply.
For details we refer to~\cite{BBMsub} who introduce a class of
{\em Gibbs-Markov flows} that are conjugate by a H\"older conjugacy to a skew product Gibbs-Markov flow.  As shown in~\cite{BBMsub}, our results on rapid and polynomial mixing go over to Gibbs-Markov flows and to nonuniformly hyperbolic flows that are modelled by Gibbs-Markov flows.  This includes the following situations.

\vspace{-2ex}
\paragraph{(i) Flows with exponential contraction along stable leaves}
Condition~(a) in Subsection~\ref{sec:skew} asserts exponential contraction along stable leaves for the uniformly hyperbolic map $F:Y\to Y$.  This means that exponential contraction is assumed only on returns to the inducing set $Y$.
Note that such a return occurs at the flow time $\varphi_n=\sum_{j=0}^{n-1}\varphi\circ F^j$, so
an alternative stronger condition is that
$d(F^ny,F^ny')\le C\gamma^{\varphi_n(y)}$ for all $n\ge0$ and all $y,y'\in Y$ with $y'\in W^s(y)$.
This condition was assumed in~\cite{M07,M09} and incorporates all of the rapidly mixing examples in Section~\ref{sec:introrapid}.  Indeed, it is automatic for flows modelled by Young towers with exponential tails.   However for flows modelled by Young towers with polynomial tails, the condition is very restrictive and excludes the slowly mixing billiard examples.

\vspace{-2ex}
\paragraph{(ii) Roof functions with bounded H\"older constants}
Condition~\eqref{eq:inf} reflects the fact that the variation of $\varphi$ on partition elements is likely to be as large as the size of $\varphi$.  The argument in Section~\ref{sec:ambient} indicates that this is the ``correct'' condition in general.  
A stronger condition is that $|\varphi(y)-\varphi(y')|\le C_1 \theta^{s(y,y')}$ along unstable leaves.
This includes~\cite{BBMsub} the slowly mixing Lorentz gas examples in Section~\ref{sec:introslow}.

\vspace{-2ex}
\paragraph{(iii) Flows with H\"older stable foliation}
Under the assumption that $T_t$ has a H\"older stable foliation $\cW^{ss}$ (in a neighborhood of the attractor $\Lambda$) we can reduce to the skew product case by using a cross-section $Y$ comprising leaves in $\cW^{ss}$.
(For flows with a $DT_t$-invariant dominated splitting $T_\Lambda M=E^{ss}\oplus E^{cu}$, results on the existence of a H\"older stable foliation $\cW^{ss}$ can be found in~\cite[Section~4]{AraujoM17} and~\cite[Theorem~6.2]{AraujoMsub}.)

\begin{example}[Intermittent solenoidal flows]
\label{ex:LSVflow}

The classical Smale-Williams solenoid construction can be adapted
(see for example~\cite[Section~5]{AlvesPinheiro08} and~\cite[Example~4.2]{MV16})
 to construct intermittent transformations $f:X\to X$ that are the invertible analogue of the intermittent maps in Example~\ref{ex:LSVsemiflow}.
These have polynomial decay rates $O(n^{-(\beta-1)})$ for any specified $\beta>1$.
Hence we can construct intermittent flows $T_t:M\to M$ with $f:X\to X$ as a Poincar\'e map and H\"older return time function $h:X\to\R^+$ with $\inf h>0$.

The examples in~\cite{AlvesPinheiro08} and some of the examples in~\cite{MV16}
have exponential contraction along stable leaves and hence fall within scenario (i).  However even the examples in~\cite{MV16} with slow contraction along stable leaves are covered by scenario~(iii).  Hence we obtain polynomial mixing $O(t^{-(\beta-1)})$ for general intermittent solenoidal flows.
Again these results are sharp by~\cite{MT17} in the Markovian case and~\cite{BMTprep} in the general case.  Optimal lower bounds, asymptotics and error rates are achieved by observables that are supported away from the neutral periodic orbit and constant along stable leaves.  
\end{example}

\section{Criteria for absence of approximate eigenfunctions}
\label{sec:approx}

The approximate eigenfunction condition in Definition~\ref{def:approx} is somewhat technical.  In this subsection, we discuss three sufficient conditions to rule out the existence of approximate eigenfunctions.
We suppose throughout that $F_t:Y^\varphi\to Y^\varphi$ is a Gibbs-Markov semiflow or skew product Gibbs-Markov flow, but it is immediate from the definitions in~\cite{BBMsub} that 
the conditions apply to the situations mentioned in Subsection~\ref{sec:nonskew}.

In Subsection~\ref{sec:Dio}, we show that a Diophantine condition on the periods of three periodic solutions ensures absence of approximate eigenfunctions.  This condition is satisfied with probability one but is not robust.  In Subsection~\ref{sec:good}, we use good asymptotics of periodic data to give an open and dense condition.  In Subsection~\ref{sec:D}, we define the temporal distortion function and give a condition involving the dimension of its range.

In preparation for Subsections~\ref{sec:Dio} and~\ref{sec:good}, we recall the relationship between periodic data and approximate eigenfunctions.
Define $\varphi_n=\sum_{j=0}^{n-1}\varphi\circ F^j$.
If $y$ is a periodic point of period $p$ for $F$ (that is, $F^py=y$),
then $y$ is periodic of period $\period=\varphi_p(y)$ for $F_t$
(that is, $F_\period y=y$).

\begin{prop} \label{prop:period}
Suppose that there exist approximate eigenfunctions on $Z_0\subset Y$.
Let $\alpha,C, b_k,n_k$ be as in Definition~\ref{def:approx}.
If $y\in Z_0$ is a periodic point with $F^py=y$ and $F_\period y=y$
where $\period=\varphi_p(y)$, then
\begin{align} \label{eq:period}
\dist(b_kn_k\period-p\psi_k,2\pi\Z)\le C(\inf\varphi)^{-1}\period|b_k|^{-\alpha}
\quad\text{for all $k\ge1$}.
\end{align}
\end{prop}

\begin{proof}
For $n,p\ge1$,
\[
M_b^{np}u-e^{ip\psi}u=M_b^n(M_b^{n(p-1)}u-e^{i(p-1)\psi}u)+e^{i(p-1)\psi}(M_b^nu-e^{i\psi}u),
\]
so $|M_b^{np}u-e^{ip\psi}u|\le |M_b^{n(p-1)}u-e^{i(p-1)\psi}u|+|M_b^nu-e^{i\psi}u|$.
Inductively, 
$|M_b^{np}u-e^{ip\psi}u|\le p|M_b^nu-e^{i\psi}u|$ and hence by~\eqref{eq:approx},
\begin{align} \label{eq:approxp}
|(M_{b_k}^{pn_k}u_k)(y)-e^{ip\psi_k}u_k(y)|\le Cp|b_k|^{-\alpha}
\quad\text{for all $y\in Z_0$, $k\ge1$, $p\ge1$.}
\end{align} 

Next, $(M_b^pv)(y)=e^{ib\varphi_p(y)}v(F^py)=e^{ib\period}v(y)$.
Hence substituting $y$ into~\eqref{eq:approxp},
we obtain $|e^{ib_kn_k\period}-e^{ip\psi_k}|\le Cp|b_k|^{-\alpha}$.
Also $\period=\varphi_p(y)\ge p\inf\varphi$.
\end{proof}

\begin{rmk} \label{rmk:tq}
We now have the notation required to provide the formula for $t_q$ promised in
Section~\ref{sec:stats}.  Recall that $q$ is a periodic point for the flow.  Let $y\in Y$ be a point lying on the periodic orbit through $q$ with $F^py=y$.  By~\cite[Corollary~6.2]{AMV15}, we can choose $t_q=\varphi_p(y)$, so $t_q$ coincides with the period of $q$ for the suspension flow $F_t$.  (If $\varphi$ is the first return time to $Y$ then $t_q$ is also the period of $q$ under $T_t$.)
\end{rmk}

\subsection{Diophantine condition on periods}
\label{sec:Dio}

\begin{prop} \label{prop:tau}
Let $y_1,y_2,y_3\in\bigcup Y_j$ be fixed points for $F$,
and let $\period_i=\varphi(y_i)$, $i=1,2,3$, be the 
corresponding periods for $F_t$.
Let $Z_0$ be the finite subsystem corresponding to the three partition elements containing $y_1,y_2,y_3$.

If $(\period_1-\period_3)/(\period_2-\period_3)$ is Diophantine, then
 there do not exist approximate eigenfunctions on $Z_0$.
\end{prop}

\begin{proof}
We give the proof when $F_t$ is a Gibbs-Markov semiflow.  It is immediate from the definitions that the result is the same for skew product Gibbs-Markov flows.

Define $\period_{i3}=\period_i-\period_3$, $i=1,2$.
For $\period_{13}/\period_{23}$ Diophantine,
there exists $\alpha'>2$ such that
$|\period_{13}/\period_{23}-m_1/m_2|\le C|m_2|^{-\alpha'}$ has only finitely many integer solutions $m_1,m_2$ for each $C>0$.

Arguing as in~\cite[Section~13]{Dolgopyat98b}, 
we choose $\alpha>\alpha'$.
Suppose for contradiction that there exist approximate eigenfunctions on $Z_0$.
Substituting $y=y_i$ in~\eqref{eq:period}, we obtain the estimates
\[
\dist(b_kn_k\period_i-\psi_k,2\pi\Z)=O(|b_k|^{-\alpha}),\quad i=1,2,3,
\]
where $|b_k|\to\infty$ and $n_k=O(\ln|b_k|)$ as $k\to\infty$.
Eliminating $\psi_k$, we obtain
\[
\dist(b_kn_k\period_{i3},2\pi\Z) =O(|b_k|^{-\alpha}), \quad i=1,2.
\]
Hence, for each $k$ we have
integers $m_1,m_2$ such that
\[
b_kn_k\period_{i3}=2\pi m_i + O(|b_k|^{-\alpha}),
\]
where in particular $m_2=O(b_kn_k)=O(b_k\ln|b_k|)$.
Also, $m_1/m_2\sim \period_{13}/\period_{23}$ so $m_1=O(m_2)$.
Hence
\[
b_kn_k\period_{i3}=2\pi m_i + O(|m_2|^{-\alpha'}),\quad i=1,2.
\]
It follows that
$\frac{\period_{13}}{\period_{23}}=\frac{m_1}{m_2}+O(|m_2|^{-\alpha'})$,
which is the desired contradiction.
\end{proof}

\subsection{Good asymptotics}
\label{sec:good}

We recall the following definition from~\cite{FMT07}:

\begin{defn} \label{def:good}
Let $y_0\in Y$ be a fixed point for $F$ with period $\period_0=\varphi(y_0)$ for the flow.
A sequence of periodic points $y_N\in Y$, $N\ge1$,
with $F^Ny_N=y_N$ has {\em good asymptotics} if their periods $\period_N=\varphi_N(y_N)$ for the flow satisfy
\[
\period_N=N\period_0+\kappa+E_N\gamma^N\cos(N\omega+\omega_N)+o(\gamma^N),
\]
where $\kappa\in\R$, $\gamma\in(0,1)$ are constants,
$E_N\in\R$ is a bounded sequence with \mbox{$\liminf_{N\to\infty}|E_N|>0$},
and either (i) $\omega=0$ and $\omega_N\equiv0$, or (ii) 
$\omega\in(0,\pi)$ and $\omega_N\in(\omega_0-\pi/12,\omega_0+\pi/12)$
for some $\omega_0$.   
\end{defn}

\begin{prop}  
If there exists
a sequence of periodic points with good asymptotics in a finite subsystem $Z_0$, then there do not
exist approximate eigenfunctions on $Z_0$.
\end{prop}

\begin{proof}
We argue as in the proof of~\cite[Theorem~1.6(a)]{FMT07}.
The proof works for any fixed $\alpha>0$.

Note that $\period_N=O(N+1)$.
Suppose for contradiction that there exist approximate eigenfunctions on $Z_0$.
Substituting~$y_N$ into~\eqref{eq:period} we obtain the estimates
\[
\dist(b_kn_k\period_N-N\psi_k,2\pi\Z)= O(N|b_k|^{-\alpha}), \qquad
\dist(b_kn_kN\period_0-N\psi_k,2\pi\Z) = O(N|b_k|^{-\alpha}),
\]
for all $k\ge1$, $N\ge1$.  Hence we can eliminate $\psi_k$ and $\period_0$
simultaneously, yielding
\begin{align} \label{eq:good}
\dist(b_kn_k(\kappa+E_N\gamma^N\cos(N\omega+\omega_N)+o(\gamma^N)),2\pi\Z)= O(N|b_k|^{-\alpha}),
\end{align}
for all $k\ge1$, $N\ge1$.

Momentarily, choose $N=N(k)=[\rho\ln|b_k|]$ in~\eqref{eq:good}.  
For $\rho>0$ large enough but fixed, 
we have $|b_k|n_k\gamma^{N(k)}=O(|b_k|^{-\alpha})$.
It follows that $\dist(b_kn_k\kappa,2\pi\Z)=O(|b_k|^{-\alpha}\ln|b_k|)$ for all $k\ge1$.
Combining this with~\eqref{eq:good} (which still holds for all $N\ge1$),
we obtain
for any $\alpha'<\alpha$ that
\begin{align} \label{eq:alpha}
\dist(b_kn_k(E_N\gamma^N\cos(N\omega+\omega_N)+o(\gamma^N)),2\pi\Z)= O(N|b_k|^{-\alpha'}),
\end{align}
for all $k\ge1$, $N\ge1$.

Let $S=\sup_N|E_N|$ and set $M(k)=[\ln(|b_k|n_kS)/(-\ln \gamma)]+1$.
Then $|b_k|n_kS\gamma^{M(k)}\in[\gamma,1)$.
Taking $N=M(k)+j$ in~\eqref{eq:alpha}, we obtain
\[
\lim_{k\to\infty} b_kn_k E_{M(k)+j}\gamma^{M(k)}\cos((M(k)+j)\omega+\omega_{M(k)+j})=0
\quad\text{for all $j\in\Z$.}
\]
Since $|b_kn_k\gamma^{M(k)}|\ge\gamma/S$ and $\liminf_{N\to\infty}|E_N|>0$,
it follows that
\begin{align} \label{eq:omega}
\lim_{k\to0}\{(M(k)+j)\omega+\omega_{M(k)+j}\}={\SMALL \frac{\pi}{2}}\bmod\pi
\quad\text{for all $j\in\Z$.}
\end{align}
This is clearly impossible when $\omega=\omega_N=0$ for all $N$, so it remains to consider the case $\omega\in(0,\pi)$ and $|\omega_N-\omega_0|<\pi/12$.
Taking differences of~\eqref{eq:omega} for different values of $j$, we obtain that
$\ell\omega\in[-\pi/6,\pi/6] \bmod\pi$ for all $\ell$, which is impossible,
providing the desired contradiction.~
\end{proof}

As shown in~\cite{FMT07}, for any finite subsystem $Z_0$, the existence of periodic points with good asymptotics in $Z_0$ is a $C^2$-open and $C^\infty$-dense condition.   (Although~\cite{FMT07} is set in the uniformly hyperbolic setting, the construction of periodic points with good asymptotics uses only the existence of a transverse homoclinic point $y_0$.)

\subsection{Temporal distance function}
\label{sec:D}

In this subsection, following~\cite[Section~5.3]{M09} and~\cite[Section~3]{AMV15},
we extend an argument of Dolgopyat~\cite[Appendix]{Dolgopyat98b} for Axiom~A flows to the nonuniformly hyperbolic setting.  
We assume throughout that we are in the setup of Section~\ref{sec:skew}.
In particular, $\varphi$ is constant along stable leaves in the 
cover $\cW^s$.

First, we recall the notion of product structure on $Y$, taking only the parts from~\cite{Young98} that are needed here.  
Assume that there is a second cover $\cW^u$ of $Y$ by disjoint measurable subsets (called
{\em unstable leaves}) such that each stable leaf intersects each unstable leaf in precisely one point.
For each $y\in Y$ let $W^u(y)$ denote the unstable leaf containing $y$.
We require that $F(W^u(y)\cap Y_j)\supset W^u(Fy)$ for all $y\in Y_j$, $j\ge1$.

Also we assume that there are constants $C>0$, $\gamma\in(0,1)$ such that
\[
|\varphi(y)-\varphi(y')|\le C \gamma^{s(y,y')}
\quad\text{for all $y,y'\in Y$ with $s(y,y')\ge1$ and $y'\in W^u(y)$}.
\]

Given $y,y'\in Y$, $y'\in W^u(y)$,
we define compatible inverse branches $F^ny$, $F^ny'$ for \mbox{$n\le-1$} as follows.  First, set $z_0=y$, $z_0'=y'$.
By transitivity of $F$, there exists
$z_1\in Y$ such that $Fz_1=z_0$.  Let $Y_{j_1}$ be the partition element containing $z_1$.  Since $F(W^u(z_1)\cap Y_{j_1}))\supset W^u(z_0)$, there exists $z_1'\in W^u(z_1)\cap Y_{j_1}$ such that $Fz_1'=z_0'$.   Inductively, we obtain a sequence of partition elements $Y_{j_n}$
and pairs of points $z_n,z_n'\in Y_{j_n}$ with $z_n'\in W^u(z_n)$ for all $n\ge1$.   
Note that 
$|\varphi(z_n)- \varphi(z_n')|\le C\gamma^{s(z_n,z_n')}\le C\gamma^n$
for all $n\ge1$, 
so we obtain a well-defined function
\[
D_0(y,y')=\sum_{n=1}^\infty (\varphi(z_n)-\varphi(z_n'))
=\sum_{n=-\infty}^{-1} (\varphi(F^ny)-\varphi(F^ny')).
\]

Now, let $y_1,y_4\in Y$ and set
$y_2=W^s(y_1)\cap W^u(y_4)$,
$y_3=W^u(y_1)\cap W^s(y_4)$.
For $n\le-1$ choose compatible inverse branches $F^{n}y_1$ and 
$F^{n}y_3$ as done in the definition of $D_0$, and similarly for
$F^ny_4$ and $F^ny_2$.
Define the {\em temporal distance function} $D:Y\times Y\to\R$,
\begin{align*}
 D(y_1,y_4)  & =\sum_{n=-\infty}^\infty \Big(\varphi(F^ny_1)-\varphi(F^ny_2)-\varphi(F^ny_3)+\varphi(F^ny_4)\Big)
 \\ & =\sum_{n=-\infty}^{-1} \Big(\varphi(F^ny_1)-\varphi(F^ny_2)-\varphi(F^ny_3)+\varphi(F^ny_4)\Big),
\end{align*}
where the last equality follows from the assumption that $\varphi$ is constant along stable leaves.
This coincides with $D_0(y_1,y_3)-D_0(y_4,y_2)$ and hence is well-defined.

\begin{thm} \label{thm:D}
Let $Z_0=\bigcap_{n=0}^\infty F^{-n}Z$ where $Z$ is a union of finitely many elements of the partition $\{Y_j\}$.  Let $\bar Z_0$ denote the corresponding finite subsystem of $\bY$.  If 
the lower box dimension of $D(Z_0\times Z_0)$ is positive,
then there do not exist approximate eigenfunctions on~$\bar Z_0$.
\end{thm}

\begin{proof}
We suppose that there exist approximate eigenfunctions on $\bar Z_0$ and
show that $\underline{\mathrm{BD}}(D(Z_0\times Z_0))=0$.

First, note that $D(y_1,y_4)=D_m(y_1,y_4)+O(\gamma^m)$ where
\begin{align*}
D_m(y_1,y_4)  
& =\sum_{n=-m}^{-1} \Big(\varphi(F^ny_1)-\varphi(F^ny_2)-\varphi(F^ny_3)+\varphi(F^ny_4)\Big) \\
& = \varphi_m(F^{-m}y_1) - \varphi_m(F^{-m}y_2)
- \varphi_m(F^{-m}y_3) + \varphi_m(F^{-m}y_4).
\end{align*}
Hence
\begin{align} \label{eq:Dm}
\exp\{ibD(y_1,y_4)\}=\frac{
\exp\{ib\varphi_m(F^{-m}y_1)\}
\exp\{ib\varphi_m(F^{-m}y_4)\}
}{
\exp\{ib\varphi_m(F^{-m}y_2)\}
\exp\{ib\varphi_m(F^{-m}y_3)\}
}
+O(|b|\gamma^m).
\end{align}

By Definition~\ref{def:approx}, 
there are sequences such that
\[
|\exp\{ib_k\bphi_{n_k}\}u_k\circ \bF^{n_k}-e^{i\psi_k}u_k|\ll |b_k|^{-\alpha}
\quad\text{on $\bar Z_0$}.
\]
Recall from Definition~\ref{def:approx} that 
$|u_k|\equiv1$, so
\[
\exp\{ib_k\varphi_{n_k}\}=
\exp\{ib_k\bphi_{n_k}\}\circ\bar\pi
=\frac{e^{i\psi_k}u_k\circ\bar\pi}{u_k\circ\bar\pi\circ F^{n_k}}+O(|b_k|^{-\alpha}) \quad\text{on $Z_0$}.
\]
Substituting into~\eqref{eq:Dm} and using that $\bar\pi y_1=\bar\pi y_2$,
$\bar\pi y_3=\bar\pi y_4$, we obtain
\[
\exp\{ib_k D(y_1,y_4)\}
=\frac{u_k\circ\bar\pi(F^{-n_k}y_1)u_k\circ\bar\pi(F^{-n_k}y_4)}
{u_k\circ\bar\pi(F^{-n_k}y_2)u_k\circ\bar\pi(F^{-n_k}y_3)}+O(|b_k|^{-\alpha})
+O(|b_k|\gamma^{n_k}).
\]

Recall from Definition~\ref{def:approx} that 
$|u_k|_\theta\ll |b_k|$.  Also,
$\alpha$ can be fixed arbitrarily large and $n_k=[\xi\ln|b_k|]$ where $\xi$ can be fixed arbitrarily large,  We choose $\xi$ so that 
$\xi\ln\theta+1\le -\alpha$ and
$\xi\ln\gamma+1\le -\alpha$.
Then
\[
|b_k|\theta^{n_k}\le\theta^{-1}|b_k|\theta^{\xi\ln|b_k|}=\theta^{-1}|b_k|^{\xi\ln\theta+1}\le \theta^{-1}|b_k|^{-\alpha},
\]
 and similarly
$|b_k|\gamma^{n_k}\le \theta^{-1}|b_k|^{-\alpha}$.  Hence
\begin{align*}
& \Big|\frac{u_k\circ\bar\pi(F^{-n_k}y_1)} {u_k\circ\bar\pi(F^{-n_k}y_3)}-1\Big|
 =|u_k\circ\bar\pi(F^{-n_k}y_1)-u_k\circ\bar\pi(F^{-n_k}y_3)|
\\ & \le 
|u_k|_\theta d_\theta(\bar\pi(F^{-n_k}y_1),\bar\pi(F^{-n_k}y_3))
\ll |b_k|\theta^{s(\bar\pi(F^{-n_k}y_1),\bar\pi(F^{-n_k}y_3))}
 \le |b_k|\theta^{n_k}
\le\theta^{-1}|b_k|^{-\alpha}.
\end{align*}
The same estimate holds with $y_1$ and $y_3$ replaced by $y_2$ and $y_4$, so
\[
|\exp\{ib_kD(y_1,y_4)\}-1|
=O(|b_k|^{-\alpha})
\;\text{for all $y_1,y_4\in Z_0$, $k\ge1$}.
\]
This means that there is a constant $C>0$ such that
\[
D(Z_0\times Z_0)\subset \bigcup_{j\in\Z}\Big(
\frac{2\pi j}{b_k}-\frac{C}{|b_k|^{\alpha+1}}\,,\, \frac{2\pi j}{b_k}+\frac{C}{|b_k|^{\alpha+1}}\Big)
\quad\text{for all $k\ge1$}.
\]
Hence 
$\underline{\mathrm{BD}}(D(Z_0\times Z_0))\le 1/(\alpha+1)$.
The result follows since $\alpha$ is arbitrarily large.
\end{proof}

\begin{rmk} \label{rmk:contact}
For Axiom~A attractors, $Z_0$ can be taken to be connected and $D$ is continuous, so absence of approximate eigenfunctions is ensured whenever
$D$ is not identically zero~\cite[Section~9]{Dolgopyat98a}.  
For nonuniformly hyperbolic flows, where the partition $\{Y_j\}$ is countably infinite, $Z_0$ is a Cantor set of positive Hausdorff dimension~\cite[Example~5.7]{M09}.  In general it is not clear how to use this property since $D$ is generally at best H\"older.   However for flows with a contact structure, a formula for $D$ in~\cite[Lemma~3.2]{KatokBurns94} can be exploited, see~\cite[Example~5.7]{M09}.
Hence when there is a contact structure (which includes the Lorentz gas flows in the introduction when there is no external forcing) absence of approximate eigenfunctions is automatic.
\end{rmk}

\part{Proof of theorems for semiflows}

In this part of the review article, we give a self-contained proof of the main results on rapid and polynomial mixing for semiflows.
In Section~\ref{sec:strat}, we recall how to use the Laplace transform of the correlation function to deduce rates of mixing.
Theorems~\ref{thm:rapid} (rapid mixing) and Theorem~\ref{thm:slow} (polynomial mixing) are proved in Sections~\ref{sec:rapid} and~\ref{sec:slow} respectively.

 We note that the proof of Theorem~\ref{thm:rapid} plays a crucial role in the proof of Theorem~\ref{thm:slow}, justifying the movement of certain contours of integration to the imaginary axis in Section~\ref{sec:trunc}.
 Theorem~\ref{thm:rapid} also illustrates many of the main ideas while avoiding numerous technical issues needed for Theorem~\ref{thm:slow}.

\section{Strategy for obtaining rates of mixing}
\label{sec:strat}

In this section, we recall how to use smoothness of the Laplace transform $\hat\rho_{v,w}$ to deduce decay rates for $\rho_{v,w}$.
Basic facts about inversion of Laplace transforms and decay rates are recalled in Subsection~\ref{sec:Laplace}.
In Subsection~\ref{sec:poll}, we prove a version of Pollicott's formula~\cite{Pollicott85} for $\hat\rho_{v,w}$.

\subsection{Laplace transforms and inversion formulas}
\label{sec:Laplace}

Define $\H=\{s\in\C:\Re s>0\}$ and $\barH=\{s\in\C:\Re s\ge0\}$.

The Laplace transform $\hat\rho_{v,w}(s)=\int_0^\infty e^{-st}\rho_{v,w}(t)\,dt$ is analytic on $\H$.
We deduce decay of $\rho_{v,w}$ from smoothness of $\hat\rho_{v,w}$.  It is convenient to make a $C^\infty$ modification so that $\rho_{v,w}$ is
unchanged for $t\ge1$ but vanishes near zero.   (Such  a modification does not affect the asymptotics of $\rho_{v,w}$ nor the smoothness of $\hat\rho_{v,w}$.)
Abusing notation, we still write $\rho_{v,w}$ and $\hat\rho_{v,w}$ for these modified functions.

Let $\cA$ be a subspace of $L^\infty(Y^\varphi)$ with norm ${\|\;\|}_\cA\ge|\;|_\infty$.
For $m\ge0$, define 
$\cA_m=\{w\in\cA:{\|w\|}_{\cA,m}<\infty\}$
to be the subspace of observables that are $m$-times differentiable in
the flow direction, 
where
${\|w\|}_{\cA,m}=\sum_{j=0}^m{\|\partial_t^jw\|}_\cA$.
Note that $\cA_m\subset L^{\infty,m}(Y^\varphi)=(L^\infty(Y^\varphi))_m$.

If $v\in L^1(Y^\varphi)$ and $w\in L^{\infty,m}(Y^\varphi)$, then $\rho_{v,w}$ is $m$-times differentiable and
$\rho_{v,w}^{(m)}=\rho_{v,\partial_t^mw}$.
Hence, performing integration by parts $m$ times, we obtain
\begin{align} \label{eq:Taylor}
\hat\rho_{v,w}(s)=s^{-m} \hat\rho_{v,\partial_t^mw}(s)\quad\text{for $s\in\H$}.
\end{align}

\begin{cor} \label{cor:Taylor}
Let $v\in L^1(Y^\varphi)$, $w\in L^{\infty,2}(Y^\varphi)$.
Then
\[
\SMALL \rho_{v,w}(t)=\int_{-\infty}^{\infty} e^{(\eps+ib)t}\hat\rho_{v,w}(\eps+ib)\,db
\quad\text{for any $\eps>0$}.
\]
\end{cor}

\begin{proof}
Note that $|\hat\rho_{v,w}(a+ib)|\le a^{-1}|\rho_{v,w}|_\infty\le
2a^{-1}|v|_1|w|_\infty$ for all $s=a+ib\in\H$.
Also, by~\eqref{eq:Taylor},
$|\hat\rho_{v,w}(a+ib)|=|s|^{-2}|\hat\rho_{v,\partial_t^2w}(a+ib)|\le 2a^{-1}|s|^{-2}|v|_1|\partial_t^2w|_\infty$.
Combining these, 
\[
|\hat\rho_{v,w}(s)|\le 4a^{-1}(|s|^2+1)^{-1}|v|_1|w|_{\infty,2}
\quad\text{for all $s=a+ib\in\H$. }
\]
Hence, we can apply the classical inverse Laplace transform formula.~
\end{proof}

\begin{lemma} \label{lem:strat}
Let $v\in L^1(Y^\varphi)$, $\eps>0$, $q\geq1$.
Suppose that 
\begin{itemize}
\item[(i)]
$s\mapsto \hat\rho_{v,w}(s)$ is continuous on $\{\Re s\in[0,\eps]\}$ and
$b\mapsto \hat\rho_{v,w}(ib)$ is $C^q$ on $\R$
for all $w\in \cA$.
\item[(ii)]
There exist constants
$C,\alpha>0$ such that
\[
|\hat\rho_{v,w}(s)|\le C(|b|+1)^{\alpha}|w|_\infty \quad\text{and}\quad
|\hat\rho_{v,w}^{(j)}(ib)|\le C(|b|+1)^{\alpha}{\|w\|}_\cA,
\]
for all $w\in \cA$, $j\le q$,
and all $s=a+ib\in\C$ with $a\in[0,\eps]$.
\end{itemize}
Let $m=\lceil\alpha\rceil+2$.  Then
there exists a constant $C'>0$ depending only on $C$, $q$, $\alpha$, such that
\[
|\rho_{v,w}(t)|\le C'{\|w\|}_{\cA,m}\, t^{-q}\quad\text{for all $w\in \cA_m$, $t>1$}.
\]
\end{lemma}

\begin{proof}
By Corollary~\ref{cor:Taylor}, 
$\rho_{v,w}(t)=\int_{-\infty}^{\infty} e^{(\eps+ib)t}\hat\rho_{v,w}(\eps+ib)\,db$.
By~\eqref{eq:Taylor} and condition~(ii),
$\hat\rho_{v,w}(a+ib)=O(|b|^{-2})$ uniformly in $a\in[0,\eps]$ as $b\to\infty$.
Using this together with the continuity property in~(i) and a standard extension of the Cauchy integral theorem,
we can move the contour of integration to the imaginary axis, so
\[
\SMALL \rho_{v,w}(t)=\int_{-\infty}^\infty e^{ibt}\hat\rho_{v,w}(ib)\,db.
\]

Again, by~\eqref{eq:Taylor} and conditions~(i,ii), $\hat\rho_{v,w}^{(j)}(ib)$ is integrable for all $j\le q$, so the result follows
by standard Fourier analysis.  (See also Proposition~\ref{prop:fourier} below.) 
\end{proof}

\subsection{Formula for $\hat\rho_{v,w}$}
\label{sec:poll}

We use the following variant of Pollicott's formula~\cite{Pollicott85}.
Since $\rho_{v,w}(t)=\int_{Y^\varphi}(v-\int_{Y^\varphi}v\,d\mu)\,w\circ F_t\,d\mu^\varphi$, we may suppose throughout that $\int_{Y^\varphi}v\,d\mu^\varphi=0$.

Given observables $v,w:Y^\varphi\to\R$, define
\begin{align*}
V(t)(y) & = v(y,\varphi(y)-t)=1_{\{0\le t\le \varphi(y)\}}v(y,\varphi(y)-t), \\
w(t)(y) & =w(y,t)=1_{\{0\le t\le \varphi(y)\}}w(y,t).
\end{align*}
The corresponding Laplace transforms
$\hV(s),\,\hw(s):Y\to\C$, $s\in\H$, are given by
\[
\hV(s)(y)=\int_0^{\varphi(y)} e^{-s(\varphi(y)-u)}v(y,u)\,du, \qquad
\hw(s)(y)=\int_0^{\varphi(y)} e^{-su}w(y,u)\,du.
\]
Also, define $v_s(y)=\int_0^{\varphi(y)}e^{su}v(y,u)\,du$.

Let $R:L^1(Y)\to L^1(Y)$ denote the transfer operator corresponding to 
$F:Y\to Y$.
So $\int_Y v\,w\circ F\,d\mu=\int_Y Rv\,w\,d\mu$ for all $v\in L^1(Y)$
and $w\in L^\infty(Y)$.

For $s\in\barH$, define the {\em twisted transfer operators}
\[
\hR(s):L^1(Y)\to L^1(Y),\qquad \hR(s)v=R(e^{-s\varphi}v).
\]

\begin{prop} \label{prop:poll}
Let $v,\,w\in L^\infty(Y^\varphi)$ with $\int_{Y^\varphi}v\,d\mu^\varphi=0$.  Then
$\rho_{v,w}=\sum_{n=0}^\infty J_n$
where
$|J_0(t)|\le |v|_\infty|w|_\infty \intphi^{-1}\int_Y 1_{\{\varphi>t\}}\varphi\,d\mu$ for all $t>0$ and
\begin{align*}
\hJ_n(s) & = 
 \intphi^{-1}{\SMALL\int}_Y e^{-s\varphi_n}v_s \,\hw(s)\circ F^n\,d\mu
\qquad \text{for all $s\in\barH$, $n\ge1$.}
\end{align*}
\end{prop}

\begin{proof}
Write
\begin{align*}
\rho(t) & 
 =\int_{Y^\varphi} 1_{\{t+u<\varphi(y)\}}v(y,u)\,w\circ F_t(y,u)\,d\mu^\varphi
\\ & \qquad \qquad +\sum_{n=1}^\infty\int_{Y^\varphi} 1_{\{\varphi_n(y)<t+u<\varphi_{n+1}(y)\}}v(y,u)\,w\circ F_t(y,u)\,d\mu^\varphi
= \sum_{n=0}^\infty J_n(t),
\end{align*}
where
\[
J_n(t)=\begin{cases} \int_{Y^\varphi} 1_{\{t+u<\varphi(y)\}}v(y,u)w(y,t+u)\,d\mu^\varphi, & n=0 \\
\int_{Y^\varphi} 1_{\{\varphi_n(y)<t+u<\varphi_{n+1}(y)\}}v(y,u)w(F^ny,t+u-\varphi_n(y))\,d\mu^\varphi, & n\ge1\end{cases}\,.
\]
In particular,
\[
|J_0(t)|\le |v|_\infty|w|_\infty \int_{Y^\varphi} 1_{\{\varphi>t\}}\,d\mu^\varphi
=|v|_\infty|w|_\infty \intphi^{-1}\int_Y 1_{\{\varphi>t\}}\varphi\,d\mu.
\]

For $n\ge1$,
\begin{align*}
\hJ_n(s) & =\int_0^\infty e^{-st}J_n(t)\,dt
\\ & =\intphi^{-1}\int_Y\int_0^{\varphi(y)}\int_{\varphi_n(y)-u}^{\varphi_{n+1}(y)-u}e^{-st}v(y,u)w(F^ny,t+u-\varphi_n(y))\,dt\,du\,d\mu.
\end{align*}
Making the substitution $u'=t+u-\varphi_n(y)$,
\begin{align*}
\hJ_n(s) & =\intphi^{-1}\int_Y\Bigl(\int_0^{\varphi(y)}e^{su}v(y,u)\,du\Bigr)
\Bigl(\int_0^{\varphi(F^ny)}e^{-su'}w(F^ny,u')\,du'\Bigr)e^{-s\varphi_n(y)}\,d\mu
\\ & = \intphi^{-1}\int_Y e^{-s\varphi_n}v_s \,\hw(s)\circ F^n\,d\mu,
\end{align*}
as required.
\end{proof}

\begin{cor} \label{cor:poll}
Let $v\in L^1(Y^\varphi)$ and $w\in L^\infty(Y^\varphi)$ with $\int_{Y^\varphi}v\,d\mu^\varphi=0$.  Then
\begin{align*}
\SMALL \hat\rho_{v,w}(s)=\hJ_0(s)+\intphi^{-1}\int_Y (I-\hR(s))^{-1}R\hV(s)\, \hw(s)\,d\mu\quad\text{for all $s\in\H$}.
\end{align*}
\end{cor}

\begin{proof}
Note that $e^{-s\varphi}v_s=\hV(s)$.  Hence
\begin{align*}
R^n(e^{-s\varphi_n}v_s)=R^{n-1}R(e^{-s\varphi_{n-1}\circ F}e^{-s\varphi}v_s)
=R^{n-1}(e^{-s\varphi_{n-1}}R(e^{-s\varphi}v_s))
=\hR(s)^{n-1}R\hV(s),
\end{align*}
and so
\begin{align*}
\hJ_n(s)
& =\intphi^{-1}\int_Y e^{-s\varphi_n}v_s \,\hw(s)\circ F^n\,d\mu
\\ & = \intphi^{-1}\int_Y R^n(e^{-s\varphi_n}v_s) \,\hw(s)\,d\mu
= \intphi^{-1}\int_Y \hR(s)^{n-1}R\hV(s) \,\hw(s)\,d\mu.
\end{align*}
The result now follows from Proposition~\ref{prop:poll}.
\end{proof}

\begin{rmk}   \label{rmk:poll}
In contrast with the formula in~\cite{Pollicott85}, we do not write
$\hat v(s)(y)=\int_0^{\varphi(y)}e^{su}v(y,u)\,du$ since such a function is not analytic on $\H$.
We reserve the hat notation for functions that are analytic on
$\H$, such as $\hw$ and $\hV$.
\end{rmk}

The following estimate is immediate from the definition of $w(t)$.
\begin{prop} \label{prop:w}
$|w(t)|_1\le |w|_\infty\,\mu(\varphi>t)$ for all $w\in L^\infty(Y^\varphi)$,
$t>0$.\qed
\end{prop}

\section{Proof of rapid mixing for semiflows}
\label{sec:rapid}

In this section, we consider Gibbs-Markov semiflows $F_t:Y^\varphi\to Y^\varphi$ for which 
the roof function $\varphi:Y\to\R^+$ lies in $L^q(Y)$ for all $q\ge1$.
For such semiflows, we prove Theorem~\ref{thm:rapid}, namely that absence of approximate eigenfunctions is a sufficient condition for rapid mixing.

Subsection~\ref{sec:back} contains background material on Gibbs-Markov maps.
The only results we assume without proof are Propositions~\ref{prop:GM} and~\ref{prop:qc} below.
Subsections~\ref{sec:V} and~\ref{sec:R} establish smoothness of the expressions $R\hV$ and $\hR$ that arose in Subsection~\ref{sec:poll}.
The key estimate, Dolgopyat's estimate, is established in Subsection~\ref{sec:Dolg} and is used to establish the smoothness of $\hT=(I-\hR)^{-1}$
in Subsection~\ref{sec:T} and thereby $\hat\rho_{v,w}$ in Subsection~\ref{sec:pf},
completing the proof of Theorem~\ref{thm:rapid}.

Throughout this section, $q\in\N=\{0,1,2\ldots\}$.
Also we take $\cA=L^\infty(Y^\varphi)$ in Lemma~\ref{lem:strat}.
If $\cA_1$ and $\cA_2$ are normed vector spaces, we denote by $\cB(\cA_1,\cA_2)$ the space of bounded linear operators from $\cA_1$ to $\cA_2$,
and we write $\cB(\cA_1)$ instead of $\cB(A_1,\cA_1)$.

\subsection{Background on Gibbs-Markov maps}
\label{sec:back}

Let $F:Y\to Y$ be a (full branch) Gibbs-Markov map as in Section~\ref{sec:GM} with partition $\{Y_j\}$.
For each $n\ge1$, let $\cD_n$ denote the partition of $Y$ into $n$-cylinders
$d=\bigcap_{i=1,\dots,n}F^{-i}Y_{j_i}$, where $j_1,\dots,j_n$ range over positive integers.
Define $p_n=\sum_{i=0}^{n-1}p\circ F^i$.

\begin{prop} \label{prop:GM}  
There is a constant $C_2\ge1$ such that
\begin{align} \label{eq:GM}
C_2^{-1}\mu(d)\le e^{p_n(y)}\le C_2\mu(d), \quad
|e^{p_n}(y)-e^{p_n}(y')|\le C_2\mu(d)d_\theta(F^ny,F^ny'),
\end{align}
for all $y,y'\in d$, $d\in\cD_n$, $n\ge1$.
\end{prop}

\begin{proof}  This is standard, see for example~\cite{Aaronson,AaronsonDenker01}.
\end{proof}

The transfer operator corresponding to a Gibbs-Markov map $F$ has the pointwise expression
$(Rv)(y)=\sum_{j\ge1}e^{p(y_j)}v(y_j)$,
where $y_j$ is the unique preimage of $y$ in $Y_j$.
More generally, $(R^nv)(y)=\sum_{d\in\cD_n}e^{p_n(y_d)}v(y_d)$, $n\ge1$,
where $y_d$ is the unique $F^n$-preimage of $y$ in $d$.

\begin{cor} \label{cor:GM}
Let $v\in \cF_\theta(Y)$.  
Then 
\[
\SMALL \|R^nv\|_\theta\le C_2\sum_{d\in\cD_n}\mu(d)(2|1_dv|_\infty+\theta^n|1_dv|_\theta) \quad\text{for $n\ge1$}.
\]
   In particular
$\|Rv\|_\theta\le 2C_2\sum_{j\ge1}\mu(Y_j)\|1_{Y_j}v\|_\theta$.
\end{cor}

\begin{proof}
For $y,y'\in Y$, it follows from~\eqref{eq:GM} that
\begin{align*}
|(R^nv)(y)- & (R^nv)(y')|  \le \sumd |e^{p_n(y_d)}-e^{p_n(y_d')}||v(y_d)|
+\sumd e^{p_n(y_d')}|v(y_d)-v(y_d')| \\
& \le C_2\sumd \mu(d)d_\theta(F^ny_d,F^ny_d') |1_dv|_\infty
+C_2\sumd \mu(d)|1_dv|_\theta d_\theta(y_d,y_d')
\\ & = C_2\sumd \mu(d)d_\theta(y,y') |1_dv|_\infty
+C_2\sumd \mu(d)|1_dv|_\theta\theta^n d_\theta(y,y').
\end{align*}
Hence
$|R^nv|_\theta\le C_2\sumd \mu(d)(|1_dv|_\infty + \theta^n|1_dv|_\theta)$.
A simpler argument shows that
$|R^nv|_\infty\le C_2\sumd \mu(d)|1_dv|_\infty$.
\end{proof}

\begin{rmk}   \label{rmk:GM}
Often we consider operators $S:\cF_\theta(Y)\to \cF_\theta(Y)$ of the form
$Sv=R(gv)$ for some fixed $g\in \cF_\theta(Y)$.  Since $\|gv\|_\theta\le \|g\|_\theta\|v\|_\theta$ it follows from Corollary~\ref{cor:GM} that
$\|S\|_\theta\le 
2C_2\sum_{j\ge1}\mu(Y_j)\|1_{Y_j}g\|_\theta$.
\end{rmk}

\begin{prop} \label{prop:qc}
The operator $R:\cF_\theta(Y)\to \cF_\theta(Y)$ has spectral radius $1$ and essential spectral radius $\theta$.  There is a simple eigenvalue at $1$ with eigenfunction $1$, and no other eigenvalues on the unit circle.
Moreover, there are constants $C_3\ge1$, $\gamma_1\in(0,1)$ such that
\[
\|R^nv-{\SMALL\int_Y}v\,d\mu\|_\theta\le C_3\gamma_1^n\|v\|_\theta 
\quad\text{for all $n\ge1$, $v\in \cF_\theta(Y)$.}
\]
\end{prop}

\begin{proof}  See for example~\cite[Section~4.7]{Aaronson},~\cite{AaronsonDenker01}.  
\end{proof}

\subsection{Smoothness of $R\hV:\cF_\theta(Y^\varphi)\to \cF_\theta(Y)$}
\label{sec:V}

Recall that
$\hV(s)(y)=\int_0^{\varphi(y)}e^{-s(\varphi(y)-u)}v(y,u)\,du$.
In this subsection, we obtain estimates for $R\hV$ as a function of $s$ and the observable $v\in \cF_\theta(Y^\varphi)$.  For convenience of notation, the dependence on $v$ is kept implicit.

\begin{prop} \label{prop:V}
$R\hV:\barH\to\cB(\cF_\theta(Y^\varphi),\cF_\theta(Y))$ is $C^\infty$ 
and \footnote{Throughout, when we estimate derivatives on $\barH$, it is to be understood that the estimates for $q=0$ hold on $\barH$ and the estimates for $q>0$ hold on the imaginary axis as required in Lemma~\ref{lem:strat}.}

\[
\SMALL 
\|R\hV^{(q)}(s)\|_\theta  \le (2C_1)^{q+4}C_2(|s|+q+1)\int_Y \varphi^{q+2}\,d\mu \,\|v\|_\theta,
\]
  for all $q\in\N$, $s\in\barH$, $v\in \cF_\theta(Y^\varphi)$.
\end{prop}

\begin{proof}
We have
$\hV^{(q)}(s)(y)=(-1)^q\int_0^{\varphi(y)}e^{-s(\varphi(y)-u)}(\varphi(y)-u)^q v(y,u)\,du$.
In particular, 
$|1_{Y_j}\hV^{(q)}(s)|_\infty\le |1_{Y_j}\varphi|_\infty^{q+1}|v|_\infty$.

Next, let $y,y'\in Y_j$ with $\varphi(y)\ge \varphi(y')$.  Then 
\[
\hV^{(q)}(s)(y)-\hV^{(q)}(s)(y')=(-1)^q(I_1+I_2+I_3+I_4),
\]
where 
\begin{align*}
I_1 & =\int_{\varphi(y')}^{\varphi(y)}e^{-s(\varphi(y)-u)}(\varphi(y)-u)^qv(y,u)\,du, \\
I_2 & =\int_0^{\varphi(y')}\{e^{-s\varphi(y)}-e^{-s\varphi(y')}\}e^{-su}(\varphi(y)-u)^qv(y,u)\,du,  \\
I_3 & =\int_0^{\varphi(y')}e^{-s(\varphi(y')-u)}\{(\varphi(y)-u)^q-(\varphi(y')-u)^q\}v(y,u)\,du,  \\
I_4 & =\int_0^{\varphi(y')}e^{-s(\varphi(y')-u)}(\varphi(y')-u)^q\{v(y,u)-v(y',u)\}\,du. 
\end{align*}
Since $v\in \cF_\theta(Y^\varphi)$, we obtain
\begin{align*}
|I_1| & \le (\varphi(y)-\varphi(y'))|1_{Y_j}\varphi|_\infty^q|v|_\infty, \quad & 
|I_2| & \le |s|(\varphi(y)-\varphi(y'))|1_{Y_j}\varphi|_\infty^{q+1}|v|_\infty, \\
|I_3| & \le q(\varphi(y)-\varphi(y'))|1_{Y_j}\varphi|_\infty^q|v|_\infty, \quad & 
|I_4| & \le |1_{Y_j}\varphi|_\infty^{q+1}{\SMALL\sup_u}|v(y,u)-v(y',u)|.
\end{align*}
Also, we have the estimates $|\varphi(y)-\varphi(y')|\le |1_{Y_j}\varphi|_\theta d_\theta(y,y')
\le C_1|1_{Y_j}\varphi|_\infty d_\theta(y,y')$ and
$|v(y,u)-v(y',u)|\le |1_{Y_j}\varphi|_\infty|v|_\theta\, d_\theta(y,y')$,
 so
\[
|1_{Y_j}\hV^{(q)}(s)|_\theta  
\le C_1(|s|+q+1)|1_{Y_j}\varphi|_\infty^{q+2}\|v\|_\theta.
\]

Hence
\[
\|1_{Y_j}\hV^{(q)}(s)\|_\theta  
\le C_1(|s|+q+2)|1_{Y_j}\varphi|_\infty^{q+2}\|v\|_\theta
\le (2C_1)^{q+3}(|s|+q+1)\infYj\varphi^{q+2}\|v\|_\theta.
\]
The result follows from Corollary~\ref{cor:GM}.
\end{proof}

\subsection{Smoothness of $\hR:\cF_\theta(Y)\to \cF_\theta(Y)$ and spectral properties}
\label{sec:R}

\begin{prop} \label{prop:R}
$\hR:\barH\to\cB(\cF_\theta(Y))$ is $C^\infty$ and
\[
\SMALL \|\hR^{(q)}(s)\|_\theta \le
C_2(2C_1)^{q+3}(|s|+q+1)\int_Y \varphi^{q+1}\,d\mu
\quad\text{for all $q\in\N$, $s\in\barH$.}
\]
\end{prop}

\begin{proof}
First note that $\hR^{(q)}(s)v=(-1)^qR(f_q(s)v)$, where
$f_q(s)=e^{-s\varphi}\varphi^q$.

It is immediate that 
$|1_{Y_j}f_q(s)|_\infty\le |1_{Y_j}\varphi|_\infty^q$.
Next, let $y,y'\in Y_j$.  Then 
$f_q(s)(y)-f_q(s)(y')=I_1+I_2$ where
\begin{align*}
I_1 & = \{e^{-s\varphi(y)}-e^{-s\varphi(y')}\}\varphi(y)^q, \qquad
I_2  = e^{-s\varphi(y')}\{\varphi(y)^q-\varphi(y')^q\}.
\end{align*}
These terms contribute
\[
|I_1|\le |s|C_1|1_{Y_j}\varphi|_\infty^{q+1}d_\theta(y,y'), \qquad
|I_2| \le  qC_1|1_{Y_j}\varphi|_\infty^qd_\theta(y,y').
\]
It follows that
$|1_{Y_j}f_q(s)|_\theta\le C_1(|s|+q)|1_{Y_j}\varphi|_\infty^{q+1}$ and 
so 
\[
\|1_{Y_j}f_q(s)\|_\theta\le C_1(|s|+q+1)|1_{Y_j}\varphi|_\infty^{q+1}\le
(2C_1)^{q+2}(|s|+q+1)\infYj\varphi^{q+1}.
\]
By Remark~\ref{rmk:GM},
$\|\hR^{(q)}(s)\|_\theta\le 2C_2\sum\mu(Y_j)\|1_{Y_j}f_q(s)\|_\theta
\le C_2(2C_1)^{q+3}(|s|+q+1) \sum\mu(Y_j)\infYj\varphi^{q+1}
\le C_2(2C_1)^{q+3}(|s|+q+1) \int_Y\varphi^{q+1}\,d\mu$.
\end{proof}

\begin{prop}  \label{prop:basic}
\begin{itemize}
\item[(a)] $|\hR(s)v|_p\le |v|_p$ for all $s\in\barH$, $v\in L^p(Y)$, $1\le p\le\infty$.
\item[(b)] $|\hR(s)^nv|_\theta\le C_4\{(|s|+1)|v|_\infty+\theta^n|v|_\theta\}$ for all $s\in\barH$, $n\ge1$, $v\in \cF_\theta(Y)$.
\end{itemize}
\end{prop}

\begin{proof}
Part~(a) is immediate.
For part~(b), see~\cite[Corollary~4.3(a)]{BHM05} where it is shown that
$|\hR(ib)^nv|_\theta\le \{C_2^2+|b|C_2\theta(1-\theta)^{-1}\sum_{j\ge1}\mu(Y_j)|1_{Y_j}\varphi|_\theta\}|v|_\infty + C_2\theta^n|v|_\theta$.
The proof given there extends immediately to $s\in\barH$ as follows.

Write $\hR(s)^nv=R^n(f_n(s)v)$ where $f_n(s)=e^{-s\varphi_n}$.
Fix an $n$-cylinder $d\in\cD_n$ and let $y,y'\in d$.   Then
\begin{align*}
|f_n(s)(y)-f_n(s)(y')|
& \le |s|\sum_{j=0}^{n-1}|\varphi(F^jy)-\varphi(F^jy')|
\le |s|\sum_{j=0}^{n-1}|1_{F^jd}\varphi|_{\theta}d_{\theta}(F^jy,F^jy')
\\ &
\le C_1|s|\sum_{j=0}^{n-1}(\infFjd\varphi)\theta^{-j}d_\theta(y,y').
\end{align*}
Hence $|1_df_n(s)|_\infty\le1$ and
$|1_df_n(s)|_\theta\le C_1|s|\sum_{j=0}^{n-1}(\infFjd\varphi)\theta^{-j}$.
It follows that 
$|1_df_n(s)v|_\infty\le |v|_\infty$ and
\[
|1_df_n(s)v|_\theta\le 
|1_df_n(s)|_\theta|v|_\infty+
|1_df_n(s)|_\infty|v|_\theta
\le C_1|s|\sum_{j=0}^{n-1}(\infFjd\varphi)\theta^{-j}|v|_\infty+|v|_\theta.
\]
Hence by Corollary~\ref{cor:GM},
\begin{align*}
\|\hR & (s)^n v\|_\theta  \le C_2\sumd\mu(d)(2|1_df_n(s)v|_\infty+\theta^n|1_df_n(s)v|_\theta)
\\
& \le C_2\sumd\mu(d)(2|v|_\infty+\theta^n|v|_\theta)
+ C_2C_1|s||v|_\infty K
 =(2C_2+C_2C_1|s|K)|v|_\infty+C_2\theta^n|v|_\theta,
\end{align*}
where
\begin{align*}
K & =\sum_{d\in\cD_n}\mu(d)\sum_{j=0}^{n-1}\theta^{n-j}\infFjd\varphi
 =\sum_{d\in\cD_n}\mu(d)\sum_{j=0}^{n-1}\theta^{n-j}\infd\varphi\circ F^j
 \\ & \le\int_Y\sum_{j=0}^{n-1}\theta^{n-j}\varphi\circ F^j\,d\mu
 = \sum_{j=0}^{n-1}\theta^{n-j}\int_Y\varphi\,d\mu
\le \theta(1-\theta)^{-1} \int_Y\varphi\,d\mu.
\end{align*}
This completes the proof of part (b).
\end{proof}

\begin{prop} \label{prop:radius}
\begin{itemize}
\item[(a)] $\hR(ib)$ has spectral radius at most $1$ and essential spectral radius at most $\theta$ for all $b\in\R$.
\item[(b)] The spectral radius of $\hR(s)$ is less than $1$ for all $s\in\H$.
\item[(c)] $\hR(ib)v=\lambda v$ for some $\lambda\in\C$ with $|\lambda|=1$, $v\in L^2(Y)$, $b\in\R$ if and only if $e^{ib\varphi}v\circ F=\bar\lambda v$.
\end{itemize}
\end{prop}

\begin{proof}
(a) 
Since $\cF_\theta(Y)$ is compactly embedded in $L^\infty(Y)$, 
the estimate on the essential spectral radius follows from Proposition~\ref{prop:basic}(a,b), see for example~\cite{Hennion93}.
Now apply Proposition~\ref{prop:basic}(a).
\\
(b) Again the essential spectral radius is strictly less than $1$
by Proposition~\ref{prop:basic}(a,b).
Also, if $\hR(s)v=\lambda v$ for some $v\in \cF_\theta(Y)$, $\lambda\in\C$,
$|\lambda|=1$,
then $\int_Y|v|\,d\mu=\int_Y |R(e^{-s\varphi}v)|\,d\mu\le \int_Y e^{-a\varphi}|v|\,d\mu$
where $a=\Re s>0$.  But $\varphi>0$ so $v=0$.
\\
(c) Recall that $M_bv=e^{ib\varphi}v\circ F$.  Note that $\hR(ib)M_b=I$ and that $M_b$ and $\hR(ib)$ are $L^2$ adjoints, i.e.\
$\langle \hR(ib)v,w\rangle=\int_Y \hR(ib)v\,\bar w\,d\mu
=\int_Y v\,\overline{M_bw}\,d\mu=\langle v,M_bw\rangle$.
If $M_bv=\bar\lambda v$, then $v=\hR(ib)M_bv=\bar\lambda\hR(ib)v$
so $\hR(ib)v=\lambda v$.
Conversely, if $\hR(ib)v=\lambda v$, then 
\begin{align*}
\langle M_bv-\bar\lambda v,  M_bv-\bar\lambda v\rangle  & =
\langle M_bv,M_bv\rangle
-\langle M_bv,\bar\lambda v\rangle
-\langle \bar\lambda v,M_bv\rangle
+\langle \bar\lambda v,\bar\lambda v\rangle
\\ & =\langle v,v\rangle
-\langle v,v\rangle
-\langle v,v\rangle
+\langle v,v\rangle= 0,
\end{align*}
 so $M_bv=\bar\lambda v$.
\end{proof}

\begin{cor} \label{cor:R}
There exists $\delta>0$ and a $C^\infty$ family of simple eigenvalues
$\lambda:\barH\cap B_\delta(0)\to\C$ such that $\lambda(0)=1$ and $\lambda(s)$ is isolated in $\spec\hR(s):\cF_\theta(Y)\to \cF_\theta(Y)$.
Moreover, $\lambda'(0)=-\intphi$.
The corresponding family of spectral projections $P(s)$ is $C^\infty$
with $P(0)v=\int_Yv\,d\mu$ for all $v\in \cF_\theta(Y)$.
\end{cor}

\begin{proof}  
The existence of the $C^\infty$ families $\lambda(s)$ and $P(s)$ follows from Propositions~\ref{prop:qc} and~\ref{prop:R}.  
Differentiating $\hR P=\lambda P$
 and applying $P$ to both sides,
\[
P\hR'P+P\hR P'=\lambda P P'+\lambda'P.
\]
Since $P\hR=\hR P=\lambda P$, this reduces to
$P\hR'P=\lambda'P$.
In particular,
$\lambda'(0)=P(0)\hR'(0)1=\int_Y\hR'(0)1\,d\mu=-\int_Y R\varphi\,d\mu=-\intphi$.
\end{proof}

\subsection{Dolgopyat estimate}
\label{sec:Dolg}

In this subsection, we prove the following key estimate.

\begin{thm} \label{thm:Dolg}
Assume absence of approximate eigenfunctions.
Then 
\begin{itemize}

\parskip=-2pt
\item[(a)] The spectral radius of $\hR(s)$ is less than $1$ for all $s\in\barH\setminus\{0\}$.
\item[(b)]
For any $\delta>0$,
there exist $\alpha,C>0$ such that
\[
\|(I-\hR(s))^{-1}\|_\theta \le C|b|^\alpha\quad\text{
for all $s=a+ib\in\C$ with $0\le a\le 1$, $|b|\ge \delta$.}
\]
\end{itemize}
\end{thm}

Define the scale of equivalent norms
\[
\|v\|_b=\max\Big\{|v|_\infty,\,\frac{|v|_\theta}{2C_4(|b|+2)}\Big\}, \quad b\in\R.
\]
By Proposition~\ref{prop:basic}(a,b), for all $n\ge1$,
\[
\|\hR(s)^n\|_b \le C_4+{\textstyle\frac12}\quad\text{for all $s=a+ib\in\C$ with $a\in[0,1]$}.
\]

For each $b$, define the unit ball $\cF_\theta(Y)_b=\{v\in \cF_\theta:\|v\|_b\le1\}$.
Let $C_5=2C_4(C_4+\frac12)$.  Then for all $v\in \cF_\theta(Y)_b$, $n\ge1$, 
\[
|\hR(s)^nv|_\infty\le 1\quad\text{and}\quad
|\hR(s)^nv|_\theta\le C_5(|b|+2)
\quad\text{for $s=a+ib\in\C$ with $a\in[0,1]$}.
\]

Let $Z$ denote a fixed subset of $Y$ consisting of a finite union of
partition elements of $Y$, with finite subsystem $Z_0=\bigcap_{j\ge0}F^{-j} Z$.
Note that $p$ is uniformly bounded on $Z_0$ and moreover
$|p_n(y)|\le n|1_{Z_0}p|_\infty$ for all $y\in Z_0$ and $n\ge1$.

\begin{lemma}  \label{lem:Dolg1}
Fix $\alpha_2>0$.    Then there exists $\alpha_1>0$ and there exists $\xi>0$ arbitrarily large 
such that the following is true for each fixed $s=a+ib\in\C$ with $a\in[0,1]$,
setting $n(b)=[\xi\ln(|b|+2)]$:

Suppose that there exists $v_0\in \cF_\theta(Y)_b$ such that
for all $y\in Z_0$ and all $j=0,1,2$,
\[
|(\hR(s)^{jn(b)}v_0)(y)|\ge 1-(|b|+2)^{-\alpha_1}.
\]
Then there exists $w\in \cF_\theta(Y)$ with
$|w(y)|\equiv1$ and $|w|_\theta\le 8C_5|(b|+2)$, and there exists $\psi\in[0,2\pi)$ such that for all $y\in Z_0$,
\[
|(M_b^{n(b)}w)(y)-e^{i\psi}w(y)| \le 8(|b|+2)^{-\alpha_2}.
\]
\end{lemma}

\begin{proof}
We write $\hat b=|b|+2$ and $n=n(b)$.    
Choose $\xi$ such that
$8C_5\hat b\theta^n=8C_5\hat b\theta^{[\xi\ln\hat b]}\le \hat b^{-\alpha_2}$ for all $b$, and set
\[
\alpha_1=\max\{1,2\alpha_2+\xi|1_{Z_0}p|_\infty\}.
\]

Write $v_j=\hR(s)^{jn}v_0$ and $v_j=r_jw_j$, where $|w_j(y)|\equiv1$
and $1-\hat b^{-\alpha_1}\le r_j(y)\le 1$ for $y\in Z_0$.
Note that $|r_j|_\theta\le|v_j|_\theta\le C_5\hat b$, so 
\[
|w_j|_\theta
=|r_j^{-1}v_j|_\theta\le 2C_5\hat b (1-\hat b^{-\alpha_1})^{-2}\le 8C_5\hat b.
\]

Rearrange $v_1=\hR(s)^nv_0$ to obtain
$w_1^{-1}\hR(s)^n(v_0)=r_1\ge 1-\hat b^{-\alpha_1}$.
Hence $1-w_1^{-1}\hR(s)^n(v_0)\le \hat b^{-\alpha_1}$.
It follows that
\begin{align*}
& R^n(1-\Re\{e^{-ib\varphi_n}w_0\,w_1^{-1}\circ F^n\})
 \le 1-R^n(r_0e^{-a\varphi_n}\Re\{e^{-ib\varphi_n}w_0\,w_1^{-1}\circ F^n\})
\\ & \qquad\qquad\qquad  = 1-\Re R^n(e^{-s\varphi_n}v_0\,w_1^{-1}\circ F^n)
 = \Re\{1-w_1^{-1}\hR^n(s)v_0\}\le \hat b^{-\alpha_1}.
\end{align*}
Hence
\[
e^{p_n(y)}[1-\Re(e^{-ib \varphi_n(y)}w_0(y)w_1^{-1}(F^ny))]\le \hat b^{-\alpha_1}
\]
for all $y\in Y$ with $F^n y\in Z_0$.   It follows that
\[
|e^{-ib\varphi_n(y)}w_0(y)-w_1(F^n y)|\le 2(e^{-p_n(y)}\hat b^{-\alpha_1})^{1/2}.
\]
Similarly, with $w_0$ and $w_1$ replaced by $w_1$ and $w_2$.
Restricting to $y\in Z_0$, we have
$e^{-p_n(y)}\hat b^{-\alpha_1}\le \hat b^{-2\alpha_2}$ and hence
\begin{align} 
|e^{-ib\varphi_n(y)}w_0(y)-w_1(F^n y)| &\le 2\hat b^{-\alpha_2}, \qquad 
|e^{-ib\varphi_n(y)}w_1(y)-w_2(F^n y)| \le 2\hat b^{-\alpha_2},
\label{eq:w}
\end{align}
for all $y\in Z_0$.
Fix $z\in Z_0$ and choose $\psi_0,\psi_1\in\R$ such that $w_j(z)=e^{i\psi_j}$ for $j=0,1$ and such that $\psi=\psi_0-\psi_1\in[0,2\pi)$.
To each $y$, we associate the point $y^*\in Z_0$ with symbol sequence
$y^*=z_0\cdots z_{n-1}y_ny_{n+1}\cdots$ (recall from Definition~\ref{def:Z0} that $F|_{Z_0}$ is a full one-sided shift).
Then $y^*$ is within distance $\theta^n$ of $z$ and $F^n y^*=F^ny$.
We obtain
\begin{align*}
|e^{-ib\varphi_n(y^*)}e^{i\psi_0}-w_1(F^n y)| &\le 2\hat b^{-\alpha_2}+
8C_5\hat b \theta^n\le 3\hat b^{-\alpha_2} \\
|e^{-ib\varphi_n(y^*)}e^{i\psi_1}-w_2(F^n y)| &\le 2\hat b^{-\alpha_2}+
8C_5\hat b \theta^n\le 3\hat b^{-\alpha_2},
\end{align*}
by the choice of $\xi$.  Hence
$|e^{-i\psi}w_1(F^n y)-w_2(F^n y)| \le 6\hat b^{-\alpha_2}$.
Substituting into~\eqref{eq:w} yields the required approximate eigenfunction
$w=w_1$.
\end{proof}

\begin{lemma}   \label{lem:Dolg2}
Let $\alpha_1,\xi_1>0$.
Suppose that for any $v\in \cF_\theta(Y)_b$
there exists $y_0\in Z_0$ and $j\le [\xi_1\ln(|b|+2)]$
such that 
\[
|\hR(s)^j v(y_0)|\le 1-(|b|+2)^{-\alpha_1}
\quad\text{for all $s=a+ib\in\C$ with $a\in[0,1]$.}
\]
Then there exists $\alpha,\xi,\eps>0,\,C>0$ such that
\begin{itemize}
\item[(a)] 
$\|\hR(s)^{[\xi\ln(|b|+2)]}\|_b \le 1-\eps(|b|+2)^{-\alpha}$
for all $s=a+ib\in\C$ with $a\in[0,1]$.
\item[(b)]
$\|(I-\hR(s))^{-1}\|_\theta \le C(|b|+2)^\alpha$
for all $s=a+ib\in\C$ with $a\in[0,1]$.
\end{itemize}
\end{lemma}

\begin{proof}
We use the pointwise estimate on iterates
of $\hR(s)$ to obtain estimates on the $L^1$, $L^\infty$ and $\|\;\|_b$ norms.  Set $\hat b=|b|+2$.

Let $\hat u=\hR(s)^jv$ and note that $|\hat u|_\infty\le1$,
$|\hat u|_\theta\le C_5\hat b$.
Hence, $|\hat u(y)|\le 1-1/(2\hat b^{\alpha_1})$ for all $y$
within distance $1/(2C_5\hat b^{\alpha_1+1})$ of $y_0$.   Call this subset $U$.
If $d\in\cD_k$ is a $k$-cylinder, then
$\diam d=\theta^k$,
so provided $\theta^k<1/(2C_5\hat b^{\alpha_1+1})$, the $k$-cylinder containing
$y_0$ lies inside $U$.  It suffices to take
$k\approx (\alpha_1+1)\ln\hat b/(-\ln\theta)$.  By~\eqref{eq:GM},
\[
\mu(U)\ge\mu(d)\ge C_2^{-1} e^{-p_k(y_0)}\ge 
C_2^{-1} e^{-k|1_{Z_0}p|_\infty}\ge \eps_1 \hat b^{-(\alpha_1+1)\alpha_2},
\]
where $\alpha_2=|1_{Z_0}p|_\infty/(-\ln\theta)$.

Next, let $u=\hR(s)^{n_1(b)}v$ where $n_1(b)=[\xi_1\ln\hat b]$.  
Again, $|u|_\infty\le1$, $|u|_\theta\le C_5\hat b$.
Also, $n_1(b)\ge j$ so $|u|_1\le |\hat u|_1$ by Proposition~\ref{prop:basic}(a).
Breaking up $Y$ into $U$ and $Y\setminus U$,
\[
|u|_1\le |\hat u|_1\le 
(1-1/(2\hat b^{\alpha_1}))\mu(U)+1-\mu(U)=1-\mu(U)/(2\hat b^{\alpha_1})\le 
1-\eps_2\hat b^{-\alpha},
\]
where $\alpha=\alpha_1+\alpha_2+\alpha_1\alpha_2$.
By Proposition~\ref{prop:qc}, and using that $\|\,|u|\,\|_\theta\le \|u\|_\theta$,
\begin{align*}
|\hR(s)^n u|_\infty & \le |R^n(|u|)|_\infty \le 
|R^n(|u|)-{\textstyle \int}_Y|u|\,d\mu|_\infty+|u|_1 
 \le C_3\gamma_1^n\|u\|_\theta + |u|_1 \\
& \le (1+C_5\hat b)C_3\gamma_1^n+1-\eps_2\hat b^{-\alpha}.
\end{align*}
Choosing $n=n_2(b)=[\xi_2\ln\hat b]$ where $\xi_2\gg 1$ ensures that
\[
|\hR(s)^{n_1(b)+n_2(b)}v|_\infty =
|\hR(s)^{n_2(b)}u|_\infty \le 1-\eps\hat b^{-\alpha}.
\]
Setting $n(b)=[\xi\ln\hat b]$ where $\xi=\xi_1+\xi_2$,
\[
|\hR(s)^{n(b)}v|_\infty \le 1-\eps\hat b^{-\alpha}.
\]
By Proposition~\ref{prop:basic}(a,b), $|\hR(s)^{n(b)+n}v|_\infty \le 
1-\eps\hat b^{-\alpha}$ for all $n\ge0$, and
\[
|\hR(s)^{n(b)+n}v|_\theta/(2C_4\hat b) \le 
C_4\{\hat b+\theta^n C_5\hat b)/(2C_4\hat b)\le
{\textstyle\frac 12}+{\textstyle\frac 12}C_5\theta^n \le {\textstyle\frac 34},
\]
for $n$ sufficiently large (independent of $b$).  Increasing $\xi$
slightly, $\|\hR(s)^{n(b)}v\|_b \le 1-\eps\hat b^{-\alpha}$
proving part~(a).

It follows that $\|(I-\hR(s)^{n(b)})^{-1}\|_b \le \eps^{-1}\hat b^{\alpha}$.
Using the identity $(I-A)^{-1}=(I+A+\dots+A^{m-1})(I-A^m)^{-1}$,
\[
\|(I-\hR(s))^{-1}\|_b 
\le \sum_{j=0}^{n(b)-1}\|\hR(s)^j\|_b \|(I-\hR(s)^{n(b)})^{-1}\|_b
\le \xi\ln\hat b\, (C_4+\textstyle{\frac12})  \eps^{-1}\hat b^\alpha= O(\hat b^{\alpha+1}).
\]
Hence $\|(I-\hR(s))^{-1}\|_\theta=O(\hat b^{\alpha+2})$.
Increasing $\alpha$, we obtain part~(b).
\end{proof}

\begin{pfof}{Theorem~\ref{thm:Dolg}}  By Lemmas~\ref{lem:Dolg1} and~\ref{lem:Dolg2}(b), there exists $b_0>0$ such that that part~(b) of the theorem holds for all $|b|\ge b_0$.  Moreover, by Proposition~\ref{prop:radius}(b) and Lemma~\ref{lem:Dolg2}(a), the spectral radius of $\hR(s)$ is less than one for all $s\in\H$ and all $s=ib$ with $|b|\ge b_0$.

Suppose that the spectral radius of $\hR(ib)$ is $1$ for some $b\in\R\setminus\{0\}$.
Then $\hR(ib)v=\lambda v$ 
for some $\lambda\in\C$ with $|\lambda|=1$ and some nonzero $v\in \cF_\theta(Y)$.
By Proposition~\ref{prop:radius}(c), $e^{ib\varphi}v\circ F=\bar\lambda v$, so $e^{imb\varphi}v^m\circ F=\bar\lambda^m v^m$ for all $m\ge1$.
Again by Proposition~\ref{prop:radius}(c), $\hR(imb)v^m=\lambda^m v$ so the spectral radius of $\hR(imb)$ is equal to $1$ for all $m\ge1$.  
Choosing $m$ so that $m|b|\ge b_0$, we obtain a contradiction, hence proving part~(a).

Finally, by part~(a) and continuity of $\hR$ (Proposition~\ref{prop:R}),
part~(b) holds for $|b|\in[\delta,b_0]$.~
\end{pfof}

\subsection{Smoothness of $\hT=(I-\hR)^{-1}:\cF_\theta(Y)\to \cF_\theta(Y)$}
\label{sec:T}

\begin{prop} \label{prop:MT}
Assume absence of approximate eigenfunctions.
Then $\hT:\barH\setminus\{0\}\to\cB(\cF_\theta(Y))$ is $C^\infty$.
Moreover, for each $q\in\N$, $\delta>0$, there exists $\alpha,C>0$ such that
$\|\hT^{(q)}(s)\|_\theta\le C|b|^\alpha$ for all $s=a+ib\in\barH$
with $0\le a\le1$, $|b|\ge \delta$.
\end{prop}

\begin{proof} 
By Proposition~\ref{prop:radius}(b) and
Theorem~\ref{thm:Dolg}(a),
$1\not\in\spec\hR(s)$ for $s\in\barH\setminus\{0\}$.
Hence $\hT(s)$ is a bounded operator on $\cF_\theta(Y)$ for all $s\in\barH\setminus\{0\}$.
By Proposition~\ref{prop:R}, $s\mapsto \hR(s)$ is $C^\infty$ on $\barH$ so
$s\mapsto \hT(s)$ is $C^\infty$ on $\barH\setminus\{0\}$.

By induction, $\hT^{(q)}$ is a finite linear combination of finite products of factors of the form $\hT$ and $\hR^{(j)}$, $1\le j\le q$.
Each of these is $O(|b|^\alpha)$ for some $\alpha>0$ by
Proposition~\ref{prop:R} and Theorem~\ref{thm:Dolg}.
\end{proof}

Let $\cF_\theta(Y^\varphi)^0=\{v\in \cF_\theta(Y^\varphi): \int_{Y^\varphi}v\,d\mu^\varphi=0\}$.  As shown in the next result, $\hT R\hV$ extends smoothly from
$\barH\setminus\{0\}$ to $\barH$ when restricted to  $\cF_\theta(Y^\varphi)^0$.

\begin{cor} \label{cor:Dolg}
Assume absence of approximate eigenfunctions.  Then
$\hT R\hV:\barH\to\cB(\cF_\theta(Y^\varphi)^0,\cF_\theta(Y))$ is $C^\infty$.
\end{cor}

\begin{proof}
By Propositions~\ref{prop:V} and~\ref{prop:MT},
it suffices to work on $\barH\cap B_\delta(0)$ for some $\delta>0$.
By Corollary~\ref{cor:R}, we can choose $\delta>0$ so that
\[
\hT R\hV=(1-\lambda)^{-1} PR\hV+Q_1,
\]
where $s\mapsto Q_1(s):\cF_\theta(Y^\varphi)\to \cF_\theta(Y)$ is $C^\infty$.  
Also $(1-\lambda(s))^{-1}=s^{-1}\intphi^{-1}(1+s\tilde\lambda(s))$,
$P(s)=P(0)+s\widetilde P(s)$, $\hV(s)=\hV(0)+s\tilde V(s)$
where $\tilde\lambda\in\R$, $\widetilde P:\cF_\theta(Y)\to \cF_\theta(Y)$,
$R\tilde V:\cF_\theta(Y^\varphi)\to \cF_\theta(Y)$ are $C^\infty$ functions of $s$.
Hence
\[
\hT(s) R\hV(s)=s^{-1}\intphi^{-1}P(0)R\hV(0)+Q_2,
\]
where $s\mapsto Q_2(s):\cF_\theta(Y^\varphi)\to \cF_\theta(Y)$ is $C^\infty$.  

Finally, restricting to $\cF_\theta(Y^\varphi)^0$,
\begin{align*}
\SMALL P(0)R\hV(0)  =\int_Y R\hV(0)\,d\mu=\int_Y \hV(0)\,d\mu
 =\int_Y\int_0^{\varphi(y)}v(y,u)\,du\,d\mu
=\intphi\int_{Y^\varphi}v\,d\mu^\varphi=0.
\end{align*}
This completes the proof.
\end{proof}

\subsection{Proof of Theorem~\ref{thm:rapid}}
\label{sec:pf}

By Corollary~\ref{cor:poll},
$\hat\rho_{v,w}=\hJ_0+\intphi^{-1}\int_Y\hT R\hV\,\hw\,d\mu$.

Since $\varphi\in L^p(Y)$ for all $p\ge1$, it follows from the Cauchy-Schwarz and Markov inequalities that
$\int_Y 1_{\{\varphi>t\}}\varphi\,d\mu\le
(\mu(\varphi>t))^{1/2}|\varphi|_2
\le (|\varphi|_p\,t^{-p})^{1/2}|\varphi|_2\ll t^{-p/2}$ for all~$p\ge1$.
By Propositions~\ref{prop:poll} and~\ref{prop:w}, for each $q\in\N$ there exists $C>0$ such that 
\[
|\hJ_0^{(q)}(s)|\le C|v|_\infty|w|_\infty 
\le C\|v\|_\theta|w|_\infty 
\quad\text{and}\quad
|\hw^{(q)}(s)|_1\le C |w|_\infty,
\]
for all $v\in \cF_\theta(Y^\varphi)$, $w\in L^\infty(Y^\varphi)$, $s\in\barH$. 

By Propositions~\ref{prop:V} and~\ref{prop:MT} and Corollary~\ref{cor:Dolg},
for each $q\in\N$, there exist $C,\alpha>0$ such that
\[
|(\hT R \hV)^{(q)}(s)|_\infty\le 
\|(\hT R \hV)^{(q)}(s)\|_\theta\le 
C(|b|+1)^\alpha \|v\|_\theta,
\]
and so
\[
|\hat\rho_{v,w}^{(q)}(s)|\le C(|b|+1)^\alpha \|v\|_\theta|w|_\infty,
\]
for all 
$v\in \cF_\theta(Y^\varphi)^0$, $w\in L^\infty(Y^\varphi)$, $s=a+ib\in\C$ with $a\in[0,1]$.
Hence we have verified the hypotheses of Lemma~\ref{lem:strat}.
Consequently $|\rho_{v,w}(t)|=O(\|v\|_\theta|w|_{\infty,m}\, t^{-q})$
for all $q\in\N$, completing the proof of Theorem~\ref{thm:rapid}.

\section{Proof of polynomial mixing for semiflows}
\label{sec:slow}

In this section, we consider Gibbs-Markov semiflows $F_t:Y^\varphi\to Y^\varphi$ for which 
the roof function $\varphi:Y\to\R^+$ 
satisfies $\mu(\varphi>t)=O(t^{-\beta})$ for some $\beta>1$.
For such semiflows, we prove Theorem~\ref{thm:slow}, namely that absence of approximate eigenfunctions is a sufficient condition to obtain the mixing rate
$O(t^{-(\beta-1)})$.

The assumption on $\varphi$ implies that $\varphi\in L^q(Y)$ for all $q<\beta$ but in general $\varphi\not\in L^\beta(Y)$.  
Nonintegrability of $\rho_{v,w}$ ({\em a priori} for $\beta>2$ and even {\em a posteriori} for $\beta\in(1,2]$) makes inversion of the Laplace transform problematic.

To circumvent this, we use a truncation idea from~\cite{M09}.
The truncated semiflows are rapid mixing by Section~\ref{sec:rapid} and all 
components of the Laplace transform are $C^\infty$.  The approach in~\cite{M09} allows for control of the errors that come from truncation.

In Subsection~\ref{sec:some}, we introduce some notation and recall some elementary properties of Fourier transforms and convolutions.
Subsections~\ref{sec:refine} and~\ref{sec:further} contain various refinements of the estimates in Section~\ref{sec:rapid}.
In Subsection~\ref{sec:trunc}, we reduce to truncated semiflows.
The Laplace transform for truncated semiflows is studied for small $b$ and large $b$ in Subsections~\ref{sec:smallb} and~\ref{sec:largeb} respectively.

\subsection{Some conventions}
\label{sec:some}

From now on, we allow $q$ to take noninteger values.  
(Eventually, we require $q\in(1,\beta)$ with $q\ge\beta-1$.  Hence simplified proofs are available for $\beta>2$, though certain estimates for intermediate results become less sharp.)

As usual, a function $f:\R\to\R$ is said to be $C^q$ if $f$ is $C^{[q]}$ and $f^{([q])}$ is $(q-[q])$-H\"older.
Moreover, we write $|f^{(q)}|\le g$ for some function $g:\R\to[0,\infty)$ if 
for all $b,b'\in\R$,
\[
|f^{(k)}(b)|\le g(b), \;k=0,1,\dots,[q],\;\text{and}\;
|f^{([q])}(b)- f^{([q])}(b')|\le (g(b)+g(b'))|b-b'|^{q-[q]}.
\]

\begin{defn} 
Let $f:\R\to\R$ be integrable.
We write $f\in\cR(a(t))$
if the inverse Fourier transform of $f$ is $O(a(t))$.
We use the same notation for Banach space valued functions.
\end{defn}

\begin{prop} \label{prop:fourier}
Let $g:\R\to\R$ be an integrable function such that $g(b)\to0$ as $b\to\pm\infty$.  If $|f^{(q)}|\le g$, then $f\in\cR(t^{-q})$.
\end{prop}

\begin{proof}
Let $S$ denote the inverse Fourier transform of $f$.  Write $q=k+r$ where $k=[q]$ and $r\in[0,1)$.
Up to a multiplicative constant, $S(t)=\int_{-\infty}^\infty e^{ibt}f(b)\,db$.
Integrating by parts,
\[
S(t)=(it)^{-1}e^{ibt}f(b)\Big|_{b=-\infty}^{b=\infty}-
(it)^{-1}\int_{-\infty}^\infty e^{ibt}f'(b)\,db=
(-it)^{-1}\int_{-\infty}^\infty e^{ibt}f'(b)\,db.
\]
Inductively, $S(t)=(-it)^{-k}\int_{-\infty}^\infty e^{ibt}f^{[k]}(b)\,db$.
In particular, $|S(t)|\le t^{-k}\int g\ll t^{-k}$ proving the result when $q=k$ is an integer.

Next, $S(t)=(-it)^{-k}\int_{-\infty}^\infty -e^{ibt}f^{[k]}(b+\pi/t)\,db$ so
$S(t)=\frac12 (-it)^{-k}\int_{-\infty}^\infty e^{ibt}\{f^{[k]}(b)-f^{[k]}(b+\pi/t)\}\,db$.  Hence
\[
|S(t)|\le \frac12 t^{-k}\int_{-\infty}^\infty  \Big(g(b)+g\big(b+\frac{\pi}{t}\big)\Big)\Big(\frac{\pi}{t}\Big)^r\,db \le \pi^r t^{-q}\int g\ll t^{-q},
\]
as required.
\end{proof}

\begin{defn}  Let $f,g:[0,\infty)\to\R$ be integrable.  The {\em convolution}
$f\star g$ is defined to be $(f\star g)(t)=\int_0^t f(x)g(t-x)\,dx$.
\end{defn}

\begin{prop}  \label{prop:conv}
Fix $b>a>0$ with $b>1$.
Suppose that $f,g:[0,\infty)\to\R$ are integrable and there exists a constant
$C>0$ such that $|f(t)|\le C(1+t)^{-a}$ and $|g(t)|\le C(1+t)^{-b}$ for $t\ge0$.
Then there exists a constant $K>0$ depending only on $a$ and $b$ such that
\[
|(f\star g)(t)|\le C^2K(1+t)^{-a} \quad\text{for $t\ge0$.}
\]
\end{prop}

\begin{proof}
Write the convolution as a sum of two integrals
\[
\SMALL I_1(t)=\int_0^{t/2} f(x)g(t-x)\,dx, \qquad
I_2(t)=\int_{t/2}^t f(x)g(t-x)\,dx.
\]
Note that $|I_1(t)|\le C^2(1+t/2)^{-b}\int_0^t (1+x)^{-a}\,dx$
and 
$|I_2(t)|\le C^2(1+t/2)^{-a}\int_0^t (1+x)^{-b}\,dx$.
Clearly we can restrict attention to $t\ge1$.
Since $b>1$, it follows that $|I_2(t)|\ll C^2t^{-a}$ and
$|I_1(t)|\ll C^2 t^{-b}(1+t^{1-a}) \le 2C^2 t^{-a}$.
\end{proof}

\subsection{Refined estimates for  $J_0$, $w_0$, $R\hV$ and $\hR$.}
\label{sec:refine}

Throughout the remainder of this section, we fix 
\[
\max\{1,\beta-1\}<q<\beta.
\]
Let $\eta\in(0,1]$ be as in the statement of Theorem~\ref{thm:slow}.
Shrinking $\eta$ if needed, we may suppose without loss that
\[
q+2\eta<\beta.
\]
Let $\theta_1=\theta^{\eta}$.  Since $\theta_1\ge\theta$,
 condition~\eqref{eq:inf} holds also with $\theta$ replaced by $\theta_1$.
Hence results from Section~\ref{sec:rapid}, for example Theorem~\ref{thm:Dolg}, hold also in $\cF_{\theta_1}(Y)$.

We begin by giving improved estimates for 
$J_0(t)$ and $w(t)$.

\begin{prop} \label{prop:varphieta}
Let $\xi\in(0,\beta)$.
There exists a constant $C>0$ such that
$\int_Y 1_{\{\varphi>t\}}\varphi^\xi\,d\mu\le Ct^{-(\beta-\xi)}$
for all $t\ge0$.
\end{prop}

\begin{proof}
Let $G(x)=\mu(\varphi\le x)$. Then
\begin{align*}
\SMALL  \int_Y &  1_{\{\varphi>t\}}\varphi^\xi\,d\mu  
\SMALL  =\int_t^\infty x^\xi\,dG(x)
=-\int_t^\infty x^\xi\,d(1-G(x)) \\ &
\SMALL  =-x^\xi(1-G(x))\bigl|_{x=t}^{x=\infty}
+\xi\int_t^\infty x^{\xi-1}(1-G(x))\,dx
\ll t^{-(\beta-\xi)}+ \int_t^\infty x^{-\beta-1+\xi}\,dx
\ll t^{-(\beta-\xi)},
\end{align*}
as required.
\end{proof}

\begin{cor} \label{cor:Jw}
$|J_0(t)|=O(|v|_\infty|w|_\infty\, t^{-(\beta-1)})$ and
$|w(t)|_1=O(|w|_\infty\, t^{-\beta})$ 
for all $v,w\in L^\infty(Y^\varphi)$.
\end{cor}

\begin{proof}
This is immediate from Proposition~\ref{prop:varphieta} together with
Propositions~\ref{prop:poll} and~\ref{prop:w}.
\end{proof}

Next, we mention an improvement to Proposition~\ref{prop:R}.

\begin{prop} \label{prop:Rslow}
 $\hR:i\R\to\cB(\cF_{\theta_1}(Y))$ is $C^q$.
Indeed, 
\[
\SMALL \|\hR^{(q)}(ib)\|_{\theta_1} \le
C_2(2C_1)^{q+4}(|b|+q+1)\int_Y \varphi^{q+\eta}\,d\mu
\quad\text{for all $b\in\R$.}
\]
\end{prop}

\begin{proof}
The proof is a refinement of that for Proposition~\ref{prop:R}.
Write 
$\hR(ib)v=R(f(b)v)$, where $f(b)=e^{-ib\varphi}$.
We claim that 
\begin{align} \label{eq:Rslow}
& |1_{Y_j}f^{(q)}(b)|_\infty  \le 2|1_{Y_j}\varphi|_\infty^q, \qquad
 |1_{Y_j}f^{(q)}(b)|_{\theta_1}  \le 2C_1(2|b|+q+1)|1_{Y_j}\varphi|_\infty^{q+\eta}.
\end{align}
It then follows that
$\|1_{Y_j}f^{(q)}(b)\|_{\theta_1}\le 
2C_1(2|b|+q+2)|1_{Y_j}\varphi|_\infty^{q+\eta}
\le (2C_1)^{q+3}(|b|+q+1)\infYj\varphi^{q+\eta}$.
Now apply Remark~\ref{rmk:GM}.

It remains to prove the claim.
Write $q=k+r$ where $k=[q]$ and $r\in[0,1)$.
Then $|1_{Y_j}f^{(k)}(b)|_\infty\le |1_{Y_j}\varphi|_\infty^k$.
Hence $|1_{Y_j}\{f^{(k)}(b+h) -f^{(k)}(b)\} |_\infty\le 2|1_{Y_j}\varphi|_\infty^k$.
Also $|1_{Y_j}f^{(k+1)}(b)|_\infty\le |1_{Y_j}\varphi|_\infty^{k+1}$ so it follows from the mean value theorem that
$|1_{Y_j}\{f^{(k)}(b+h) -f^{(k)}(b)\} |_\infty\le |1_{Y_j}\varphi|_\infty^{k+1}|h|$.
Combining these two estimates and using the inequality $\min\{1,x\}\le x^r$ which holds for all $x\ge0$, $r\in[0,1]$,
\[
|1_{Y_j}\{f^{(k)}(b+h) -f^{(k)}(b)\} |_\infty\le 
2|1_{Y_j}\varphi|_\infty^k\min\{1,|1_{Y_j}\varphi|_\infty|h|\}
\le 2|1_{Y_j}\varphi|_\infty^{k+r}|h|^r,
\]
yielding the first part of~\eqref{eq:Rslow}.

Next, for $y,y'\in Y_j$, we have
$f^{(k)}(b)(y)-f^{(k)}(b)(y')=(-i)^k(I_1+I_2)$ where
\[
I_1=\{e^{-ib\varphi(y)}-e^{-ib\varphi(y')}\}\varphi(y)^k, \qquad
I_2=e^{-ib\varphi(y')}\{\varphi(y)^k-\varphi(y')^k\}.
\]
We have
$|I_1|\le |b|C_1|1_{Y_j}\varphi|_\infty^{k+1}d_{\theta}(y,y')$ and
$|I_2|\le kC_1|1_{Y_j}\varphi|_\infty^kd_{\theta}(y,y')$.
Also, $|I_1|\le 2|1_{Y_j}\varphi|_\infty^k$,
so
\begin{align*}
|I_1| & \le 2|b|C_1|1_{Y_j}\varphi|_\infty^k\min\{1,|1_{Y_j}\varphi|_\infty d_{\theta}(y,y')\}
\\ & \le 2|b|C_1|1_{Y_j}\varphi|_\infty^{k+\eta} d_{\theta}(y,y')^\eta
=2|b|C_1|1_{Y_j}\varphi|_\infty^{k+\eta} d_{\theta_1}(y,y').
\end{align*}
Hence $g^{(k)}(b)=
1_{Y_j}\{f^{(k)}(b)(y)-f^{(k)}(b)(y')\}$ satisfies
\[
|g^{(k)}(b)|\le C_1(2|b|+k)|1_{Y_j}\varphi|_\infty^{k+\eta} d_{\theta_1}(y,y').
\]
Now we repeat the mean value theorem argument above to obtain
\begin{align*}
|g^{(k)}(b+h)-g^{(k)}(b)| & \le 
2C_1(2|b|+k+1)|1_{Y_j}\varphi|_\infty^{k+\eta} d_{\theta_1}(y,y')\min\{
1,|1_{Y_j}\varphi|_\infty |h|\}
\\ & \le 2C_1(2|b|+k+1)|1_{Y_j}\varphi|_\infty^{k+r+\eta} d_{\theta_1}(y,y')|h|^r.
\end{align*}
Hence 
$|1_{Y_j}\{f^{(q)}(b)(y)-f^{(q)}(b)(y')\}|\le 
2C_1(2|b|+q+1)|1_{Y_j}\varphi|_\infty^{q+\eta} d_{\theta_1}(y,y')$,
yielding the second part of~\eqref{eq:Rslow}.
\end{proof}

\begin{rmk} Clearly, the estimate for $\hR^{(q)}$ holds equally for $\hR^{(q')}$ for all $q'<q$.  We use this observation without comment throughout.
\end{rmk}

\begin{rmk} \label{rmk:MVT}
During this section, we obtain many estimates of the form
$|f^{(k)}(b)|\ll \varphi^{k+\ell}$ for all $k\in\N$.  By the mean value theorem argument used in the proof of Proposition~\ref{prop:Rslow}, it follows that
$|f^{(q)}(b)|\ll \varphi^{q+\ell}$ for all $q\in[0,\infty)$.
From now on, we write ``by the MVT argument'' and omit the details.
\end{rmk}

\begin{cor} \label{cor:MTslow}
Assume absence of approximate eigenfunctions.
Then $\hT=(I-\hR)^{-1}:i\R\setminus\{0\}\to\cB(\cF_{\theta_1}(Y))$ is $C^q$.
Moreover, for all $\delta>0$, there exists $\alpha,C>0$ such that
$\|\hT^{(q)}(ib)\|_{\theta_1}\le C|b|^\alpha$ for all $b\in\R$ with $|b|\ge\delta$.
\end{cor}

\begin{proof}
The proof of the Dolgopyat estimate in Theorem~\ref{thm:Dolg} is completely unchanged
(Proposition~\ref{prop:basic}(b) used only the integrability of $\varphi$).
Hence $\hT$ is $C^q$ by Proposition~\ref{prop:Rslow}.

Let $k\in\N$ with $k<\beta$.
By induction, $\hT^{(k)}$ is a finite linear combination of finite products of factors of the form $\hT$ and $\hR^{(j)}$, $1\le j\le k$.
Each of these is $O(|b|^\alpha)$ for some $\alpha>0$ by
Theorem~\ref{thm:Dolg} and Proposition~\ref{prop:Rslow}.
Hence there exist constants $C$, $\alpha>0$ such that 
$\|\hT^{(k)}(ib)\|_{\theta_1}\le C|b|^\alpha$ for each $k=0,\dots,[q]$.

Next, write $q=k+r$ where $k=[q]$ and $r\in[0,1)$.  
By the resolvent identity together with Theorem~\ref{thm:Dolg} and Proposition~\ref{prop:Rslow},
\[
\|\hT(ib)-\hT(ib')\|_{\theta_1}\le \|\hT(ib)\|_{\theta_1}\|\hR(ib)-\hR(ib')\|_{\theta_1}\|\hT(ib')\|_{\theta_1}\le C^3|b|^{3\alpha}|b-b'|^r,
\]
so $\|\hT^{(r)}(ib)\|_{\theta_1}\le C^3|b|^{3\alpha}$.

Finally, $\hT^{(q)}$ is a finite linear combination of finite products of factors of the form $\hT$, $\hT^{(r)}$ and $\hR^{(p)}$, $p\le q$, each of which is now covered.
\end{proof}

In the last part of this subsection, we refine the estimate for $R\hV$.
First, we 
recall a basic calculus estimate from~\cite{MT17}.

\begin{prop} \label{prop:calculus}
Let $g(x)=(e^{ix}-1)/x$.
For any $k\ge0$, there exists a constant $C>0$ such that
$|g^{(k)}(x)|\le C$ and
$|g^{(k)}(x)|\le C/|x|$ for all $x\in\R$.
\end{prop}

\begin{proof}
This is~\cite[Proposition~13.2]{MT17}.
We give the proof for completeness.

Define the analytic functions $q_k,r_k:\C\to\C$ for $k\ge1$,
\[
q_k(z)=e^z-\sum_{m=0}^{k-1}\frac{z^m}{m!}, \quad r_k(z)=\frac{q_k(z)}{z^k}.
\]
By Taylor's theorem, there exists $\xi$ between $0$ and $z$ such that
\[
q_k(z)
=\sum_{m=0}^{k-1}q_k^{(m)}(0)z^m/m! + q_k^{(k)}(\xi)z^k/k!
= e^{\xi}z^k/k!,
\]
so that $|q_k(ix)|\le |x|^k/k!\,$
Similarly,
\[
q_k(z)
=\sum_{m=0}^{k-2}q_k^{(m)}(0)z^m/m! + q_k^{(k-1)}(\xi)z^{k-1}/(k-1)!
= (e^{\xi}-1)z^{k-1}/(k-1)!,
\]
so that $|q_k(ix)|\le |x|^{k-1}/(k-1)!$

Next, note by induction that
$r_1^{(k)}\in \R\{e^z/z,\,e^z/z^2,\dots,e^z/z^k,\,(e^z-1)/z^{k+1}\}$.
But $e^z/z^j-r_j\in\R\{1/z,\dots,1/z^j\}$.
Hence there exist constants $a_1,\dots,a_{k+1}$ and a polynomial~$p$ of degree at most $k$ such that
\[
\SMALL r_1^{(k)}(z)=\sum_{j=1}^{k+1}a_jr_j(z)+p(z)/z^{k+1}.
\]
Since all terms in this identity are analytic with the possible exception of the last one, we deduce that $p\equiv0$.
Hence
\[
\SMALL r_1^{(k)}(z)=\sum_{j=1}^{k+1}a_jr_j(z)
=\sum_{j=1}^{k+1}a_jq_j(z)/z^j.
\]

Since $g(x)=ir_1(ix)$,
the result follows by substituting in the estimates for $q_j$.
\end{proof}

\begin{prop} \label{prop:v}
Let $v\in \cF_{\theta}(Y^\varphi)$.  Then 
$|v(y,u)-v(y',u)|\le 4C_1\|v\|_{\theta}\,\infYj\varphi^\eta d_{\theta_1}(y,y')$
for all 
$(y,u),\,(y',u)\in Y^\varphi$ with $y,y'\in Y_j$, $j\ge1$.
\end{prop}

\begin{proof}
We have
$|v(y)-v(y')|\le\min\{2|v|_\infty,|v|_{\theta}\,\varphi(y)d_{\theta}(y,y')\}
\le 2\|v\|_{\theta}|1_{Y_j}\varphi|_\infty^\eta d_{\theta}(y,y')^\eta
\le 4C_1\|v\|_{\theta}\,\infYj\varphi^\eta d_{\theta_1}(y,y')$.
\end{proof}

It is convenient to split $v$ into a part independent of $u$ and a part that vanishes at $u=0$.

\begin{prop} \label{prop:Vslow0}
$R\hV^{(q)}:i\R\to\cB(\cF_\theta(Y^\varphi),\cF_{\theta_1}(Y))$ is $C^q$
for $v$ independent of $u$.
Moreover, there exists $C>0$ such that
\[
\SMALL \|R\hV^{(q)}(ib)\|_{\theta_1}\le C\|v\|_{\theta} \int_Y\varphi^{q+2\eta}\,d\mu\, |b|^{-(1-\eta)},
\]
for all $b\in\R\setminus\{0\}$ and $v\in \cF_{\theta}(Y^\varphi)$ such that $v$ is independent of $u$.  
\end{prop}

\begin{proof}  
Write $R\hV(ib)=iR(f(b)v)$ where 
\[
f(b)=
b^{-1}(e^{-ib\varphi}-1)
=\varphi g(b\varphi),
\qquad g(x)=x^{-1}(e^{-ix}-1).
\]
By Proposition~\ref{prop:calculus}, for any $k\in\N$, there exists a constant $C>0$ such that
for all $x\in\R$,
\begin{align} \label{eq:MT17}
|g^{(k)}(x)|\le C\min\{1,|x|^{-1}\}=C|x|^{-1}\min\{1,|x|\}\le C|x|^{-(1-\eta)}.
\end{align}
Hence
\[
|1_{Y_j}f^{(k)}(b)|_\infty \ll 
|1_{Y_j}\varphi|_\infty^{k+1}(|b||1_{Y_j}\varphi|_\infty)^{-(1-\eta)}
=|1_{Y_j}\varphi|_\infty^{k+\eta}|b|^{-(1-\eta)}.
\]
By the MVT argument
\[
|1_{Y_j}f^{(q)}(b)|_\infty \ll 
|1_{Y_j}\varphi|_\infty^{q+\eta}|b|^{-(1-\eta)} \le
(2C_1)^{q+1}\infYj\varphi^{q+\eta}|b|^{-(1-\eta)}.
\]

Next, let $y,y'\in Y_j$.  Using the identity
$xg(bx)-x'g(bx')=e^{ibx'}(x-x')g(b(x-x'))$,
\[
f(b)(y)-f(b)(y')=e^{ib\varphi(y')}(\varphi(y)-\varphi(y'))\,g\big(b(\varphi(y)-\varphi(y'))\big).
\]
By~\eqref{eq:MT17},
\begin{align*}
|f^{(k)}(b)(y)-f^{(k)}(b)(y')| &  \ll
\sum_{j=0}^k \varphi(y')^{k-j}|\varphi(y)-\varphi(y')|^{1+j}
|g^{(j)}\big(b(\varphi(y)-\varphi(y'))\big) \\
& \ll 
\sum_{j=0}^k \varphi(y')^{k-j}|\varphi(y)-\varphi(y')|^{j+\eta}
|b|^{-(1-\eta)} 
\\ & \le
(k+1)(2C_1)^{k+1} \infYj\varphi^{k+\eta}  d_{\theta_1}(y,y')
|b|^{-(1-\eta)}.
\end{align*}
By the MVT argument
\begin{align*}
|f^{(q)}(b)(y)-f^{(q)} & (b)(y')|  \ll
\infYj\varphi^{q+\eta}  d_{\theta_1}(y,y')
|b|^{-(1-\eta)}.
\end{align*}
It follows that $\|1_{Y_j}f^{(q)}(b)\|_{\theta_1}\ll \infYj\varphi^{q+\eta} |b|^{-(1-\eta)}$.
By Proposition~\ref{prop:v},
$\|1_{Y_j}f^{(q)}(b)v\|_{\theta_1} \ll
\|v\|_\theta\, \infYj\varphi^{q+2\eta} |b|^{-(1-\eta)}$.
Now apply Corollary~\ref{cor:GM}.
\end{proof}

In the remainder of this subsection, we work with the function spaces
$\cF_{\theta,\eta}(Y^\varphi)$ with norm $\|\cdot\|_{\theta,\eta}=|\cdot|_\theta+|\cdot|_{\infty,\eta}$
as defined in Section~\ref{sec:obs}.

\begin{prop} \label{prop:Vslow1}
There is a constant $C>0$ such that
$\|RV(t)\|_{\theta_1}\le C\|v\|_{\theta,\eta}\,t^{-q}$ 
for all $v\in \cF_{\theta,\eta}(Y^\varphi)$ with $v(y,0)\equiv0$, and all $t>1$.
\end{prop}

\begin{proof}
Recall that $V(t)(y)=1_{\{\varphi(y)>t\}}v(y,\varphi(y)-t)$.
By~\eqref{eq:inf},
\begin{align*}
|1_{Y_j}V(t)|_\infty & \le |v|_\infty 1_{\{|1_{Y_j}\varphi|_\infty>t\}}
\le |v|_\infty 1_{\{\infYj \varphi>t/(2C_1)\}}.
\end{align*}
Also, for $y,y'\in Y_j$, $j\ge1$, with $\varphi(y)\ge\varphi(y')$,
\[
V(t)(y)-V(t)(y')=\begin{cases} 
v(y,\varphi(y)-t)- v(y',\varphi(y')-t), & \varphi(y')>t \\
v(y,\varphi(y)-t), & \varphi(y)>t\ge \varphi(y') \\ 
 0, & \varphi(y)\le t
\end{cases}.
\]
If $\varphi(y')>t$, then using Proposition~\ref{prop:v},
\begin{align*}
& |V(t)(y)-V(t)(y')|  \le
1_{\{|1_{Y_j}\varphi|_\infty>t\}}\{|v(y,\varphi(y)-t)-v(y',\varphi(y)-t)|
\\ & 
\qquad \qquad 
\qquad \qquad 
\qquad \qquad 
\qquad \qquad 
+ |v(y',\varphi(y)-t)-v(y',\varphi(y')-t)|\}
\\ & \qquad  \le 1_{\{|1_{Y_j}\varphi|_\infty>t\}}\{4C_1\infYj\varphi^\eta|v|_{\theta}\, d_{\theta_1}(y,y')+ 
|v|_{\infty,\eta} |\varphi(y)-\varphi(y')|^\eta \}
\\ & \qquad  \le 5C_1 1_{\{|1_{Y_j}\varphi|_\infty>t\}}\|v\|_{\theta,\eta} 
\infYj\varphi^\eta \,d_{\theta_1}(y,y').
\end{align*}
If $\varphi(y)>t\ge \varphi(y')$, then
\begin{align*}
 & |V(t)(y)-V(t)(y')|    =
1_{\{\varphi(y)>t\ge \varphi(y') \}}|v(y,\varphi(y)-t)|
\\ & \qquad  =1_{\{\varphi(y)>t\ge \varphi(y') \}}|v(y,\varphi(y)-t)-v(y,0)|
\le 1_{\{\varphi(y)>t\ge \varphi(y') \}}|v|_{\infty,\eta}|\varphi(y)-t|^\eta
\\ & \qquad \le 1_{\{\varphi(y)>t\}}|v|_{\infty,\eta}|\varphi(y)-\varphi(y')|^\eta
\le C_11_{|1_{Y_j}\varphi|_\infty>t\}}|v|_{\infty,\eta}\infYj\varphi^\eta \,d_{\theta_1}(y,y').
\end{align*}
Hence in all cases,
\[
 |V(t)(y)-V(t)(y')|\le 5C_1 1_{\{|1_{Y_j}\varphi|_\infty>t\}}\|v\|_{\theta,\eta}\infYj\varphi^\eta \,d_{\theta_1}(y,y').
\]
By~\eqref{eq:inf},
\[
\|1_{Y_j}V(t)\|_{\theta_1}\le 
5C_1 \|v\|_{\theta,\eta}1_{\{\infYj \varphi>t/(2C_1)\}}\infYj\varphi^\eta.
\]
Hence by Corollary~\ref{cor:GM},
\begin{align*}
\|RV(t)\|_{\theta_1} & \le 10C_1C_2\|v\|_{\theta,\eta}\sum\mu(Y_j)1_{\{\infYj \varphi>t/(2C_1)\}}\infYj\varphi^\eta
\\ & \SMALL \le 10C_1C_2\|v\|_{\theta,\eta}\int_Y 1_{\{\varphi>t/(2C_1)\}}\varphi^\eta\,d\mu.
\end{align*}
Now apply Proposition~\ref{prop:varphieta}.
\end{proof}

\begin{cor} \label{cor:Vslow}
Let $\kappa:\R\to\R$ be $\C^\infty$ with $|\kappa^{(k)}(b)|=O((b^2+1)^{-1})$ for all $k\in\N$.
Then 
$\kappa R\hV\in\cR(\|v\|_{\theta,\eta}\, t^{-q})$
in $\cF_{\theta_1}(Y)$ for all $v\in \cF_{\theta,\eta}(Y^\varphi)$.
\end{cor}

\begin{proof}
Write $v(y,u)=v_0(y)+v_1(y,u)$ where $v_0(y)=v(y,0)$.
The result follows by combining Propositions~\ref{prop:Vslow0} and~\ref{prop:Vslow1} and applying Proposition~\ref{prop:fourier}.
\end{proof}

\subsection{Further estimates for dealing with the singularity at zero}
\label{sec:further}

By Propositions~\ref{prop:qc} and~\ref{prop:Rslow},
there exists $\delta>0$ such that $\hR(ib)$ has a $C^q$ family of simple eigenvalues $\lambda(b)$, $|b|<\delta$, with $\lambda(0)=1$
and $\lambda'(0)=-i\intphi$.
Let $P(b)$, $|b|<\delta$, denote the corresponding $C^q$ family of spectral projections, with $P(0)v=\int_Y v\,d\mu$ for $v\in L^1(Y)$.

\begin{prop} \label{prop:v0}
$b^{-1}P(0)R\hV(ib)\in\cR(|v|_\infty\, t^{-(\beta-1)})$ for all 
$v\in L^\infty(Y^\varphi)$ with \mbox{$\int_{Y^\varphi} v\,d\mu^\varphi=0$}.
\end{prop}

\begin{proof}
Let $\hat g(s)=s^{-1}P(0)R\hV(s)=s^{-1}\int_Y\int_0^{\varphi(y)}e^{-s(\varphi(y)-u)}v(y,u)\,du\,d\mu$.
Since $v$ has mean zero,
\[
\hat g(s)=\int_Y\int_0^{\varphi(y)}s^{-1}(e^{-s(\varphi(y)-u)}-1)v(y,u)\,du\,d\mu.
\]
Hence
\[
g(t)=-\int_Y \int_0^{\varphi(y)}1_{\{\varphi(y)>t+u\}}v(y,u)\,du\,d\mu.
\]
By Proposition~\ref{prop:varphieta},
\[
|g(t)|\le |v|_\infty \int_Y \int_0^\varphi 1_{\{\varphi>t\}}\,du\,d\mu
= |v|_\infty \int_Y \varphi 1_{\{\varphi>t\}}\,d\mu
\ll |v|_\infty \,t^{-(\beta-1)},
\]
as required.
\end{proof}

Define $\widetilde R(s)=s^{-1}(\hR(s)-\hR(0))$, $s\in\H$.

\begin{prop} \label{prop:tildeR}
Let $q_1>0$.
There exists a constant $C>0$ such that
\[
\|\widetilde R^{(q_1)}(ib)\|_{\theta_1} \le \begin{cases}
C|b|^{-(1-\eta)} & q_1<\beta-2\eta \\
C & q_1<\beta-1 \end{cases}\quad\text{for all $b\in\R\setminus\{0\}$.}
\]
\end{prop}

\begin{proof}
Write
$\widetilde R(ib)v= R(f(b)v)$ where $f(b)=\int_0^\varphi e^{-ibt}\,dt$.
This is the same function as in the proof of Proposition~\ref{prop:Vslow0}, leading to the same conclusion for $q_1<\beta-2\eta$.

The argument for $q_1<\beta-1$ is much simpler since
$f^{(k)}(b)=\int_0^\varphi e^{-ibt}(-it)^k\,dt$
and $\varphi^{q_1+1}$ is integrable.  We omit the details.
\end{proof}

For $b\in(-\delta,\delta)$, define
\[
\widetilde P(b)=b^{-1}(P(b)-P(0)), \quad
\tilde\lambda(b)=(ib\intphi)^{-1}(1-\lambda(b)).
\]

\begin{prop} \label{prop:tildeP}
The conclusion of Proposition~\ref{prop:tildeR}
holds with $\|\widetilde R^{(q_1)}(ib)\|_{\theta_1}$ replaced by
$\|\widetilde P^{(q_1)}(b)\|_{\theta_1}$ 
or $|\tilde \lambda^{(q_1)}(b)|$ for all $|b|<\delta$. 
\end{prop}

\begin{proof}
Let $\Gamma\subset\C$ be a sufficiently small circle centered at $1$.
Then $P(b)=(2\pi i)^{-1}\int_\Gamma (\xi-\hR(ib))^{-1}\,db$ for all
$|b|<\delta$.  Hence
\[
\SMALL \widetilde P(b)=(2\pi i)^{-1}\int_\Gamma (\xi-\hR(ib))^{-1}\widetilde R(ib) (\xi-\hR(0))^{-1}\,d\xi.
\]
By Proposition~\ref{prop:Rslow}, 
$(\xi-\hR)^{-1}$ is $C^{q_1}$ uniformly in $\xi\in\Gamma$.   Hence
the estimates for $\widetilde P^{(q_1)}$ follow from Proposition~\ref{prop:tildeR}.

Next, write 
\begin{align} \label{eq:lambdaP}
\tilde\lambda(b)P(b)=
\{-\widetilde R(ib)P(b)+(I-\hR(0))\widetilde P(b)\}/(i\intphi).
\end{align}
Since $P$ is $C^{q_1}$, it follows from Proposition~\ref{prop:tildeR}
and the estimates for $\widetilde P^{(q_1)}$ that $\|(\tilde\lambda P)^{(q_1)}(b)\|_{\theta_1}$ also satisfies these estimates.
Let $f(b)\in \cF_{\theta_1}(Y)$ denote the eigenfunction corresponding to $\lambda(b)$ normalised so that $\int_Y f(b)\,d\mu=1$.  
(If necessary, shrink $\delta$ so that $f(b)>0$ ensuring that the normalisation exists.) Then $b\mapsto f(b)$ is $C^{q_1}$ and 
\[
\SMALL \tilde\lambda(b)=\int_Y \tilde\lambda(b)f(b)\,d\mu
=\int_Y\tilde\lambda(b)P(b)\lambda(b)^{-1}f(b)\,d\mu,
\]
yielding the estimates for $\tilde\lambda^{(q_1)}$.
\end{proof}

\subsection{Truncation}
\label{sec:trunc}

Given $N\ge1$, we replace $\varphi$ by $\varphi\wedge N=
\min\{\varphi,N\}$.
Consider the suspension semiflows $F_t$ and $F_{N,t}$ on $Y^\varphi$ and $Y^{\varphi\wedge N}$ respectively.
Let $\rho_{v,w}$ and $\rho^{\rm trunc}_{v,w}$ denote the respective correlation functions.
In particular, 
$\rho^{\rm trunc}_{v,w}(t)=
\int_{Y^{\varphi\wedge N}}v\,w\circ F_{N,t}\,d\mu^{\varphi\wedge N}-
\int_{Y^{\varphi\wedge N}}v\,\,d\mu^{\varphi\wedge N}
\int_{Y^{\varphi\wedge N}}w\,\,d\mu^{\varphi\wedge N}$
where the observables 
$v,w:Y^{\varphi\wedge N}\to\R$ are the restrictions of $v,w:Y^\varphi\to\R$ to
$Y^{\varphi\wedge N}$.

\begin{prop} 
\label{prop:trunc}
There are constants $C,\,t_0>0$, $N_0\ge1$ such that
\[
|\rho_{v,w}(t)-\rho^{\rm trunc}_{v,w}(t)|\le C|v|_\infty|w|_\infty
(tN^{-\beta}+N^{-(\beta-1)}),
\]
for all
$v,w\in L^\infty(Y^\varphi)$, $N\ge N_0$, $t> t_0$. \qed
\end{prop}

\begin{proof}
This proof follows~\cite[Section~3]{M09} and~\cite[Appendix~A]{MT14}.
We begin with some definitions and preliminary estimates.

Let $F_{N,t}:Y^{\varphi\wedge N}\to Y^{\varphi\wedge N}$ denote the truncated
flow.
Define the top of the tower
$T_N=Y^\varphi\setminus Y^{\varphi\wedge N}$
and the borderline region
$\partial T_N=\{(y,N)\in Y^\varphi: \varphi(y)>N\}$.
Also, we define the thickened borderline region
\[
\partial T_N'=\{(y,u)\in Y^\varphi: \varphi(y)>N,\,N-1\le u\le N\}.
\]

Now,
$\mu^\varphi(\partial T_N') =\intphi^{-1}\int_Y 1_{\{\varphi>N\}}\,d\mu \ll N^{-\beta}$ and
by Proposition~\ref{prop:varphieta},
\[
\mu^\varphi(T_N)
=\intphi^{-1}\int_Y 1_{\{\varphi>N\}}(\varphi-N)\,d\mu
\le \intphi^{-1}\int_Y 1_{\{\varphi>N\}}\varphi\,d\mu\ll N^{-(\beta-1)}.
\]
We note the related estimate 
\[
\intphi-|\varphi\wedge N|_1=\int_Y(\varphi-(\varphi\wedge N))\,d\mu
=\int_Y 1_{\{\varphi>N\}}(\varphi-N)\,d\mu\ll N^{-(\beta-1)}.
\]
Choosing $N$ sufficiently large, we can suppose
that $|\varphi\wedge N|_1\ge\frac12\intphi$ and hence that
$|\intphi^{-1}-|\varphi\wedge N|_1^{-1}|\ll N^{-(\beta-1)}$.

Fix $t\ge1$.  
Note that if $F_s(y,u)\in T_N$ for some $s\in[0,t]$, then either $(y,u)\in T_N$
or $F_s(y,u)\in \partial T_N$ for some $s\in[0,t]$.
Hence
\begin{align*}
 \{(y,u)\in  Y^{\varphi\wedge N}: F_t(y,u)\neq F_{N,t}(y,u)\}
   & \subset\{(y,u)\in Y^{\varphi\wedge N}: F_s(y,u)\in T_N\;\text{for some $s\in[0,t]$}\}
\\ & \subset T_N\cup \bigcup_{s\in[0,t+1]}F_s^{-1}\partial T_N
 = T_N\cup \bigcup_{j=0,\dots,[t+1]}F_j^{-1}\partial T_N'.
\end{align*}
It follows that
\begin{align*}
\mu^\varphi & \{(y,u)\in Y^{\varphi\wedge N}: F_t(y,u)\neq F_{N,t}(y,u)\}
\\ & \le \mu^\varphi(T_N)+ \sum_{j=0}^{[t+1]}\mu^\varphi(F_j^{-1}\partial T_N')
=\mu^\varphi(T_N)+ [t+1]\mu^\varphi(\partial T_N')\ll N^{-(\beta-1)}+tN^{-\beta}.
\end{align*}

Now,
\begin{align*}
\int_{Y^\varphi}v\,w\circ F_t\,d\mu^\varphi-
\int_{Y^{\varphi\wedge N}}v\,w\circ F_{N,t}\,d\mu^{\varphi\wedge N}=I_1+I_2+I_3,
\end{align*}
where
\begin{align*}
I_1 & = \int_{T_N}v\,w\circ F_t\,d\mu^\varphi, \qquad
I_2  = \int_{Y^{\varphi\wedge N}}v\,(w\circ F_t-w\circ F_{N,t})\,d\mu^\varphi, \\
I_3 & = (\intphi^{-1}-|\varphi\wedge N|^{-1})\int_Y\int_0^{\varphi\wedge N}v\,w\circ F_{N,t}\,du\, d\mu^{\varphi\wedge N}.
\end{align*}
It follows readily from the preceding calculations that
\begin{align*}
|I_1| & \le |v|_\infty |w|_\infty\, \mu^\varphi(T_N)
\ll |v|_\infty |w|_\infty\, N^{-(\beta-1)},  \qquad
|I_3|  \ll |v|_\infty |w|_\infty\, N^{-(\beta-1)},
\\
|I_2| & \le 2|v|_\infty|w|_\infty 
\mu^\varphi \{F_t(y,u)\neq F_{N,t}(y,u)\}
\ll |v|_\infty |w|_\infty\, (N^{-(\beta-1)}+tN^{-\beta}).
\end{align*}
Hence $|\int_{Y^\varphi}v\,w\circ F_t\,d\mu^\varphi-
\int_{Y^{\varphi\wedge N}}v\,w\circ F_{N,t}\,d\mu^{\varphi\wedge N}|\ll
|v|_\infty |w|_\infty\, (N^{-(\beta-1)}+tN^{-\beta})$.

A simpler calculation shows that $|
\int_{Y^\varphi}v\,d\mu^\varphi-
\int_{Y^{\varphi\wedge N}}v\,d\mu^{\varphi\wedge N}|\ll |v|_\infty\,N^{-(\beta-1)}$
and hence that 
$|\int_{Y^\varphi}v\,d\mu^\varphi
\int_{Y^\varphi}w\,d\mu^\varphi-
\int_{Y^{\varphi\wedge N}}v\,d\mu^{\varphi\wedge N}
\int_{Y^{\varphi\wedge N}}w\,d\mu^{\varphi\wedge N}|
\ll |v|_\infty|w|_\infty\,N^{-(\beta-1)}$.
This completes the proof.
\end{proof}

Below we prove:
\begin{lemma}  \label{lem:trunc}
Assume absence of approximate eigenfunctions.
There are constants $C>0$, $m\ge1$, $N_1\ge1$ such that
\[
|\rho^{\rm trunc}_{v,w}(t)|\le C\|v\|_{\theta,\eta}|w|_{\infty,m}\,
t^{-(\beta-1)},
\]
for all
$v\in \cF_{\theta,\eta}(Y^{\varphi\wedge N})$,
$w\in L^{\infty,m}(Y^{\varphi\wedge N})$, $N\ge N_1$, $t>1$. 
\end{lemma}

\begin{pfof}{Theorem~\ref{thm:slow}}
In general, $w\in L^{\infty,m}(Y^\varphi)$ need not restrict to 
$w\in L^{\infty,m}(Y^{\varphi\wedge N})$,
but we can choose
$w_N\in L^{\infty,m}(Y^{\varphi\wedge N})$ so that
$|w_N|_{\infty,m}\ll |w|_{\infty,m}$ and 
$w_N\equiv w$ outside the set
$S_N=\{((y,u)\in Y^{\varphi\wedge N}:\varphi(y)>N, \,u\in(N-1,N]\}$.
Then
\begin{align*}
 |\rho^{\rm trunc}_{v,w} & (t)  -  
\rho^{\rm trunc}_{v,w_N}(t)|
 \le |v|_\infty(|w|_\infty+|w_N|_\infty)\mu^{\varphi\wedge N}(F_{N,t}^{-1}S_N)
 \\ & =|v|_\infty(|w|_\infty+|w_N|_\infty)\mu^{\varphi\wedge N}(S_N)
\ll |v|_\infty|w|_\infty\mu(\varphi>N)
\ll |v|_\infty|w|_\infty\, N^{-\beta}.
\end{align*}
Taking $N=[t]$, the result follows directly from
Proposition~\ref{prop:trunc} and Lemma~\ref{lem:trunc}.
\end{pfof}

In the remainder of this subsection, we outline the strategy for proving Lemma~\ref{lem:trunc}.

By assumption, we can fix
a finite union $Z$ of partition elements such that the 
corresponding finite subsystem $Z_0$ does not support approximate eigenfunctions.
Choose $N_1\ge |1_Z \varphi|_\infty$.

For each fixed $N$, the truncated roof function $\varphi\wedge N$ is bounded
and hence the results in Section~\ref{sec:rapid} apply.  In particular,
$\hat\rho_{v,w}^{\rm trunc}$ is $C^\infty$ on $\barH$ and contours of integration can be moved to the imaginary axis.
From now on we suppress the superscript ``${\rm trunc}$'' for sake of readability.
The calculations in Proposition~\ref{prop:poll} and
Corollary~\ref{cor:poll} now proceed on the imaginary axis in identical fashion to the calculation on $\H$.  Hence
\[
\rho_{v,w}(t)=\int_{-\infty}^\infty e^{ibt}\hat\rho_{v,w}(ib)\,db,
\qquad
\hat\rho_{v,w}=\hJ_{0,v,w}+
\intphi^{-1}\int_Y \hT R\hV\,\hw\,d\mu,
\]
where the constituent parts $\hJ_{0,v,w}$, $\hT=(I-\hR)^{-1}$, $\hV$, $\hw$ are $C^\infty$ ``truncated'' versions of the originals.

It follows from rapid mixing for the truncated semiflow that 
$t\mapsto (\rho_{v,w})^{(m)}(t)$ lies in $L^1(\R)$ for 
all $m\ge0$, so we can use integration by parts to show that
\[
\hat\rho_{v,w}(ib)=(ib)^{-m}\hat\rho_{v,\partial_t^mw}(ib)
\quad\text{for all $b\neq0$, $m\ge0$}.
\]
Choose 
$\psi:\R\to[0,1]$ to be $C^\infty$ and compactly supported such that $\psi\equiv1$ on a neighborhood of zero.
Let $\kappa_m(b)=(1-\psi(b))(ib)^{-m}$.
Note that
\begin{align} \label{eq:psikappa}
\psi, \, \kappa_m \in\cR(t^{-p})
\quad\text{for all $p>0$, $m\ge2$}.
\end{align}
We have
\begin{align} \label{eq:list} \nonumber
\hat\rho_{v,w}(ib) & =\psi(b)\hat\rho_{v,w}(ib)+\kappa_m(b)\hat\rho_{v,\partial_t^mw}(ib) \\ \nonumber
& =
 \psi(b)\hJ_{0,v,w}(ib)
+ \intphi^{-1}\psi(b)\int_Y \hT(ib)R\hV(ib)\,\hw(ib)\,d\mu
\\ & \qquad +\kappa_m(b)\hJ_{0,v,\partial_t^mw}(ib) 
+ \intphi^{-1}\kappa_m(b)\int_Y \hT(ib)R\hV(ib)\,\widehat{\partial_t^m w}(ib)\,d\mu.
\end{align}
It remains to estimate the inverse Fourier transform of each term in~\eqref{eq:list}.

\begin{prop} \label{prop:J0est}
After truncation, uniformly in $N\ge1$, 
\[
\psi\hJ_{0,v,w}\in\cR(|v|_\infty|w|_\infty\, t^{-(\beta-1)}), \qquad
\kappa_m\hJ_{0,v,w}\in\cR(|v|_\infty|w|_\infty\, t^{-(\beta-1)}),
\]
for all $m\ge2$, $v\in L^\infty(Y^\varphi)$,
 $w\in L^\infty(Y^\varphi)$.
\end{prop}

\begin{proof}
By Corollary~\ref{cor:Jw},
$|\hJ_{0,v,w}| \in\cR(|v|_\infty|w|_\infty\, t^{-(\beta-1)})$
uniformly in $N$.
Using~\eqref{eq:psikappa} and Proposition~\ref{prop:conv},
$\psi\hJ_{0,v,w}\in \cR(t^{-\beta}\star
(|v|_\infty|w|_\infty\, t^{-(\beta-1)})
=\cR(|v|_\infty|w|_\infty\, t^{-(\beta-1)}))$.
Similarly, for
$\kappa_m\hJ_{0,v,w}$.
\end{proof}

The main two estimates are as follows.  
Recall that $q\in(\beta-1,\beta)$.

\begin{lemma} \label{lem:smallb}
There exists $N_1\ge1$ such that
after truncation, uniformly in $N\ge N_1$, there exist $m\ge2$ and $\psi$ such that
\begin{align*}
\SMALL \psi\int_Y \hT R\hV\,\hw\,d\mu
\in\cR( \|v\|_{\theta,\eta}|w|_\infty \, t^{-(\beta-1)}),
\end{align*}
for all
$v\in \cF_{\theta,\eta}(Y^\varphi)$ with $\int_{Y^\varphi}v\,d\mu^\varphi=0$,
$w\in L^\infty(Y^\varphi)$. 
\end{lemma}

\begin{lemma} \label{lem:largeb}
There exists $N_1\ge1$ such that
after truncation, uniformly in $N\ge N_1$, there exist $m\ge2$ such that
\begin{align*}
\SMALL \kappa_m\int_Y \hT R\hV\,\hw\,d\mu
\in\cR(\|v\|_{\theta,\eta}|w|_\infty\, t^{-q}),
\end{align*}
for all
$v\in \cF_{\theta,\eta}(Y^\varphi)$,
$w\in L^\infty(Y^\varphi)$. 
\end{lemma}

\begin{pfof}{Lemma~\ref{lem:trunc}} 
Substituting the results from Proposition~\ref{prop:J0est} and Lemmas~\ref{lem:smallb} and~\ref{lem:largeb}
into~\eqref{eq:list}, we obtain that
$\hat\rho_{v,w}\in\cR(\|v\|_{\theta,\eta}|w|_{\infty,m}\, t^{-(\beta-1)})$
uniformly in $N\ge N_1$ for all 
$v\in \cF_{\theta,\eta}(Y^\varphi)$ with 
$\int_{Y^\varphi}v\,d\mu^\varphi=0$,
$w\in L^{\infty,m}(Y^\varphi)$.
\end{pfof}

The proofs of Lemmas~\ref{lem:smallb} and~\ref{lem:largeb} are presented in 
Subsections~\ref{sec:smallb} and~\ref{sec:largeb} respectively.

\subsection{Small $b$: Proof of Lemma~\ref{lem:smallb}}
\label{sec:smallb}

Let $\delta>0$ be as in Subsection~\ref{sec:further}.
Temporarily, we introduce the notation 
$\hR_N(s)v=R(e^{-s(\varphi\wedge N)}v)$.

\begin{prop} \label{prop:RRN}
$\lim_{N\to\infty}\|\hR_N^{(q)}(ib)-\hR^{(q)}(ib)\|_{\theta_1}=0$
uniformly for $|b|<\delta$.
\end{prop}

\begin{proof}
Write
$\hR(ib)v-\hR_N(ib)v=R(f(b)v)$ where
$f(b)=g_1(b)-g_2(b)$, $g_1(b)=e^{-ib\varphi}$, $g_2(b)= e^{-ib(\varphi\wedge N)}$.
Then
$\hR^{(k)}(ib)v-\hR_N^{(k)}(ib)v=R(f^{(k)}(b)v)$ for all $k\in\N$ where
\[
f^{(k)}(b)=g_1^{(k)}(b)-g_2^{(k)}(b), \quad
g_1^{(k)}(b)=(-i)^ke^{-ib\varphi}\varphi^k, \quad
g_2^{(k)}(b)=
(-i)^ke^{-ib(\varphi\wedge N)}(\varphi\wedge N)^k.
\]

If $\supYj\varphi\le N$, then $f^{(k)}(b)\equiv0$ on $Y_j$.

If $\supYj\varphi>N$, then 
$|1_{Y_j}g_i^{(k)}(b)|_\infty \le 
\supYj\varphi^k\le 2C_1
\infYj\varphi^k$ for $i=1,2$.
By the MVT argument
$|1_{Y_j}f^{(q)}(b)|_\infty \le 4C_1 \infYj\varphi^{q}$.

Next, 
for $y_1,y_2\in Y_j$, we have
$g_1(b)(y_1)-g_1(b)(y_2)= I_1+I_2$ where
\[
I_1=\{e^{-ib\varphi(y_1)}-e^{-ib\varphi(y_2)}\}\varphi(y_1)^k, \qquad
I_2=e^{-ib\varphi(y_2)}\{\varphi(y_1)^k-\varphi(y_2)^k\}.
\]
Note that
\begin{align*}
|I_1| & \le 2 |b|^\eta |\varphi(y_1)-\varphi(y_2)|^\eta\varphi(y_1)^k
 \le 2|b|^\eta |1_{Y_j}\varphi|_{\theta}^\eta\supYj\varphi^k d_{\theta}(y_1,y_2)^\eta \ll \infYj\varphi^{k+\eta}d_{\theta_1}(y,y'), \\
|I_2| & \le k|\varphi(y_1)-\varphi(y_2)|\supYj\varphi^{k-1}
 \le k|1_{Y_j}\varphi|_{\theta}\supYj\varphi^{k-1}d_{\theta}(y_1,y_2)
\ll k\,\infYj\varphi\, d_{\theta_1}(y,y'),
\end{align*}
so $|1_{Y_j}g_1^{(k)}(b)|_{\theta_1}\ll k\infYj\varphi^{k+\eta}$.
By the MVT argument, 
$|1_{Y_j}g_1^{(q)}(b)|_{\theta_1}\ll \infYj\varphi^{q+\eta}$.
Similarly $|1_{Y_j}g_2^{(q)}(b)|_{\theta_1}\ll \infYj\varphi^{q+\eta}$.
Hence
$\|1_{Y_j}f^{(k)}(b)\|_{\theta_1}\ll 
1_{\{\supYj\varphi>N\}}\infYj\varphi^{k+\eta}
\le 1_{\{\infYj\varphi>N/(2C_1)\}}\infYj\varphi^{k+\eta}$.

By Remark~\ref{rmk:GM},
\begin{align*}
\|\hR^{(q)}(ib)- & \hR_N^{(q)}(ib)\|_{\theta_1}
 \le 2C_2\sum\mu(Y_j)\|1_{Y_j}f^{(q)}(b)\|_{\theta_1}
\\ & \ll \sum 1_{\{\infYj\varphi>N/(2C_1)\}}
\mu(Y_j)\infYj\varphi^{q+\eta}
 \le {\SMALL \int}_Y 1_{\{\varphi>N/(2C_1)\}}
\varphi^{q+\eta}\,d\mu.
\end{align*}
The result follows since $\varphi\in L^{q+\eta}(Y)$.
\end{proof}

By Proposition~\ref{prop:RRN}, we can fix $\delta>0$ and $N_1\ge1$ such that
for all $N\ge N_1$ there exists a
$C^q$ family of simple eigenvalues $\lambda_N(b)$, $|b|<\delta$, for
$\hR_N(ib)$ 
with $\lambda_N(0)=1$ and $\lambda_N'(0)=1$. 
There is also a corresponding $C^q$ family of spectral projections
$P_N(b)$ with $P_N(0)v=\int_Y v\,d\mu$.

From now on, we write $\hR$, $P$ and $\lambda$ instead of
$\hR_N$, $P_N$ and $\lambda_N$.
Recall that $\hT=(I-\hR)^{-1}$.
Choose the $C^\infty$ function $\psi:\R\to[0,1]$ so that $\supp\psi\in(-\delta,\delta)$, and write
\begin{align} \label{eq:decomp}
\hT(ib)R\hV(ib)=(1-\lambda(b))^{-1}P(b)R\hV(ib)+
\hT(ib)(I-P(b))R\hV(ib).
\end{align}
We begin by dealing with the second term in~\eqref{eq:decomp}.

\begin{lemma} \label{lem:Q}
$\psi\int_Y\hT(I-P)R\hV\,\hw\,d\mu\in\cR(\|v\|_{\theta,\eta}
|w|_\infty\,t^{-q})$ for all 
$v\in \cF_{\theta,\eta}(Y^\varphi)$, $w\in L^\infty(Y^\varphi)$.
\end{lemma}

\begin{proof}
By definition of $P$ and Proposition~\ref{prop:Rslow},
$\hT(ib)(I-P(b))$ is $C^q$ on $\cF_{\theta_1}(Y)$ uniformly in $N$.
By Proposition~\ref{prop:fourier},
$\psi\hT(I-P)\in\cR(t^{-q})$.

Choose $\psi_1$ to be $C^\infty$ with compact support such that $\psi_1\equiv1$ on $\supp\psi$.
By Corollary~\ref{cor:Vslow}, 
\begin{align*}
\psi\hT(I-  P)  R\hV & = \{\psi\hT(I-  P)\}\{\psi_1  R\hV\} \in
\cR(t^{-q}\star \|v\|_{\theta,\eta}\, t^{-q}),
\end{align*}
in $\cF_{\theta_1}(Y)$ and hence in $L^\infty(Y)$.
Since $q>1$, it follows from Proposition~\ref{prop:conv}
that $\psi\hT(I-  P)  R\hV\in \cR(\|v\|_{\theta,\eta}\, t^{-q})$
in $L^\infty(Y)$.

By Corollary~\ref{cor:Jw}, $\hw\in\cR(|w|_\infty\,t^{-\beta})$ in $L^1(Y)$,
so the result follows from Proposition~\ref{prop:conv}.~
\end{proof}

Still working with the truncated roof function $\varphi\wedge N$, define
\begin{align*}
\widetilde R(s) & =s^{-1}(\hR(s)-\hR(0)), \quad
\widetilde P(b)=b^{-1}(P(b)-P(0)),  \quad
\tilde\lambda(b)  =(ib|\varphi\wedge N|_1)^{-1}(1-\lambda(b)).
\end{align*}
The estimates in Propositions~\ref{prop:tildeR} and~\ref{prop:tildeP} remain valid uniformly in $N$.

\begin{prop} \label{prop:lambdaP}
$\psi\tilde\lambda^{-1}\in\cR(t^{-q})$.
\end{prop}

\begin{proof} 
By Proposition~\ref{prop:fourier}, it suffices to show that
$|(\tilde\lambda^{-1})^{(q)}(b)|\ll |b|^{-(1-\eta)}$.

By Proposition~\ref{prop:tildeP},
$|\tilde\lambda^{(p)}(b)|\ll |b|^{-(1-\eta)}$ for $p<\beta-2\eta$ and
$|\tilde\lambda^{(p)}(b)|\ll 1$ for $p<\beta-1$.
In particular, $\tilde\lambda$ is H\"older uniformly in $N$ and~$b$. 
Recall that $\tilde\lambda(0)=1$.  Shrinking $\delta$, we obtain
$|\tilde\lambda(b)|\ge{\SMALL\frac12}$ for all $|b|<\delta$, $N\ge N_1$.

\noindent{\bf The case $\beta\ge2$.}
Note that $(\tilde\lambda^{-1})^{(q)}$ is a finite linear combination of finite products with factors $\tilde\lambda^{-1}$ (each bounded by $2$) and $\tilde\lambda^{(q_1)}\cdots \tilde\lambda^{(q_k)}$ where $q_1+\dots +q_k=q$.
If $q_{i_j}\ge\beta-1$ for distinct values of $i_j$, then
$\beta\le 2\beta-2\le q_{i_1}+q_{i_2}\le q$
contradicting the assumption that $q<\beta$.
Hence there exists at most one $i$ such that $q_i\ge\beta-1$ and
$|\tilde\lambda^{(q_1)}\cdots \tilde\lambda^{(q_k)}|\ll |\tilde\lambda^{(q_i)}|\ll |b|^{-(1-\eta)}$.   It follows that $|(\tilde\lambda^{-1})^{(q)}(b)|\ll |b|^{-(1-\eta)}$.

\noindent{\bf The case $\beta\in(1,2)$.}
Let $q=1+r<\beta$.  Then $(\tilde\lambda^{-1})'=-\tilde\lambda^{-2}\tilde\lambda'$ and 
$|(\tilde\lambda^{-1})'(b)
-(\tilde\lambda^{-1})'(b')|\ll |\tilde\lambda'(b)-\tilde\lambda'(b')|\ll 
|b|^{-(1-\eta)}|b-b'|^r$ for $|b|\le |b'|$.  Hence
$|(\tilde\lambda^{-1})^{(q)}(b)|\ll |b|^{-(1-\eta)}$.
\end{proof}

\begin{pfof}{Lemma~\ref{lem:smallb}}
By Lemma~\ref{lem:Q}, it remains to handle the first term
in~\eqref{eq:decomp}, namely 
\begin{align*}
& (1-\lambda(b))^{-1}P(b)R\hV(ib)  =(ib|\varphi\wedge N|)^{-1}\tilde\lambda(b)^{-1}P(b)R\hV(ib)
\\ & \qquad =(i|\varphi\wedge N|)^{-1}\tilde\lambda(b)^{-1}b^{-1}P(0)R\hV(ib)
+(i|\varphi\wedge N|)^{-1}\tilde\lambda(b)^{-1}\widetilde P(b)R\hV(ib).
\end{align*}

Choose $q\in(1,\beta)$, $q\ge \beta-1$.
By Propositions~\ref{prop:fourier},~\ref{prop:tildeP} and~\ref{prop:lambdaP}, 
$\psi\widetilde P,\,\psi\tilde \lambda^{-1}\in\cR(t^{-q})$.
By Proposition~\ref{prop:v0},
\begin{align*}
\psi(b) \tilde\lambda(b)^{-1} b^{-1}P(0)R\hV(ib)
& \in\cR(t^{-q}\star |v|_\infty\, t^{-(\beta-1)}).
\end{align*}

Choose $\psi_1$ to be $C^\infty$ with compact support such that $\psi_1\equiv1$ on $\supp\psi$.
Using also Corollary~\ref{cor:Vslow}, 
\begin{align*}
\psi\tilde\lambda^{-1}\widetilde P R\hV
 & =\{\psi\tilde\lambda^{-1}\}\{\psi_1\widetilde P\}\{\psi_1 R\hV\}
 \in
\cR(t^{-q}\star t^{-q}\star \|v\|_{\theta,\eta}\,t^{-q})\quad\text{in $L^\infty(Y)$}.
\end{align*}
Hence by Proposition~\ref{prop:conv},
$\psi(1-\lambda)^{-1}PR\hV
\in\cR(\|v\|_{\theta,\eta}\,t^{-(\beta-1)})$ in $L^\infty(Y)$.
Combining with Corollary~\ref{cor:Jw}, we obtain
$\psi\int_Y (1-\lambda)^{-1}PR\hV\,\hw\,d\mu
\in\cR(\|v\|_{\theta,\eta}|w|_\infty\,t^{-(\beta-1)})$ as required.
\end{pfof}

\subsection{Large $b$: Proof of Lemma~\ref{lem:largeb}}
\label{sec:largeb}

Define $\kappa_m(b)=(1-\psi(b))(ib)^{-m}$ where $\psi$ is chosen as in Lemma~\ref{lem:smallb}.

\begin{prop} \label{prop:Test}
There exists $C>0$ and $m\ge2$ such that
\begin{align*}
 \kappa_m\, \hT \in\cR( t^{-q})
\;\text{in $\cB(\cF_{\theta_1}(Y))\;$ for all $N\ge N_1$, $t>1$.}
\end{align*}
\end{prop}

\begin{proof}
By the choice of $N\ge N_1$, absence of approximate eigenfunctions passes over to the truncated semiflow.
By Corollary~\ref{cor:MTslow}, 
$\|\hT^{q}(ib)\|_{\theta_1}\le C|b|^\alpha$ for $b\in\supp\kappa_m$.
Note that all constants entering into $C$ and $\alpha$ are universal ($C_2$, $C_3$, etc) or depend only on the values of $\varphi$ on the finite subsystem $Z_0$.  In particular, $C$ and $\alpha$ are independent of $N$ for all $N\ge N_1$.  

Hence 
$\|(\kappa_m \hT)^{(q)}(b)\|_{\theta_1}\ll(|b|+1)^{\alpha-m}$.
Choosing $m\ge\alpha+2$ ensures integrability and the result follows
from Proposition~\ref{prop:fourier}.
\end{proof}

\begin{pfof}{Lemma~\ref{lem:largeb}}
Write $\kappa_m=\kappa_{m-2}\kappa_2$, where
$\kappa_i$ is $C^\infty$, vanishes in a neighborhood of zero, and is $O(|b|^{-i})$, for $i=2$ and  $i=m-2$.

By Corollaries~\ref{cor:Jw} and~\ref{cor:Vslow}, and Proposition~\ref{prop:Test}, 
we obtain (for sufficiently large~$m$)
\begin{align*}
\kappa_m{\textstyle\int_Y} \hT R\hV\,\hw \,d\mu
& ={\textstyle\int_Y} (\kappa_{m-2}\hT)(\kappa_2 R\hV)\,\hw \,d\mu
 \in\cR\big(t^{-q}\star
\|v\|_{\theta,\eta} t^{-q}\star |w|_\infty\, t^{-\beta}\big).
\end{align*}
Since $q\in(1,\beta)$, the result follows from Proposition~\ref{prop:conv}.
\end{pfof}

\section{Some open questions}

We conclude this review article with some open questions.  

\begin{itemize}[leftmargin=0.2in]
\item Improve the Diophantine criterion for absence of approximate eigenfunctions by reducing the number of periods from three to two in Proposition~\ref{prop:tau}.  It would suffice to show that we can take $\psi_k=1$ in~\eqref{eq:approx} as is the case for uniformly hyperbolic systems~\cite{Dolgopyat98b}.

(In particular, the proof of rapid mixing for two falling balls 
in~\cite{BalintNV16} seems to rely on such an improvement. Alternatively, one could try to verify the good asymptotics condition~\cite{FMT07} described in Section~\ref{sec:good}; this would also lead to a stronger conclusion (robust rapid mixing rather than almost sure) in~\cite{BalintNV16}.)
\item  For flows with polynomial decay of correlations, the decay rates in this article are optimal but the class of observables is not.  An open problem is to remove the requirement that observables are smooth in the flow direction.  
In general, this is a currently intractable problem even in the superpolynomial case.  However in situations where there is additional structure in which exponential decay methods have proven successful (smooth stable foliation or contact structure) there is the possibility of combining these methods with the truncation method in Section~\ref{sec:trunc}.  A key example is the infinite horizon planar periodic Lorentz gas where the optimal decay rate $O(t^{-1})$ is obtained in~\cite{BBMsub} but for a restricted class of observables.
\item The statistical properties for time-one maps of rapid mixing flows in Section~\ref{sec:stats} are restricted to observables that are sufficiently regular in the flow direction, and this is currently the best available result even when the flow is exponentially mixing.  To fix ideas, consider the finite horizon planar periodic Lorentz gas.  The time-one map is exponentially mixing for H\"older observables by~\cite{BaladiDemersLiverani18}, but currently this does not lead to any statistical properties.
On the other hand, the proof of superpolynomial decay for sufficiently regular observables does lead to the statistical properties listed in Subsection~\ref{sec:stats}.  Hence a natural question is to to investigate how to use the result or method in~\cite{BaladiDemersLiverani18} to extract statistical properties.  This is currently the topic of work in progress with Mark Demers and Matthew Nicol.

\item As mentioned in Remark~\ref{rmk:contact}, Theorem~\ref{thm:D} gives a simple criterion for absence of approximate eigenfunctions when $Z_0\subset Y$ is connected, namely that the temporal distance function $D$ is not identically zero. Such a nonintegrability property is not immediately of use in general: $Z_0$ is a Cantor set of positive Hausdorff dimension and $Y$ is often connected, but
$D$ is only H\"older.   On the other hand, $Z_0$ can be constructed using any finite subcollection of the partition elements $Y_j$, and so in some sense exhausts $Y$.  The question is whether $D$ has sufficient structure beyond being H\"older (which on its own is clearly insufficient) to imply that the lower box dimension of $D(Z_0\times Z_0)$ is close to that of $D(Y\times Y)$ for suitable chosen $Z_0$.  This would rule out approximate eigenfunctions when $Y$ is connected and $D$ is not identically zero.
\item 
Methods for suspension semiflows and flows can be adapted to toral extensions of maps, replacing roof functions $\varphi:Y\to\R^+$ by
cocycles $\varphi:Y\to\R^d$.  Rapid mixing for toral extensions (and general compact group extensions) of uniformly expanding/hyperbolic maps is analysed in depth in~\cite{Dolgopyat02}.  Results on rapid and slow mixing for toral extensions of nonuniformly expanding maps are obtained in~\cite{MTtoralsub}.
The results for skew product flows in Section~\ref{sec:skew} should also go over to toral extensions of nonuniformly hyperbolic transformations with $\varphi$ constant along stable leaves.  

For toral extensions that are not skew products, we expect that the methods described in~\cite{BBMsub} apply when there is exponential contraction along stable leaves (Section~\ref{sec:nonskew}(i)), or when $\varphi$ has bounded H\"older constants (Section~\ref{sec:nonskew}(ii)), but there is no analogue of situation~(iii) from Section~\ref{sec:nonskew}.  An open problem is to understand more fully toral extensions of nonuniformly hyperbolic transformations (and hence understand more fully nonuniformly hyperbolic flows).
\item  Continuing the previous question, although the results in~\cite{BBMsub} described here deal with many important classes of flows, the situation is still not as satisfactory as for semiflows.  For example, Chernov \& Zhang~\cite{ChernovZhang05b} consider a family of periodic dispersing billiards where the nonvanishing curvature hypothesis on scatterers is violated.   The associated billiard maps exhibit polynomial decay rates $O(n^{-b})$ for any prescribed $b\in(1,\infty)$.   However the flows do not seem to be covered by the methods in~\cite{BBMsub} even though the results in this article yield decay rates $O(t^{-b})$ at the semiflow level.
\end{itemize}

  \paragraph{Acknowledgements}
 This work was supported in part by a
 European Advanced Grant {\em StochExtHomog} (ERC AdG 320977).
We are grateful to P\'eter B\'alint and Oliver Butterley for very helpful discussions, and to the referees for making numerous suggestions that greatly improved the readability of the paper.

{\small

\def\polhk#1{\setbox0=\hbox{#1}{\ooalign{\hidewidth
  \lower1.5ex\hbox{`}\hidewidth\crcr\unhbox0}}}

}

\end{document}